\UseRawInputEncoding

\documentclass[12pt]{amsart}

\usepackage{amssymb, tikz}

\usepackage{hyperref}

\usepackage{amsmath}

\usepackage{amsmath, amsthm, amscd, amsfonts}

\usepackage{mathtools}

\usepackage[english]{babel}
\usepackage{setspace}

\let \mtcl = \mathcal
\let \a = \alpha
\let \al = \aleph
\let \b = \beta
\let \g = \gamma
\let \d = \delta

\let \l = \lambda
\let \k = \kappa
\let \m = \mu
\let \n = \nu
\let \p = \pi
\let \r = \rho

\let \s = \sigma
\let \x = \xi
\let \z= \zeta
\let \o = \omega

\let \it = \item

\DeclareMathOperator{\SATP}{SATP}
\DeclareMathOperator{\SH}{SH}

\DeclareMathOperator{\dom}{dom}

\DeclareMathOperator{\cf}{cf}
\DeclareMathOperator{\Col}{Col}

\def\MSB{{\mathbb{S}}}
\def\MPB{{\mathbb{P}}}
\def\MQB{{\mathbb{Q}}}

\def\k{\kappa}

\def\l{\lambda}

\def\a{\alpha}

\def\g{\gamma}
\def\d{\delta}
\def\b{\beta}

\newtheorem{theorem}{Theorem}[section]

\newtheorem{lemma}[theorem]{Lemma}

\newtheorem{corollary}[theorem]{Corollary}

\newtheorem{definition}[theorem]{Definition}

\newtheorem{remark}[theorem]{Remark}

\newtheorem{notation}[theorem]{Notation}
\newtheorem{claim}[theorem]{Claim}
\newtheorem{subclaim}[theorem]{Subclaim}
\numberwithin{equation}{section}
\newtheorem{convention}[theorem]{Convention}

\def\MPB{{\mathbb{P}}}
\def\MQB{{\mathbb{Q}}}

\def\k{\kappa}

\def\l{\lambda}

\def\a{\alpha}

\def\g{\gamma}
\def\d{\delta}
\def\b{\beta}

\def\l{\lambda}

\def\rmark{\mbox{$\rm\bf\rule{0.06em}{1.45ex}\kern-0.05em R$}}
\def\pmark{\mbox{$\rm\bf\rule{0.06em}{1.45ex}\kern-0.05em P$}}
\def\nmark{\mbox{$\rm\bf\rule{0.06em}{1.45ex}\kern-0.05em N$}}
\def\vdash{\mbox{$\rm\| \kern-0.13em -$}}
\newcommand{\lusim}[1]{\smash{\underset{\raisebox{1.2pt}[0cm][0cm]{$\sim$}}
{{#1}}}}

\newcommand{\ZFC}{\hbox{ZFC}}

\newcommand{\GCH}{\hbox{GCH}}
\newcommand{\CH}{\hbox{CH}}
\newcommand{\Ord}{\hbox{Ord}}

\newcommand{\ot}{\hbox{ot}}

\newcommand{\PFA}{\hbox{PFA}}

\DeclareMathOperator{\ssup}{sup}
\DeclareMathOperator{\cl}{cl}

\usepackage{pb-diagram}

\def\l{\lambda}

\def\rmark{\mbox{$\rm\bf\rule{0.06em}{1.45ex}\kern-0.05em R$}}
\def\pmark{\mbox{$\rm\bf\rule{0.06em}{1.45ex}\kern-0.05em P$}}
\def\nmark{\mbox{$\rm\bf\rule{0.06em}{1.45ex}\kern-0.05em N$}}
\def\vdash{\mbox{$\rm\| \kern-0.13em -$}}

\title[The special Aronszajn tree property at $\aleph_2$ and $\GCH$]{The special Aronszajn tree property at $\aleph_2$ and $\GCH$}

\author[D.\ Asper\'o]{David Asper\'{o}}

\address{David Asper\'o, School of Mathematics, University of East Anglia, Norwich NR4 7TJ, UK}

\email{d.aspero@uea.ac.uk}

\author[M.\  Golshani]{Mohammad Golshani}

\address{Mohammad Golshani, School of Mathematics, Institute for Research in Fundamental Sciences (IPM), P.O.\ Box:
19395--5746, Tehran, Iran.}

\email{golshani.m@gmail.com}

\thanks{}

\date{}

\begin{document}



\begin{abstract}
Starting from the existence of a weakly compact cardinal, we build a generic extension of the universe in which $\GCH$
holds and all $\aleph_2$-Aronszajn trees are special and hence there are no $\aleph_2$-Souslin trees. This result answers a well-known open question from the 1970's.
\end{abstract}

 \maketitle

\section{introduction}
Let $\kappa$ be an uncountable regular cardinal. Let us recall that a $\kappa$-tree is a tree $T$ of height $\kappa$ all of whose levels are smaller than $\kappa$, and that a $\kappa$-tree is called a $\kappa$-Aronszajn
tree if it has no $\kappa$-branches. Also, $T$ is called a $\kappa$-Souslin tree if
it has no $\kappa$-branches and no antichains of size $\kappa$.
When $\kappa=\lambda^+$ is a successor cardinal, a $\kappa$-Aronszajn tree is said to be \emph{special} if and only if it is a union of $\l$ antichains.\footnote{$A\subseteq T$ is an antichain of $T$ iff $A$ consists of pairwise incomparable nodes.}
Let us make the following definition:
\begin{definition}
\begin{enumerate}
\item Souslin's Hypothesis at $\kappa$, $\SH_\kappa$, is the statement ``there are no $\kappa$-Souslin trees''.
\item The special Aronszajn tree property at $\kappa=\lambda^+$, $\SATP_\kappa$, is the statement ``there exist $\kappa$-Aronszajn trees and all such trees are
special'' (see \cite{golshani-hayut}).
\end{enumerate}
\end{definition}

Aronszajn trees were introduced by Aronszajn (see \cite{kurepa}), who proved the existence, in $\ZFC$, of a special $\aleph_1$-Aronszajn tree. Later, Specker (\cite{specker}) showed that $2^{{<}\lambda}=\lambda$
 implies the existence of special $\lambda^+$-Aronszajn trees for $\lambda$ regular,
 and Jensen (\cite{jensen}) produced special $\lambda^+$-Aronszajn trees for singular $\lambda$ in $L$.

In \cite{solovay-tennenbaum}, Solovay and Tennenbaum proved the consistency of Martin's Axiom + $2^{\aleph_0} > \aleph_1$ and showed that this implies
$\SH_{\aleph_1}$. This was later extended by
Baumgartner, Malitz and Reinhardt \cite{b-m-r}, who showed that Martin's Axiom + $2^{\aleph_0} > \aleph_1$ implies $\SATP_{\aleph_1}$.
 Later,
Jensen (see \cite{devlin} and \cite{shelah}) produced a model of $\GCH$ in which $\SATP_{\aleph_1}$ holds.

The situation at $\aleph_2$ turned out to be more complicated. In \cite{jensen}, Jensen proved that the existence of an $\al_2$-Souslin tree follows from each of the hypotheses $\CH+\diamondsuit(S^2_1)$ and $\Box_{\o_1}+\diamondsuit(S^2_0)$ (where, given $m<n<\o$, $S^n_m=\{\a<\al_n\,\mid\, \cf(\a) =\al_m\}$). The second result was improved by Gregory in \cite{gregory}, where he proved that $\GCH$ together with the existence of a non-reflecting stationary subset of $S^2_0$ yields the existence of an $\al_2$-Souslin tree. In \cite{laver-shelah}, Laver and Shelah produced, relative to the existence of a weakly compact cardinal, a model of $\ZFC+\CH$ in which the special Aronszajn tree property at $\aleph_2$ holds. But in their model $2^{\aleph_1} > \aleph_2$,
and the task of finding a model of $\ZFC$+$\GCH$+$\SATP_{\aleph_2}$, or even of  $\ZFC$+$\GCH$+$\SH_{\aleph_2}$, remained a major open problem. The earliest published mention of this problem seems to appear in \cite{kanamori-magidor} (see also \cite{laver-shelah}, \cite{todorcevic}, \cite{shelah-stanley}, \cite{shelah-anniv-issue}, or \cite{rinot0}).

In this paper we solve the above problem by proving the following theorem.
\begin{theorem}
\label{main theorem}
Suppose $\kappa$ is a weakly compact cardinal. Then there exists a set-generic extension of the universe in which
$\GCH$ holds, $\kappa=\aleph_2$, and the special Aronszajn tree property  at $\aleph_2$ (and hence Souslin's Hypothesis at $\aleph_2$) holds.
\end{theorem}

\begin{remark}
\normalfont
\begin{enumerate}
\item Our argument can be easily extended to deal with the successor of any regular cardinal.

\item By results of Shelah and Stanley (\cite{shelah-stanley}) and of Rinot (\cite{rinot}), our large cardinal assumption is optimal. Specifically:

\begin{enumerate}
\item It is proved in \cite{shelah-stanley} that if $\o_2$ is not weakly compact in $L$, then either $\Box_{\o_1}$ holds or there is a non-special $\aleph_2$-Aronszajn tree; in particular, $\GCH+\SATP_{\aleph_2}$ implies that $\o_2$ is weakly compact in $L$ by one of Jensen's results mentioned above.

\item Rinot proved in \cite{rinot} that if $\GCH$ holds, $\l\geq\o_1$ is a cardinal, and $\Box(\l^+)$ holds, then there is a $\l$-closed $\l^+$-Souslin tree; on the other hand, Todor\v{c}evi\'{c} (\cite{Todor}) proved that if $\k\geq\o_2$ is a regular cardinal and $\Box(\k)$ fails, then $\k$ is weakly compact in $L$.
\end{enumerate}
\end{enumerate}
\end{remark}

The rest of the paper is devoted to the proof of Theorem \ref{main theorem}.
We will next give an (inevitably) vague and incomplete description of the forcing witnessing the conclusion of this theorem.

The construction of the forcing witnessing Theorem \ref{main theorem} combines a natural iteration for specializing $\al_2$-Aronszajn trees, due to Laver
and Shelah (\cite{laver-shelah}), with ideas from \cite{aspero-mota2}.
More specifically, we build a certain countable support forcing iteration $\langle \MQB_\beta\,\mid\,\beta\leq\kappa^{+}\rangle$ with side conditions. The first step of the construction is essentially the L\'{e}vy collapse of the weakly compact cardinal $\kappa$ to become $\omega_2$. At subsequent stages, we consider forcings for specializing $\aleph_2$-Aronszajn trees by countable approximations.
Conditions in a given $\MQB_\beta$, for $\b>0$, will consist of a working part $f_q$, together with a certain side condition. The working part $f_q$ will be a countable function with domain contained in $\b$ such that for all $\a\in\dom(f_q)$,
\begin{itemize}
\item $f_q(\a)$ is a condition in the  L\'{e}vy collapse if $\a=0$, and
\item if $\a>0$, $f_q(\a)$ is a countable subset of $\k\times\o_1$.
\end{itemize} \noindent Letting $\a=\a_0+\n$, where $\a_0$ is a multiple of $\o_1$ and $\n<\o_1$, any two distinct members of $f_q(\a)$, when $\a>0$, will be forced to be incomparable nodes in a certain $\k$-Aronszajn tree $\lusim{T}_{\a_0}$ on $\k\times\o_1$ chosen via a given bookkeeping function $\Phi:\mtcl X\longrightarrow H(\k^{+})$, where $\mtcl X$ denotes the set of multiples of $\o_1$ below $\k^+$.

The side condition will be a countable directed graph $\tau_q$ whose vertices are ordered pairs of the form $(N, \g)$, where $N$ is an elementary submodel of $H(\kappa^{+})$  such that $\arrowvert N\arrowvert=\arrowvert N\cap\kappa\arrowvert$ and $^{{<}|N|}N\subseteq N$, and where $\g$ is an ordinal in the closure of $N\cap (\b+1)$ in the order topo\-lo\-gy. Given any such $(N, \g)$, $\g$ is to be seen as a \emph{marker for $N$} in $q$, telling us up to which stage is $N$ `active' as a model. We will tend to call such pairs $(N, \g)$ \emph{models with markers}. Whenever $\langle (N_0, \g_0), (N_1, \g_1)\rangle$ is an edge in $\tau_q$, for a condition $q$, $(N_0, \in)$ and $(N_1, \in)$ are $\in$-isomorphic via a (unique) isomorphism $\Psi_{N_0, N_1}$ such that $\Psi_{N_0, N_1}(\x)\leq\x$ for every ordinal $\x\in N_0$ and such that $\Psi_{N_0, N_1}$ is in fact an isomorphism between the structures $(N_0, \in, \Phi_{\a})$ and $(N_1, \in, \Phi_{\Psi_{N_0, N_1}(\a)})$ whenever $\a\in N_0\cap\g_0$ is such that $\Psi_{N_0, N_1}(\a)<\g_1$, for a certain sequence $(\Phi_\a)_{\a<\k^+}$ of increasingly expressive predicates contained in $H(\k^+)$.

In the above situation, $N_0$ and $N_1$ are to be seen as `twin models', relative to $q$, with respect to all stages $\a$ and $\Psi_{N_0, N_1}(\a)$ such that $\a\in  N_0\cap\g_0$ and $\Psi_{N_0, N_1}(\a)< \g_1$.
This means that the natural restriction of $f_q(\a)$ to $N_0$, i.e., $f_q(\a)\cap N_0$, is to be copied over, via $\Psi_{N_0, N_1}$, into the restriction of $f_q(\Psi_{N_0, N_1}(\a))$ to $N_1$, i.e., we require that $$\Psi_{N_0, N_1}(f_q(\a)\cap N_0)=f_q(\a)\cap N_0 \subseteq f_q(\Psi_{N_0, N_1}(\a)),$$ and similarly for the restriction of $\tau_q\restriction\alpha+1$ to $N_0$ (with the restriction $\tau_q\restriction\alpha+1$ being defined naturally).

We can describe our copying procedure by saying that we are \emph{copying into the past information coming from the future via the edges in $\tau_q$}. Given an edge $\langle (N_0, \g_0), (N_1, \g_1)\rangle$ as above and some $\a\in N_0\cap\g_0$ such that $\bar\a=\Psi_{N_0, N_1}(\a)<\g_1$, the intersection of $f_{q}(\bar\a)$ with $\d_{N_0}\times\o_1$ may certainly contain more information than the intersection of $f_q(\a)$ with $\d_{N_0}\times\o_1$. Thanks to the way we are setting up the copying procedure---namely, only copying from the future into the past via edges as we have described---it is straightforward to see that our construction is in fact a forcing iteration, in the sense that $\MQB_\a$ is a complete suborder of $\MQB_\b$ for all $\a<\b$. This need not be true in general, in forcing constructions of this sort, if we allow also to copy `from the past into the future'.\footnote{It turns out that, in our specific construction, and thanks to clause (7) in the definition of condition, we could in fact have required to copy information in both directions, i.e. that, in the above situation, full symmetry obtains, below $\d_{N_0}\times\o_1$, between stages $\a$ and $\Psi_{N_0, N_1}(\a)$.
 However, the current presentation, only deriving full symmetry for a dense set of conditions, seems to be cleaner.}

For technical reasons, given an edge $\langle (N_0, \g_0), (N_1, \g_1)\rangle$ in $\tau_q$ and a stage $\a\in N_0\cap\g_0$, we would like to require $\MQB_{\a+1}\cap N_0$ to be a complete suborder of $\MQB_{\a+1}$; indeed, having this would be useful in the proof that our construction has the $\k$-chain condition.\footnote{We elaborate on this point at the end of Section \ref{definition of forcing notion}.} This cannot be accomplished, while defining $\MQB_{\a+1}$, on pain of circularity. However, a certain approximation to the above situation can be meaningfully stipulated, which we do,\footnote{This is clause (7) in the definition of condition.} and this suffices for our purposes.

Our construction is  $\sigma$-closed for all $\b\leq\k^{+}$.
In particular, forcing with $\MQB_{\k^{+}}$ preserves $\omega_1$ and $\CH$.
The preservation of all higher cardinals proceeds by showing that the construction has the $\kappa$-chain condition. For this, we use the weak compactness of $\kappa$ in an essential way. The proof of the $\k$-c.c.\ of $\MQB_\b$, for each $\b<\k^+$, is modelled after the corresponding proof in \cite{laver-shelah}; in fact it is a natural adaptation, to the current setting, of the proof in \cite{laver-shelah} of the $\k$-c.c.\ of the main forcing in that paper. The fact that the length of our iteration is not greater than $\kappa^+$ seems to be needed in this proof.
Finally, the copying, for a given condition $q$, of all information coming from $q$ via the edges  occurring in $\tau_q$ is crucially used in the proof that our forcing preserves $2^{\aleph_1}=\aleph_2$ (s.\ the proof of Lemma \ref{preserving ch}).

Side conditions are often employed in forcing constructions with the purpose of guaranteeing that certain cardinals are preserved. In the present construction, on the other hand, they are used to ensure that the relevant level of $\GCH$\footnote{$2^{\aleph_1}=\aleph_2$.} is preserved. This use of side conditions is taken from \cite{aspero-mota2}, where they are crucially used in the proof of $\CH$-preservation. It is worth observing that, while in the construction from \cite{aspero-mota2} a certain amount of structure is needed among the models occurring in the side condition,\footnote{Using the terminology of \cite{aspero-mota1}, they need to come from a symmetric system.} no structure whatsoever (for the underlying set of models) is needed in the present construction.
We should point out that even if it preserves $2^{\aleph_1}=\aleph_2$, our construction does add new subsets of $\omega_1$ after collapsing $\k$ to become $\o_2$, although only $\aleph_2$-many of them (cf.\ the construction in \cite{aspero-mota2}, where  $\CH$ is preserved but $\al_1$-many new reals are added).

The paper is organized as follows. In Section \ref{Models with markers and edges} we define the notions of model with marker and edge, which we will be using throughout the paper, and prove some of their basic properties.
 In Section \ref{definition of forcing notion} we define our forcing construction and prove some if its basic properties.
 In Section \ref{preservation properties} we show that the forcing has the $\kappa$-chain condition. This is the most elaborate proof in the paper.\footnote{Cf.\ the proof in \cite{laver-shelah}, where the hardest part is to prove that the forcing is $\k$-c.c., or
 the proof in \cite{aspero-mota2}, where the hardest part is to prove that the forcing is proper.}  Finally, in Section \ref{Completing the proof of Main theorem}  we complete the proof of Theorem \ref{main theorem}. The main argument in this section is to show that our forcing preserves $2^{\al_1}=\al_2$.

\section{Models with markers and edges}\label{Models with markers and edges}

In this section we set up the side condition part of our main forcing construction and discuss some of its properties. As we will see, our side condition forcing (i.e., the collection of our side conditions, with the natural extension relation) is a trivial forcing notion in the sense that any two conditions are compatible.

Let us fix, for the remainder of this paper, a weakly compact cardinal $\kappa$, and let us assume, without loss of generality, that $2^{\mu}=\mu^{+}$ for every cardinal $\mu\geq\kappa$.\footnote{In fact, if $\kappa$ is weakly compact, then $\GCH$ at every cardinal $\mu\geq\kappa$ can be easily arranged by collapsing cardinals with conditions of size $\leq \kappa$, which will preserve the weak compactness of $\kappa$.}

 Given functions $f_0,\ldots, f_n$, for $n<\o$, we let $$f_n \circ \ldots\circ f_0$$ denote the function $f$ with domain the set of $x\in \dom(f_0)$ such that for every $i<n$, $(f_i\circ\ldots\circ f_0)(x)\in\dom(f_{i+1})$, and such that for every $x\in \dom(f)$, $f(x)=f_n((f_{n-1}\circ\ldots\circ f_0)(x))$. For a function $f$ and a set $x$ we let $f(x)$ denote the empty set whenever $x\notin \dom(f)$.

Throughout the paper, if $N$ is a set such that $N\cap\kappa$ is an ordinal, we denote this ordinal by $\delta_N$ and call it \emph{the height of $N$}. If $X$ is a set, we set $$\cl(X)=X\cup\{\a\in\Ord\,\mid\,\a=\sup(X\cap\a)\}$$ If, in addition, $\g$ is an ordinal, we let $\g_X$ be the highest ordinal $\x\in \cl(X)$ such that $\x\leq\g$.

 Let $\mtcl X=\{\omega_1\cdot \alpha\,\mid\,\a <\k^+\}$.\footnote{Where, here and elsewhere, the dot in $\o_1\cdot\a$ denotes ordinal multiplication.}
  Given an ordinal $\a <\k^+$, there is a unique representation $\a=\a_0+\nu$, where $\a_0\in \mtcl X$ and $\nu<\o_1$. We will  denote the above ordinal $\a_0$ by $\m(\a)$.

Let $$\Phi:\mtcl X \longrightarrow H(\kappa^{+})$$ be such that for each $x \in  H(\kappa^+)$, $\Phi^{-1}(x)$ is a stationary subset of
$\mtcl X$.

This function $\Phi$ exists by $2^\kappa=\k^+$. Also, let $F:\k^{+}\longrightarrow H(\k^{+})$ be a bijection which is definable over the structure $\langle H(\k^{+}), \in, \Phi\rangle$. (We may for example let $\langle W_\alpha \mid \alpha<\kappa^+\rangle$ be the $<_\Phi$-increasing enumeration of $\{\Phi^{-1}(x) \mid,x\in H(\kappa^+)\}$, where $<_\Phi$ is defined by setting $\Phi^{-1}(x)<_\Phi\Phi^{-1}(y)$ iff $\min(\Phi^{-1}(x))<\min(\Phi^{-1}(y))$, and then we may define $F$ so that $F^{-1}(x)=\alpha$ if and only $\Phi^{-1}(x)=W_\alpha$.)
 Let also $\Phi_0$ be the satisfaction predicate for the structure $\langle H(\k^{+}), \in, \Phi\rangle$.

\begin{definition}[Models with markers]
\label{defmarker}
An ordered pair $(N, \g)$ is called a \emph{model with marker} if and only if:
\begin{enumerate}
\item $(N, \in, \Phi_0)\prec (H(\k^{+}), \in, \Phi_0)$. \footnote{Given $N\subseteq H(\k^+)$ and predicates $P_0,\ldots, P_n \subseteq H(\k^{+})$, we will tend to write $(N, \in, P_0,\ldots, P_n)$ as short-hand for $(N, \in, P_0\cap N,\ldots, P_n\cap N)$.}
\item $N\cap\kappa\in\kappa$, $\arrowvert N\arrowvert=\arrowvert N\cap\kappa\arrowvert$, and  $^{{<}\arrowvert N\arrowvert}N\subseteq N$.
\item $\g\in \cl(N)\cap\k^{+}$.

\end{enumerate}
\end{definition}
We will often use, without mention, the fact that $(N, \g)\in N'$ whenever $(N, \g)$ and $(N', \g')$ are models with markers and $N\in N'$.\footnote{Note that $N \in N'$ implies $\g \in N'$ as well. This is because, $\cl(N) \cap \k^+ \in N'$ and $\cl(N) \cap \k^+$ has size less than $\k$,
so $\cl(N) \cap \k^+\subseteq N'$.}

\begin{notation}
Given models $N_0$ and $N_1$ such that $(N_0, \in)\cong (N_1, \in)$, we will denote the unique $\in$-isomorphism $\Psi:(N_0, \in)\to (N_1, \in)$ by $\Psi_{N_0, N_1}$.
\end{notation}

Given any nonzero ordinal $\eta<\kappa^{+}$,  let $e_\eta$ be the $F$-least surjection from $\kappa$ onto $\eta$. Let $\vec e=\langle e_\eta\mid 0<\eta<\k^{+}\rangle$. We will say that a model $N\subseteq H(\kappa^{+})$ is \emph{closed under $\vec e$}  if $e_\eta(\x)\in N$  for every nonzero $\eta\in N\cap\k^{+}$  and every $\x\in \kappa\cap N$.

\begin{lemma}\label{agreement} Suppose $N_0$ and $N_1$ are models closed under $\vec e$ of the same height. Then $N_0\cap N_1\cap\kappa^{+}$ is an initial segment of both $N_0\cap\kappa^{+}$ and $N_1\cap\kappa^{+}$. In particular, if $(N^i_0, \g^i_0)$ and $(N^i_1, \g^i_1)$ (for $i\leq n$) are models with markers such that $(N^i_0, \in, \Phi_0)\cong (N^i_1, \in, \Phi_0)$ for all $i\leq n$,
then $$(\Psi_{N^n_0, N^n_1}\circ\ldots\circ \Psi_{N^0_0, N^0_1})(x)=x$$ for every $x\in \dom(\Psi_{N^n_0, N^n_1}\circ\ldots\circ \Psi_{N^0_0, N^0_1})\cap N^n_1$.
\end{lemma}
\begin{proof}
Let us first prove the first assertion. Given any nonzero $\eta\in N_0\cap N_1\cap\kappa^{+}$ and any $\a\in N_0\cap\eta$ there is some $\xi\in N_0\cap\kappa$ such that $e_\eta(\xi)=\alpha$. But since $\eta$ and $\xi$ are both members of $N_1$, we also have that $\alpha=e_\eta(\xi)\in N_1$.

As to the second assertion, let us first consider the case in which $\d_{N^i_0}=\d_{N^{i'}_0}$ for all $i$, $i'$.
By the choice of $\vec{e}$, each of the models $N^i_\epsilon$, for $i \leq n$ and $\epsilon \in \{0, 1\}$, is closed under $\vec{e}$.

Let $x\in \dom(\Psi_{N^n_0, N^n_1}\circ\ldots\circ \Psi_{N^0_0, N^0_1}) \cap N^n_1$ and $\a=F^{-1}(x)$. We first prove by induction on $i\leq n$ that $$\ot(N^0_0\cap \a)=\ot(N^i_1\cap (\Psi_{N^i_0, N^i_1}\circ\ldots\circ\Psi_{N^0_0, N^0_1})(\a))$$ and $$(\Psi_{N^i_0, N^i_1}\circ\ldots\circ\Psi_{N^0_0, N^0_1})(x)=F((\Psi_{N^i_0, N^i_1}\circ\ldots\circ\Psi_{N^0_0, N^0_1})(\a)).$$

For $i=0$ this is true since $\Psi_{N^0_0, N^0_1}$ is an isomorphism between the structures $(N^0_0, \in, \Phi)$ and $(N^0_1, \in, \Phi)$. For $i>0$, assuming the above equalities hold for $i-1$, we have that $(\Psi_{N^{i-1}_0, N^{i-1}_1}\circ\ldots\circ\Psi_{N^0_0, N^0_1})(x)\in N^i_0$ and therefore $$(\Psi_{N^{i-1}_0, N^{i-1}_1}\circ\ldots\circ\Psi_{N^0_0, N^0_1})(\a)\in N^i_0$$ since  $$(\Psi_{N^{i-1}_0, N^{i-1}_1}\circ\ldots\circ\Psi_{N^0_0, N^0_1})(x)=F((\Psi_{N^{i-1}_0, N^{i-1}_1}\circ\ldots\circ\Psi_{N^0_0, N^0_1})(\a))$$ and $F:\kappa^+\longrightarrow H(\kappa^+)$ is a bijection. Then we have that $$\ot(N^i_0\cap (\Psi_{N^{i-1}_0, N^{i-1}_1}\circ\ldots\circ \Psi_{N^0_0, N^0_1})(\a))$$ and $$\ot(N^i_1\cap (\Psi_{N^0_i, N^i_1}\circ\Psi_{N^{i-1}_0, N^{i-1}_1}\circ\ldots\circ \Psi_{N^0_0, N^0_1})(\a))$$ are equal (since $\Psi_{N^i_0, N^i_1}$ is an $\in$-isomorphims between $N^i_0$ and $N^i_1$) and $$\ot( N^{i-1}_1\cap (\Psi_{N^{i-1}_0, N^{i-1}_1}\circ\ldots\circ\Psi_{N^0_0, N^0_1})(\a))$$ and $$\ot(N^i_0\cap (\Psi_{N^{i-1}_0, N^{i-1}_1}\circ\ldots\circ\Psi_{N^0_0, N^0_1})(\a))$$ are equal (by the first assertion as $(\Psi_{N^{i-1}_0, N^{i-1}_1}\circ\ldots\circ\Psi_{N^0_0, N^0_1})(\a)\in N^{i+1}_1\cap N^i_0$). Hence, $$\ot(N^0_0\cap \a)= \ot(N^i_1\cap (\Psi_{N^i_0, N^i_1}\circ\ldots\circ\Psi_{N^0_0, N^0_1})(\a))$$ For the  second conclusion, we note that $$(\Psi_{N^i_0, N^i_1}\circ\ldots\circ\Psi_{N^0_0, N^0_1})(x)=(\Psi_{N^i_0, N^i_1}\circ F\circ \Psi_{N^{i-1}_0, N^{i-1}_1}\circ\ldots\circ \Psi_{N^0_0, N^0_1})(\a),$$ which is equal to $F(\Psi_{N^i_0, N^i_1}\circ\ldots\Psi_{N^0_0, N^0_1})(\a))$ since $\Psi_{N^i_0 N^i_1}$ is an isomorphism between $(N^i_0, \in, \Phi)$ and $(N^i_1, \in, \Phi)$.

Finally, we note that $\ot(N^n_1\cap\a)=\ot(N^0_0\cap\a)$ by the first assertion as $\a=F^{-1}(x)\in N^0_0\cap N^n_1$. Hence, $$\ot(N^n_1\cap (\Psi_{N^n_0, N^n_1}\circ\ldots\circ \Psi_{N^0_0, N^0_1})(\a))=\ot(N^n_1\cap\a)$$ and therefore $(\Psi_{N^n_0, N^n_1}\circ\ldots\circ \Psi_{N^0_0, N^0_1})(\a)=\a$. Then also $$(\Psi_{N^n_0, N^n_1}\circ\ldots\circ \Psi_{N^0_0, N^0_1})(x)=F(\Psi_{N^n_0,  N^n_1}\circ\ldots\circ \Psi_{N^0_0, N^0_1})(\a)=F(\a)=x$$

Suppose now that $\d_{N^i_0}\neq\d_{N^{i'}_0}$ for some $i\neq i'$ and let $i^\ast$ be such that $\d_{N^{i^\ast}_0}=\min\{\d_{N^i_0}\mid i\leq n\}$. Given any $i$ and $\epsilon$, we know that $(N^i_\epsilon, \in, \Phi_0)\prec (H(\k^+), \in, \Phi_0)$ and that $N^i_\epsilon$ is closed under sequences of length less than $|N^i_\epsilon|$. Also, $|N^{i^{\ast\ast}}_{\epsilon}|=|\d_{N^{i^\ast}_0}|$ for every $\epsilon\in\{0, 1\}$ and every $i^{\ast\ast}$ such that $\d_{N^{i^{\ast\ast}}_0}=\d_{N^{i^\ast}_0}$. But now it easily follows from the above facts that there is a sequence $(\langle \bar N^i_{0}, \bar N^i_{1}\rangle)_{i\leq n}$ of pairs of models with the following properties.

\begin{itemize}
\item For all $i\leq n$, $(\bar N^i_0, \in, \Phi)$ and $(\bar N^i_1, \in, \Phi)$ are isomorphic elementary submodels of $(H(\k^+), \in, \Phi)$ and $\d_{\bar N^i_0}=\d_{N^{i^*}_0}$.
\item For every $i\leq n$ and every $\epsilon\in \{0, 1\}$, $\bar N^i_\epsilon= N^i_\epsilon$ or $\bar N^i_\epsilon\in N^i_\epsilon$.
\item $x\in \dom(\Psi_{\bar N^n_{0}, \bar N^n_{1}}\circ\ldots\circ \Psi_{\bar N^0_{0}, \bar N^1_{1}}) \cap \bar N^n_1$

\item $(\Psi_{N^n_0, N^n_1}\circ\ldots\circ \Psi_{N^0_0, N^0_1})(x)=(\Psi_{\bar N^n_{0}, \bar N^n_{1}}\circ\ldots\circ \Psi_{\bar N^0_{0}, \bar N^1_{1}})(x)$
\end{itemize}
Indeed, we can define $\bar N^{i^*+j}_0$, $\bar N^{i^*+j}_1$, $\bar N^{i^*-k}_0$ and $\bar N^{i^*-k}_1$, for $j\leq n-i^*$ and $k\leq i^*$, by recursion as follows.
\begin{itemize}
\item If $\d_{N^{i^*+j}_0}=\d_{N^{i^*}_0}$, then $\bar N^{i^*+j}_0=N^{i^*+j}_0$ and $\bar N^{i^*+j}_1=N^{i^*+j}_1$, and if $\d_{N^{i^*-k}_0}=\d_{N^{i^*}_0}$, then $\bar N^{i^*-k}_0=N^{i^*-k}_0$ and $\bar N^{i^*-k}_1=N^{i^*-k}_1$.
\item If $\d_{N^{i^*}_0}<\d_{N^{i^*+j}_0}$ (in which case $j>0$), then $\bar N^{i^*+j}_0\in N^{i^*+j}_0$ is such that
\begin{itemize}
\item $(N^{i^*+j}_0, \in, \Phi)\prec (H(\kappa^+), \in, \Phi)$,
\item $(\bar N^{i^*+j-1}_1, \in, \Phi)\cong (\bar N^{i^*+j}_0, \in, \Phi)$, and
\item $\bar  N^{i^*+j-1}_1\cap N^{i^*+j}_0\subseteq \bar N^{i^*+j}_0$,
\end{itemize}
\noindent and $\bar N^{i^*+j}_1=\Psi_{N^{i^*+j}_0, N^{i^*+j}_1}(\bar N^{i^*+j}_0)$.
\item If $\d_{N^{i^*}_0}<\d_{N^{i^*-k}_0}$ (in which case $k>0$), then $\bar N^{i^*-k}_1\in N^{i^*-k}_1$ is such that
\begin{itemize}
\item $(N^{i^*-k}_1, \in, \Phi)\prec (H(\kappa^+), \in, \Phi)$,
\item $(\bar N^{i^*-k+1}_0, \in, \Phi)\cong (\bar N^{i^*-k}_1, \in, \Phi)$, and
\item $\bar N^{i^*-k+1}_0\cap N^{i^*-k}_1\subseteq \bar N^{i^*-k}_1$,
\end{itemize}
\noindent and $\bar N^{i^*-k}_0=\Psi_{N^{i^*-k}_1, N^{i^*-k}_0}(\bar N^{i^*-k}_1)$.
\end{itemize}

But now we are done by the previous case.
\end{proof}

It will be convenient to use the following pieces of terminology: given models with markers $(N_0, \g_0)$, $(N_1, \g_1)$,
we will say that \emph{$(N_0, \g_0)$ and $(N_1, \g_1)$ are twin models (with markers)} if and only if $(N_0, \in, \Phi_0)\cong (N_1, \in, \Phi_0)$.
If $\Psi_{N_0, N_1}(\a)\leq\a$ for every ordinal $\a\in N_0$, then we say that \emph{$N_1$ is a projection of $N_0$}.

\begin{definition}[Edge]
Suppose $\vec{\Phi}=(\Phi_\alpha)_\alpha$ is a sequence  of predicates of $H(\k^+)$ of length less than $\kappa^+$.
An ordered pair $$\langle (N_0, \g_0), (N_1, \g_1)
\rangle$$ of models with markers is called a $\vec{\Phi}$-\emph{edge} if and only if  the following are satisfied:
\begin{enumerate}
\item $(N_0, \g_0)$ and $(N_1, \g_1)$ are twin models with markers;
\item for every $\epsilon\in\{0, 1\}$ and every $\a\in N_\epsilon\cap\g_\epsilon$, $(N_\epsilon, \in, \Phi_\a)\prec (H(\k^+), \in, \Phi_\a)$;
\item $N_1$ is a projection of $N_0$;
\item $\Psi_{N_0, N_1}$ is an isomorphism between the structures $(N_0, \in, \Phi_{\a})$ and $(N_1, \in, \Phi_{\bar\a})$ for every $\a\in N_0\cap\g_0$ such that $\bar\a:=\Psi_{N_0, N_1}(\a)<\g_1$.
\end{enumerate}
Moreover, if
$\g_0\leq\b$ and $\g_1\leq\b$, then we call $\langle (N_0, \g_0), (N_1, \g_1)\rangle$ a $\vec{\Phi}$-\emph{edge below $\b$}.
\end{definition}

\begin{definition}[Generalized edge]
An ordered pair $$e=\langle (N_0, \g_0), (N_1, \g_1)\rangle$$ of models with markers is called  a $\vec{\Phi}$-\emph{anti-edge} if $$e^{-1}\coloneqq \langle (N_1, \g_1), (N_0, \g_0)\rangle$$ is a $\vec{\Phi}$-edge. We say that an ordered pair $e$ is a \emph{generalized $\vec{\Phi}$-edge} if it is a $\vec{\Phi}$-edge or a $\vec{\Phi}$-anti-edge.
\end{definition}

\begin{convention}
If $\tau$ is a set of generalized $\vec{\Phi}$-edges, we say that a generalized $\vec{\Phi}$-edge $e$ \emph{comes from $\tau$} in case $e\in \tau$ or $e^{-1}\in \tau$. We also set $\tau^{-1}=\{e^{-1}: e \in \tau    \}$.
\end{convention}
Given a generalized $\vec{\Phi}$-edge $e=\langle (N_0, \g_0), (N_1, \g_1)\rangle$ and an ordinal $\a$, we let $e\restriction\a$ denote the generalized $(\vec{\Phi}\upharpoonright \a)$-edge $$\langle (N_0, \min\{\a, \g_0\}_{N_0}), (N_1, \min\{\a, \g_1\}_{N_1})\rangle.\footnote{Recall that $\min\{\a, \g_\epsilon\}_{N_\epsilon}$, for $\epsilon \in \{0, 1\},$ is the highest ordinal $\xi \in \cl(N_\epsilon)$ such that $\xi \leq \min\{\a, \g_\epsilon\}$.}$$ Given a collection $\tau$ of $\vec{\Phi}$-edges and given an ordinal $\a$, we denote by $\tau\restriction\a$ the set $\{e\restriction\a\,\mid\, e\in \tau\}$. Note that $\tau\restriction\alpha$ is a collection of $(\vec{\Phi}\upharpoonright \a)$-edges below $\a$.

Given a sequence $\vec\Phi=(\Phi_\alpha)_\alpha$ of predicates of $H(\k^+)$, we say that \emph{for each $\a$, $\Phi_\a$ codes $\langle\Phi_\beta\mid\beta<\alpha\rangle$ in a uniform way} in case there is a formula $\varphi(x, y)$ in the language for the structures $(H(\kappa^+), \in, \Phi_\a)$ such that for all $\beta<\alpha$ less than the length of $\vec\Phi$, and for each $a\in H(\kappa^+)$, $a\in \Phi_\b$ if and only if $(H(\kappa^+), \in, \Phi_\a)\models\varphi(\b, a)$.

Given models with markers $(N, \g)$, $(N_0, \g_0)$ and $(N_1,  \g_1)$, if $N\in N_0$ and $(N_0, \in)\cong  (N_1, \in)$, then we let $\pi^{\g, N}_{N_0, \g_0, N_1, \g_1}$ denote the supremum of the set of ordinals of the form $\Psi_{N_0, N_1}(\x)$, where
\begin{itemize}
\item $\x\in N \cap (\g+1)$,
\item $\x<\g_0$, and
\item $\Psi_{N_0, N_1}(\x)< \g_1$.
\end{itemize}
We let also $\pi^{\g, N}_{e}$ denote $\pi^{\g, N}_{N_0, \g_0, N_1, \g_1}$ whenever $e=\langle (N_0, \g_0), (N_1, \g_1)\rangle$ is a $\vec{\Phi}$-edge.
Given generalized $\vec{\Phi}$-edges $e=\langle (N_0, \g_0), (N_1, \g_1)\rangle$ and $e'=\langle (N_0', \g_0'), (N_1', \g_1')\rangle$ such that $e'\in N_0$, we denote $$\langle (\Psi_{N_0, N_1}(N_0'), \p^{\g_0', N_0'}_{N_0, \g_0, N_1, \g_1}),  (\Psi_{N_0, N_1}(N_1'), \p^{\g_1', N_1'}_{N_0, \g_0, N_1, \g_1})\rangle$$ by $\Psi_e(e')$. Note that if $\vec\Phi$ is such that for each $\alpha$, $\Phi_\alpha$ codes $\langle\Phi_\beta\mid\,\beta<\alpha\rangle$ in a uniform way, then $\Psi_e(e')$ is a generalized $\vec{\Phi}$-edge.

\begin{definition}[Closedness under copying]
Suppose $\vec\Phi$ is such that for each $\alpha$, $\Phi_\alpha$ codes $\langle\Phi_\beta\mid\,\beta<\alpha\rangle$ in a uniform way. A set $\tau$ of $\vec{\Phi}$-edges is \emph{closed under copying} in case for all $\vec{\Phi}$-edges $e=\langle (N_0, \g_0), (N_1, \g_1)\rangle$ and $e'=\langle (N'_0, \g'_0), (N'_1, \g'_1)\rangle$ in $\tau$ such that $e'\in  N_0$ there are ordinals $\g_0^*\geq \p^{\g_0', N_0'}_{e}$ and $\g_1^*\geq \p^{\g_1', N_1'}_{e}$ such that $$\langle (\Psi_{N_0, N_1}(N'_0), \g^*_0), (\Psi_{N_0, N_1}(N'_1), \g_1^*)\rangle \in\tau.$$
\end{definition}

Given a sequence $\vec{\mtcl E}=(\langle (N^i_0, \g^i_0), (N^i_1, \g^i_1)\rangle\,\mid\,i< n)$ of generalized $\vec{\Phi}$-edges, we will tend to denote the expression $$\Psi_{N^{n-1}_0, N^{n-1}_1}\circ\ldots\circ \Psi_{N^0_0, N^0_1}$$ by $\Psi_{\vec{\mtcl E}}$. If $\vec{\mtcl E}$ is the empty sequence, we let $\Psi_{\vec{\mtcl E}}$ be the identity function.

\begin{definition}[Pure side conditions forcing]
Suppose $\vec{\Phi}=(\Phi_\alpha)_\alpha$ is a sequence of predicates of $H(\k^+)$ such that for each $\alpha$, $\Phi_\alpha$ codes $\langle\Phi_\beta\mid\,\beta<\alpha\rangle$ in a uniform way. Let $\beta < \kappa^+$. Let
$\MPB^{\textsf{e}}_{\vec{\Phi}, \beta}$ be the set of all countable sets $\tau$ of $\vec{\Phi}$-edges below
$\beta$ which are closed under copying and $\vec e$.
Given conditions $\tau_0$ and $\tau_1$ in $\MPB^{\textsf{e}}_{\vec{\Phi}, \beta}$, let $\tau_1 \leq \tau_0$ if for every $\langle (N_0,  \g_0), (N_1, \g_1)\rangle \in \tau_{0}$ there are $\g_0'\geq\g_0$ and $\g_1' \geq \g_1$ such that $\langle (N_0,  \g'_0), (N_1, \g'_1)\rangle\in \tau_{1}$.
\end{definition}

The next lemma shows that
$\MPB^{\textsf{e}}_{\vec{\Phi}, \beta}$  is the trivial forcing notion.

\begin{lemma}
\label{pedgeistrivial}
Given sets $\tau^0$ and $\tau^1$ of $\vec{\Phi}$-edges, there exists a smallest set $\tau=\tau^0 \oplus\tau^1$ of $\vec{\Phi}$-edges which contains both $\tau_0$
and $\tau_1$ and is closed under copying. Furthermore, if both $\tau^0$ and $\tau^1$ are from $\MPB^{\textsf{e}}_{\vec{\Phi}, \beta}$, then so is
$\tau.$
\end{lemma}

\begin{proof}
Let $\tau^0\oplus \tau^1$ be the natural amalgamation of $\tau^0$ and $\tau^1$ obtained by taking copies of $\vec{\Phi}$-edges as dictated by suitable functions $\Psi_{\vec{\mtcl E}}$, so that $\tau^0\oplus \tau^1$ is closed under copying (where the $\vec{\Phi}$-edges generated by this copying procedure have minimal marker so that $\tau^0\oplus\tau^1$ is closed under copying).
To be more specific,
$\tau^0\oplus \tau^1 = \bigcup_{n<\o}\tau_n$, where
$(\tau_n)_n$ is the following sequence.

\begin{enumerate}
\item $\tau_0=\tau^0\cup\tau^1$
\item For each $n<\o$, $\tau_{n+1}=\tau_n\cup \tau_n'$, where $\tau_n'$ is the set of $\vec{\Phi}$-edges of the form $\Psi_e(e')$,
for $\vec{\Phi}$-edges $e=\langle (N_0, \g_0), (N_1, \g_1)\rangle$ and $e'$ in $\tau_n$ such that $e'\in N_0$.
\end{enumerate}
Then $\tau=\tau^0\oplus\tau^1$ is as required.
\end{proof}

Given $\tau^0$ and $\tau^1$, two sets of $\vec{\Phi}$-edges, the construction in the proof of Lemma \ref{pedgeistrivial} of $\tau^0\oplus \tau^1$ as $\bigcup_{n<\o}\tau_n$ gives rise to a natural notion of rank on the set of generalized $\vec{\Phi}$-edges coming from $\tau^0\oplus\tau^1$. Specifically, given $n<\o$, a generalized $\vec{\Phi}$-edge $e=\langle (N_0, \g_0), (N_1, \g_1)\rangle$ coming from $\tau^0\oplus\tau^1$ has \emph{$(\tau^0, \tau^1)$-rank} $n$ iff $n$ is least such that $e$ comes from $\tau_n$. Alternatively, we may define the \emph{$(\tau^0, \tau^1)$-rank of $e$} as follows.
\begin{itemize}
\item $e$ has $(\tau^0, \tau^1)$-rank $0$ if $e$ comes from $\tau^0\cup \tau^1$.
\item For every $n<\o$, a generalized $\vec{\Phi}$-edge $e$ coming from $\tau^0\oplus\tau^1$ has $(\tau^0, \tau^1)$-rank $n+1$ iff $e$ does not have $(\tau^0, \tau^1)$-rank $m$ for any $m\leq n$ and there are $\vec{\Phi}$-edges $e_0=\langle (N_0, \g_0), (N_1, \g_1)\rangle$ and $e_1$ coming from $\tau^0\oplus\tau^1$ and such that
\begin{itemize}
\item the maximum of the $(\tau^0, \tau^1)$-ranks of $e_0$ and $e_1$ is $n$,
\item $e_1\in N_0$, and \item $e=\Psi_{e_0}(e_1)$.
\end{itemize}
\end{itemize}

\begin{definition}[$\tau$-thread]
Given a set $\tau$ of $\vec{\Phi}$-edges, a sequence $\vec{\mtcl E}=(\langle (N^i_0, \g^i_0), (N^i_1, \g^i_1)\rangle\,\mid\,i< n)$ of generalized $\vec{\Phi}$-edges coming from $\tau$, and $x\in N^0_0$, we will call $\langle \vec{\mtcl E}, x\rangle$ a \emph{$\tau$-thread} in case $x\in \dom(\Psi_{\vec{\mtcl E}})$. In the above situation, if $x=(y, \a)$, where $y\in H(\k^+)$ and $\a<\k^+$, we call $\langle \vec{\mtcl E}, x\rangle$ a \emph{correct $\tau$-thread} if and only if
\begin{enumerate}
\item $\a< \g^0_0$,
\item $\Psi_{\vec{\mtcl E}}(\a)\in \g^{n-1}_1$, and
\item $\Psi_{\vec{\mtcl E}}$ is a (partially defined) elementary embedding from the structure $(N^0_0, \in, \Phi_\a)$ into the structure $(N^{n-1}_1, \in, \Phi_{\Psi_{\vec{\mtcl E}}(\a)})$.
\end{enumerate}
\end{definition}

We will sometimes just say \emph{thread} when $\tau$ is not relevant.
It will be useful to consider the following strengthening of the notion of correct thread:
\begin{definition}
Given a set $\tau$ of $\vec{\Phi}$-edges, a sequence $$\vec{\mtcl E}=(\langle (N^i_0, \g^i_0), (N^i_1, \g^i_1)\rangle \mid i<n)$$ of generalized $\vec{\Phi}$-edges coming from $\tau$, and $x=(y, \a)\in H(\k^+)\times\k^+$, $\langle\vec{\mtcl E}, x\rangle$ is a  \emph{connected $\tau$-thread} in case
\begin{enumerate}
\item $\langle\vec{\mtcl E}, x\rangle$ is a $\tau$-thread,
\item $\a<\g^0_0$, and
\item for each $i<n$,
\begin{enumerate}
\item $(\Psi_{N^0_i, N^1_i}\circ\ldots\circ\Psi_{N^0_0, N^0_1})(\a)<\g^i_1$, and
\item $(\Psi_{N^0_i, N^1_i}\circ\ldots\circ\Psi_{N^0_0, N^0_1})(\a)<\g^{i+1}_0$ if $i+1<n$.\end{enumerate}\end{enumerate}
\end{definition}

\begin{remark} \normalfont Given a set $\tau$ of $\vec{\Phi}$-edges, every connected $\tau$-thread is correct.\end{remark}

The following lemma can be easily proved by induction on the $(\tau^0, \tau^1)$-rank of the members of $\tau^0\oplus\tau^1$.

\begin{lemma}\label{high-amalg} Suppose for each $\alpha$, $\Phi_\alpha$ codes $\langle\Phi_\beta\mid\,\beta<\alpha\rangle$ in a uniform way.  Let $\tau^0$ and $\tau^1$ be sets of $\vec{\Phi}$-edges, and let $\l<\k$ be an ordinal such that all members of $\tau^0$ involve models of height less than $\l$. Suppose $\tau^1$ is closed under copying. Then all members of $\tau^0\oplus\tau^1$ involving models of height at least $\l$ are in $\tau^1$.
\end{lemma}

As we will see, the following lemma will enable us to ease our path through the proof of Claim \ref{separating-pair-compatibility}, in Section \ref{preservation properties}, in a significant way.

\begin{lemma}\label{orbits}
 Suppose for each $\alpha$, $\Phi_\alpha$ codes $\langle\Phi_\beta\mid\,\beta<\alpha\rangle$ in a uniform way.
For all sets $\tau^0$ and $\tau^1$ of $\vec{\Phi}$-edges, every set $x$, and every
 $\tau^0\oplus \tau^1$-thread $\langle\vec{\mtcl E}, x\rangle$ there is a
 $\tau^0\cup\tau^1$-thread $\langle\vec{\mtcl E}_\ast, x\rangle$ such that $$\Psi_{\vec{\mtcl E}}(x)=\Psi_{\vec{\mtcl E}_\ast}(x)$$ Furthermore, if $x=(y, \a)\in H(\k^+)\times\k^+$ and $\langle\vec{\mtcl E}, x\rangle$ is connected, then $\vec{\mtcl E}_\ast$ may be chosen to be connected as well.

\end{lemma}

\begin{proof}
Let $\vec{\mtcl E}=(e_i\,\mid\, i\leq n)$, where $e_i=\langle (N^i_0, \g^i_0), (N^i_1, \g^i_1)\rangle$ for each $i$. We aim to prove that there is a
$\tau^0\oplus\tau^1$-thread $\langle\vec{\mtcl E}_\ast, x\rangle$ with the following properties.
\begin{enumerate}
\item $\Psi_{\vec{\mtcl E}}(x)=\Psi_{\vec{\mtcl E}_\ast}(x)$
\item If every $e_i$ has $(\tau^0, \tau^1)$-rank $0$, then $\vec{\mtcl E}_\ast=\vec{\mtcl E}$.
\item If some $e_i$ has positive $(\tau^0, \tau^1)$-rank, then the maximum $(\tau^0, \tau^1)$-rank of the members of $\vec{\mtcl E}_\ast$ is strictly less than the maximum $(\tau^0, \tau^1)$-rank of the members of $\vec{\mtcl E}$.
\item The following holds, where $\vec{\mtcl E}_\ast=(e^\ast_i\,\mid\,i\leq n^*)$.
\begin{enumerate}
\item $e^\ast_0  = e_0$ if $e_0$ comes from $\tau^0\cup\tau^1$.

\item If there are generalized $\vec{\Phi}$-edges $e=\langle(N_0, \g_0), (N_1, \g_1)\rangle$ and $e'$ coming from $\tau^0\oplus \tau^1$, both of rank less than the rank of $e_0$, such that
$e'\in N_0$ and $e_0=\Psi_e(e')$,
then $$e^\ast_0= \langle (N_1, \g_1), (N_0, \g_0)\rangle,$$ where $\langle (N_0, \g_0), (N_1, \g_1)\rangle$ is some generalized $\vec{\Phi}$-edge as a\-bove.
\item $e^*_{n^*}  = e_n$ if $e_n$ comes from $\tau^0\cup\tau^1$.
\item If there are generalized $\vec{\Phi}$-edges $e=\langle(N_0, \g_0), (N_1, \g_1)\rangle$ and $e'$ coming from $\tau^0\oplus \tau^1$, both of rank less than the rank of $e_n$, such that $e'\in N_0$ and $e_n=\Psi_e(e')$,
 then $$e^\ast_{n^*}= \langle (N_0, \g_0), (N_1, \g_1)\rangle,$$ where $\langle (N_0, \g_0), (N_1, \g_1)\rangle$ is some generalized $\vec{\Phi}$-edge as a\-bove.
\end{enumerate}
\item If $x=(y, \a)\in H(\k^+)\times\k^+$ and $\langle\vec{\mtcl E}, x\rangle$ is connected, then $\vec{\mtcl E}_\ast$ is connected.
\end{enumerate}


The proof of (1)--(5) will be by induction on $n$. We may obviously assume that there is some $i<n$ such that $e_i$ does not come from $\tau^0\cup\tau^1$. Then there are generalized $\vec{\Phi}$-edges $e=\langle (N_0, \g_0), (N_1, \g_1)\rangle$ and $e'\in N_0$ coming from $\tau^0\oplus \tau^1$, both of rank less than $e_i$, and such that $e_i=\Psi_e(e')$.

By induction hypothesis there is a
$\tau^0\oplus \tau^1$-thread $\langle\vec{\mtcl E}_0, x\rangle$, together with a
$\tau^0\oplus \tau^1$-thread of the form $\langle\vec{\mtcl E}_2, \Psi_{\vec{\mtcl E}\restriction i+1}(x)\rangle,$ such that $$\Psi_{\vec{\mtcl E}_0}(x)=\Psi_{\vec{\mtcl E}\restriction i}(x)$$ and
$$\Psi_{\vec{\mtcl E}_2}(\Psi_{\vec{\mtcl E}\restriction i+1}(x))=\Psi_{\vec{\mtcl E}\restriction [i+1, n)}(\Psi_{\vec{\mtcl E}\restriction i+1}(x))=\Psi_{\vec{\mtcl E}}(x),$$ and such that the relevant instances of (1)--(5) hold for $\langle\vec{\mtcl E}_0, x\rangle$ and $\langle\vec{\mtcl E}_2, \Psi_{\vec{\mtcl E}\restriction i+1}(x)\rangle$.
Also, by the choice of $e_i$, the thread $\langle \vec{\mtcl E}_1, \Psi_{\vec{\mtcl E}\restriction i}(x)\rangle$ sa\-tis\-fies the instances of (1) and (3)--(5) corresponding to $\langle (e_i), \Psi_{\vec{\mtcl E}\restriction i}(x)\rangle$, where $(e_i)$ is the sequence whose only member is $e_i$, and where $\vec{\mtcl E}_1=(e^{-1}, e', e)$.
But now we may take $\vec{\mtcl E}_\ast$ to be the concatenation of $\vec{\mtcl E}_0$, $\vec{\mtcl E}_1$, and $\vec{\mtcl E}_2$.

Finally, it follows from clause (3) that after iterating the above construction some finite number of times we obtain a $\tau^0\cup\tau^1$-thread $\langle\vec{\mtcl E}_\ast, x\rangle$ as desired.
\end{proof}

\section{Definition of the forcing and its basic properties}
\label{definition of forcing notion}

We shall now define our sequence $\langle\MQB_\beta \mid \beta \leq \kappa^{+}\rangle$ of forcing notions and our sequence $\langle\Phi_\b \mid 0<\b<\k^+\rangle$ of predicates.\footnote{The reader should keep in mind the overview of the construction we gave in the introduction.} We recall that $\Phi_0$ has already been defined.

For each $\a\in\mtcl X$, assuming $\MQB_\a$ has been defined and there is a $\MQB_\a$-name $\lusim{T}\in H(\k^+)$ for a $\k$-Arosnzajn tree,
we let $\lusim{T}_\a$ be such a $\MQB_\a$-name. Further, if $\Phi(\a)$ is a $\MQB_\a$-name for a $\kappa$-Aronszajn tree,
then we let
 $\lusim{T}_\a = \Phi(\a)$. For simplicity of exposition we will assume that the universe of $\lusim{T}_\a$ is forced to be $\kappa \times \omega_1$ and that for each $\rho < \kappa,$ its $\rho$-th level is $\{\rho\} \times \omega_1$. We will often refer to members of $\k\times\o_1$ as \emph{nodes}.

As we will see, each forcing notion $\MQB_\beta$ in our construction will consist of ordered pairs of the form $q=(f_q, \tau_q)$, where $f_q$ is a function and $\tau_q$ is a set of edges below $\b$.
Given a nonzero ordinal $\a<\k^+$ and an ordinal $\d<\k$, we will write $\MQB^{\d}_{\a+1}$ to denote the suborder of $\MQB_{\a+1}$ consisting of those conditions $q$ such that $\d_{N_0}<\d$ for every edge $\langle (N_0, \g_0), (N_1, \g_1)\rangle$ coming from $\tau_q$ such that at least one of $\g_0$, $\g_1$ is $\a+1$.\footnote{We note that there is no requirement on the heights of the nodes occurring in $f_q(\a)$. Also, despite possible first impressions, there is no circularity in the definition of $\MQB^\d_{\a+1}$ (s.\ Remark \ref{clarifying}).}

Now suppose that $\beta \leq \kappa^{+}$ and that $\MQB_\alpha$ and $\Phi_\a$ have been defined for all $\alpha < \beta$.
Given an ordered pair $q=(f_q, \tau_q)$, where $f_q$ is a function and $\tau_q$ is a set of edges, and given an ordinal $\a$, we denote by $q\restriction\alpha$ the ordered pair $(f_q\restriction \a, \tau_q\restriction\a)$.

We are now ready to define $\MQB_\beta$ and $\Phi_\b$.

We start with the definition of $\MQB_\beta$. A condition in $\MQB_\beta$ is an ordered pair of the form $q= (f_q, \tau_q)$ with the following properties.

\begin{enumerate}

\item $f_q$ is a countable function such that $$\dom(f_q) \subseteq \b$$ and such that the following holds for every $\alpha \in \dom(f_q)$.

\begin{enumerate}

\item If $\a=0$, then $f_q(\a)$ is a condition in $\Col(\omega_1,\, {<}\kappa)$, the L\'{e}vy collapse turning $\kappa$ into $\aleph_2$, i.e., $f_q(0)$ is a countable function with domain included in $\kappa\times\omega_1$ such that $(f_q(0))(\rho, \xi)<\rho$ for all $(\rho, \xi)\in\dom(f_q(0))$.

\item If $\a>0$, then $f_q(\alpha)\in [\kappa\times \omega_1]^{{\leq}\al_0}$.
\end{enumerate}

\item $\tau_q$ is a countable set of $(\vec{\Phi}\restriction \b)$-edges below $\b$.

\item $q\restriction\a\in\MQB_\a$ for all $\alpha < \beta$.

\item For every nonzero $\alpha <\b$ such that $\lusim{T}_{\m(\a)}$ is defined,\footnote{We will see that in fact each $\lusim{T}_{\m(\a)}$ is defined.}
if $x_0\neq x_1$ are nodes in $f_q(\a)$, then $q\restriction\m(\a)$ forces $x_0$ and $x_1$ to be incomparable in $\lusim{T}_{\m(\a)}$.

\item $\tau_q$ is closed under copying.

\item For every edge $\langle (N_0, \g_0), (N_1, \g_1)\rangle\in \tau_q$ and  every $\a\in N_0\cap\g_0$ such that $\bar\a:=\Psi_{N_0, N_1}(\a)<\g_1$, if $\a\neq 0$, $\a\in\dom(f_q)$, and $x\in f_q(\a)\cap N_0$, then
\begin{enumerate}
\item $\bar\a\in\dom(f_q)$, and
\item $x \in f_q(\bar\a)$.\footnote{Note that $\Psi_{N_0, N_1}(x)=x$.}
 \end{enumerate}

  \item Suppose $\a<\b$, $e=\langle (N_0, \a+1), (N_1, \g_1)\rangle$ is a generalized $(\vec{\Phi}\upharpoonright \a+1)$-edge coming from  $\tau_q\restriction\a+1$,  $r\in\MQB^{\d_{N_0}}_{\a+1}$ is such that $e\restriction\a$ comes from $\tau_r$,
  $\vec{\mtcl E}=(\langle (N^i_0, \g^i_0), (N^i_1, \g^i_1)\rangle\,\mid\,i< n)$ is a sequence of generalized $(\vec{\Phi}\restriction \a+1)$-edges coming from $\tau_r\cup\{e\}$ such that $\langle (N^0_0, \g^0_0), (N^0_1, \g^0_1)\rangle=e$ and $\langle \vec{\mtcl E}, (\emptyset, \a)\rangle$ is a correct thread. Let $\d=\min\{\d_{N^i_0} \mid i<n\}$ and $\bar\a= \Psi_{\vec{\mtcl E}}(\a)$. Suppose
  $r\restriction\m(\a)$ forces every two distinct nodes in $f_r(\bar\a)\cap (\d\times\o_1)$ to be incomparable in $\lusim{T}_{\m(\a)}$.
 Then there is an extension $r^*\in\MQB^{\d_{N_0}}_{\a+1}$ of $r$ such that
 \begin{enumerate}
 \item $f_{r}(\a)\cap(\d\times\o_1)\subseteq f_{r^*}(\bar\a)\cap(\d\times\o_1)$, and
\item $r^*\restriction \m(\a)$ forces every two distinct nodes in $$(f_{r}(\bar\a)\cap (\d\times\o_1))\cup  f_r(\a)$$ to be incomparable in $\lusim{T}_{\m(\a)}$.
 \end{enumerate}
 \end{enumerate}

The extension relation on $\MQB_\b$ is defined in the following way:

Given $q_1, q_0 \in \MQB_\beta$, $q_1 \leq_{\MQB_{\beta}} q_0$ ($q_1$ is an extension of $q_0$) if and only if  the following holds.
\begin{enumerate}
\item $\dom(f_{q_0}) \subseteq \dom(f_{q_1})$

\item For every $\a\in \dom(f_{q_0})$, $f_{q_0}(\a)\subseteq f_{q_1}(\a)$.

\item For every $\langle (N_0,  \g_0), (N_1, \g_1)\rangle \in \tau_{q_0}$ there are $\g_0'\geq\g_0$ and $\g_1' \geq \g_1$ such that $\langle (N_0,  \g'_0), (N_1, \g'_1)\rangle\in \tau_{q_1}$.
\end{enumerate}

Finally, if $\b>0$, then $\Phi_\b$ is a subset of $H(\k^+)$ canonically coding $\langle \Phi_\a \mid \a<\b\rangle$, $\langle \MQB_\a\mid \a\leq\b\rangle$ and $\langle \Vdash_{\MQB_\a}^{H(\k^+)} \mid \a\leq\b\rangle$, where for each $\a\leq\b$, $\Vdash_{\MQB_\a}^{H(\k^+)}$ denotes the forcing relation restricted to formulas with $\MQB_\a$-names in $H(\k^+)$ as parameters.

We may, and will, assume that the definition of $\langle \Phi_\b\mid 0<\beta<\k^+\rangle$ is uniform in $\b$.

\begin{remark} \normalfont
Having fixed the sequence $\vec{\Phi}=\langle \Phi_\b\mid 0<\beta<\k^+\rangle$ of predicates as above, by an edge we always mean a $\vec{\Phi}$-edge, and similarly for other concepts.
\end{remark}

\begin{remark} \normalfont
Given any $\a < \k^+$, there is a natural map
$$\pi_\a: \MQB_\a \to \MPB^{\textsf{e}}_{\vec{\Phi}\restriction \a, \a},$$
defined by $\pi_\a(q)=\tau_q$. However, $\pi_\a$ is not necessarily
a projection of forcing notions, as given a condition $q \in \MQB_\a$ there might exist
$\tau_q \subseteq \tau \in \MPB^{\textsf{e}}_{\vec{\Phi}\restriction \a, \a}$ such that $\tau_{q'} \not\leq_{\MPB^{\textsf{e}}_{\vec{\Phi}\restriction \a, \a}} \tau$ for all
${q'} \leq q$.
\end{remark}

\begin{remark}\label{clarifying} \normalfont
Despite possible first impressions due to the presence of clause (7), our definition of $\MQB_{\a+1}$-condition, for a given $\a<\k^+$, is not circular. Rather, the definition of `$q$ is a $\MQB_{\a+1}$-condition' is to be seen, because of that clause, as  being by recursion on the supremum of the set of heights of models $N_0$ occurring in edges $\langle (N_0,\g_0), (N_1,\g_1)\rangle$ in $\tau_q$. Indeed, given any $q$ satisfying clauses (1)--(6), in order to verify whether or not $q$ satisfies also (7) we check whether for each generalized $(\vec{\Phi}\upharpoonright \a+1)$-edge $e=\langle (N_0, \a+1), (N_1, \g_1)\rangle$  coming from  $\tau_q\restriction\a+1$ it is the case that some condition holds depending only on $\langle\MQB_\b\mid\b\leq\a\rangle$, $e$, and $\MQB^{\d_{N_0}}_{\a+1}$, which consists of conditions $q'$ with $\d_{N_0'}<\d_{N_0}$ for every $\langle (N_0', \g_0'), (N_1', \g_1')\rangle\in \tau_{q'}$.
\end{remark}

Before moving on to the next subsection, we will briefly address the need for, and nature of, clause (7) in our definition of condition. As already mentioned in the introduction, the proof that our forcing satisfies the $\kappa$-chain condition is an adaptation, in our present context, of the Laver-Shelah proof that their forcing in \cite{laver-shelah} has the $\kappa$-chain condition. The only potential obstacles to making such an adaptation work may come from our present requirements that a condition $q$ be closed under copying of all the relevant information, as dictated by the presence of edges $\langle (N_0, \g_0), (N_1, \g_1)\rangle$ in its side condition $\tau_q$, and where this includes the information coming from the working part $f_q$.

In the proof of the $\k$-c.c.\ of $\MQB_\b$, given $A\subseteq\MQB_\b$ such that $|A|=\k$, we need to find two distinct conditions in $A$ which are compatible. As we said, we would like to do that following the ideas in the $\k$-c.c.\ proof from \cite{laver-shelah} as closely as possible. Now, due to technical reasons coming from the present copying requirements, in order to do this we seem to need to work under the assumption that all conditions in $A$ have an additional property, namely that they are what we call \emph{adequate conditions} (s.\ Definition \ref{adequate-condition}). One of the requirements for $q$ to be an adequate condition is that it be suitably closed under copying not only via edges from $\tau_q$, but also via the corresponding anti-edges. In particular, if $\langle (N_0, \g_0), (N_1, \g_1)\rangle$ is an edge in $\tau_q$, then not only are we to copy the information from the working part sitting in $N_0$ into $N_1$ (into the past) but also to copy the information sitting in $N_1$ into $N_0$ (into the future); and similarly for the edges in $N_1$ with markers at most $\a$ for $\a\in\dom(f_q)\cap N_1\cap\g_1$ such that $f_q(\a)\cap N_1\neq\emptyset$ and $\Psi_{N_1, N_0}(\a)<\g_0$.


Now, the main obstacle for proving that the set of adequate conditions is dense---and this is the motivation behind clause (7)---is the following: take the situation in which there is an edge $\langle (N_0, \g_0), (N_1, \g_1)\rangle\in\tau_q$ with $\a\in N_0\cap\g_0$, $\bar\a:=\Psi_{N_0, N_1}(\a)<\g_1$, some $x\in f_q(\bar\a)$ of height less than $\d_{N_1}$ ($=\d_{N_0}$), and some $y\in f_q(\a)$ of height at least $\d_{N_0}$. If $q'$ were to be any adequate condition extending $q$, it would have to be the case that $x\in f_{q'}(\a)$. However, unless we have an extra clause preventing it, it could for example be that $y$ is forced to be above $x$ in  $\lusim{T}_{\m(\a)}$, which would make it impossible for such a $q'$ to exist.\footnote{The problematic configuration can of course be described in slightly more general terms.}

Our way around this difficulty is to incorporate, in our definition, a clause which stipulates that the above operation can be carried out. This is in essence what clause (7) says.\footnote{A more naive (and simpler-looking) approach would be to require that if $\langle (N_0, \g_0), (N_1, \g_1)\rangle$ is an edge from $\tau_q$, then $\MQB_{\a+1}\cap N_0$ is a complete suborder of $\MQB_{\a+1}$. This would have the intended effect. However, such a condition cannot be expressed without circularity.} Fortunately, the intended content can indeed be expressed (cf.\ the previous footnote); our device for doing so is to phrase this content by reference to a well-defined suborder $\MQB_{\a+1}^\d$ of $\MQB_{\a+1}$---namely the set of conditions $q\in \MQB_{\a+1}$ all of whose edges of form $\langle (N_0, \a+1), (N_1, \g_1)\rangle$ are such that $\d_{N_0}<\d$ (but allowing all nodes in $f_q(\alpha)$ to be of any height below $\k$). Hence, due to the presence of this clause (7), the definition of $q$ being a $\MQB_{\a+1}$-condition is ultimately to be seen as being given by recursion on the supremum of the collection of heights of models occurring in edges of the form $\langle (N_0, \a+1), (N_1, \g_1)\rangle$ (i.e., those edges not coming from the restriction of $q$ to $\a$).\footnote{Let us reconsider for a second the situation described a few lines earlier. Suppose $q\in\MQB_{\a+1}$,  $\langle (N_0, \g_0), (N_1, \g_1)\rangle\in\tau_q$, and $\a$ and $\bar\a$ are as in that description. Suppose $x\in f_q(\bar\a)$ is of height less than $\d_{N_1}$ and $y$ is a node of height at least $\d_{N_0}$ such that, say, $q\restriction \m(\a)$ happens to force $y$ to be above $x$ in $\lusim{T}_{\mu(\a)}$. It is then of course impossible to extend $q$ to a condition $q'$ such that $x\in f_{q'}(\a)$. However, we can certainly pick $\a'$ such that $\m(\a')=\m(\a)$ and such that $q$ can be extended (trivially) by making $f_{q'}(\a')=\{x\}$. This will ensure that the generic specializing function  for $\lusim{T}_{\m(\a)}$ will be defined everywhere (cf.\ the proof of Lemma
\ref{satp}).}

\subsection{Basic properties of $\langle\MQB_\b\,\mid\,\b\leq\k^{+}\rangle$}\label{basic-properties}

Our first lemma follows immediately from the choice of the predicates $\Phi_\a$.

\begin{lemma}\label{definability} For every nonzero $\a<\k^+$, $\MQB_\a$ and $ \Vdash_{\MQB_\a}^{H(\k^+)}$ are definable over the structure $$(H(\k^+), \in, \Phi_{\a})$$ without parameters. Moreover, this definition is uniform in $\a$.
\end{lemma}

Our next lemma follows from the fact that $\MQB_1$ is essentially the L\'{e}vy collapse turning $\k$ into $\o_2$.

 \begin{lemma}\label{levy} $\MQB_1$ forces $\kappa=\aleph_2$.
 \end{lemma}

The following lemma is also an easy consequence of the definition of condition.

\begin{lemma}\label{limit} For every $\b\leq \k^{+}$, $\MQB_\a\subseteq\MQB_\b$ for all $\a<\b$. \end{lemma}

 Lemma \ref{compl} follows easily from the definition of $\langle\MQB_\a\,\mid\,\a\leq\k^+\rangle$.

 \begin{lemma}\label{compl}
 For all $\a<\b\leq\k^{+}$, $q\in \MQB_\b$, and $r\in\MQB_\a$, if $r\leq_{\MQB_\a} q\restriction \a$, then $$(f_r\cup f_q\restriction [\a,\,\b), \tau_q\cup\tau_r)$$ is a common extension of $q$ and $r$ in $\MQB_\b$. \end{lemma}
\begin{proof}
Let $p=(f_r\cup f_q\restriction [\a,\,\b), \tau_q\cup\tau_r).$ We show that $p$ satisfies items (1)--(7) of the definition of a $\MQB_\b$-condition. It suffices to consider (5) and (6), as all other clauses can be proved easily.

We first show that  $p$ satisfies clause (5). Thus let $$e=\langle (N_0, \g_0), (N_1, \g_1)\rangle\in \tau_q\cup\tau_r$$ and $$e'=\langle (N'_0, \g'_0), (N'_1, \g'_1)\rangle \in (\tau_q \cup \tau_r)\cap N_0.$$
We have to show that there are ordinals $\g_0^*\geq \p^{\g_0', N_0'}_{N_0, \g_0, N_1, \g_1}$ and $\g_1^*\geq \p^{\g_1', N_1'}_{N_0, \g_0, N_1, \g_1}$ such that $$\langle (\Psi_{N_0, N_1}(N'_0), \g^*_0), (\Psi_{N_0, N_1}(N'_1), \g_1^*)\rangle \in\tau_q \cup \tau_r.$$
We divide the proof into three cases:
\begin{enumerate}
\item Both $e$ and $e'$ belong to $\tau_q$ (resp.\ $\tau_r$). Then the conclusion is immediate as $q$ (resp.\ $r$) is a condition in $\MQB_\b$.

\item $e \in \tau_q$ and $e' \in \tau_r$. Then  $e \restriction \a \in \tau_{q \upharpoonright \alpha},$ so as $r \leq_{\MQB_\a} q \restriction \a,$ we can find $\g_0^{''}, \g_1^{''} \leq \a$ such that  $\g_0^{''}\geq \min\{\g_0, \a\}_{N_0}$, $\g_1^{''}\geq \min\{\g_1, \a\}_{N_1}$
and
 $\langle (N_0, \g_0^{''}), (N_1, \g_1^{''})          \rangle \in \tau_r$. Hence, as $r$ is a condition, for some
$$\g_0^*\geq \p^{\g_0', N_0'}_{N_0, \g_0^{''}, N_1, \g_1^{''}}$$ and $$\g_1^*\geq \p^{\g_1', N_1'}_{N_0, \g_0^{''}, N_1, \g_1^{''}}$$
 we have
  $$\langle (\Psi_{N_0, N_1}(N'_0), \g^{*}_0), (\Psi_{N_0, N_1}(N'_1), \g_1^{*})\rangle \in\tau_r.$$
As $r \in \MQB_\a$, $e' \in \tau_r$, and $\g_0', \g_1' \leq \a$, and by the choice of $\g_0'', \g_1''$, one can easily show that
$$\p^{\g_0', N_0'}_{N_0, \g_0^{''}, N_1, \g_1^{''}} \geq \p^{\g_0', N_0'}_{N_0, \g_0, N_1, \g_1},$$
and
$$\p^{\g_1', N_1'}_{N_0, \g_0^{''}, N_1, \g_1^{''}} \geq \p^{\g_1', N_1'}_{N_0, \g_0, N_1, \g_1}.$$
Thus $\g_0^* \geq \p^{\g_0', N_0'}_{N_0, \g_0, N_1, \g_1}$ and $\g_1^* \geq \p^{\g_1', N_1'}_{N_0, \g_0, N_1, \g_1}$,
from which  the result follows.

\item $e \in \tau_r$ and $e' \in \tau_q$. Then $e' \restriction \a \in \tau_{q \restriction \a}$ and so 
we can find $\g_0^{''}, \g_1^{''} \leq \a$ such that  $\g_0^{''}\geq \min\{\g'_0, \a\}_{N'_0}$, $\g_1^{''}\geq \min\{\g'_1, \a\}_{N'_1}$
and
 $\langle (N'_0, \g_0^{''}), (N'_1, \g_1^{''})          \rangle \in \tau_r$.
 By the discussion after Definition \ref{defmarker}, $\langle (N'_0, \g_0^{''}), (N'_1, \g_1^{''})          \rangle \in N_0$. Hence, as $r$ is a condition, we can find
$$\g_0^*\geq \p^{\g_0^{''}, N_0'}_{N_0, \g_0, N_1, \g_1}$$ and
$$\g_1^*\geq \p^{\g_1^{''}, N_1'}_{N_0, \g_0, N_1, \g_1}$$ such that $$\langle (\Psi_{N_0, N_1}(N'_0), \g^*_0), (\Psi_{N_0, N_1}(N'_1), \g_1^*)\rangle \in\tau_r.$$
As $r \in \MQB_\a$, $e \in \tau_r$, and $\g_0, \g_1 \leq \a$, and by the choice of $\g_0^{''}, \g_0^{''}$, one can again easily show that
$$\p^{\g_0^{''}, N_0'}_{N_0, \g_0, N_1, \g_1} \geq \p^{\g_0', N_0'}_{N_0, \g_0, N_1, \g_1}$$
and
$$\p^{\g_1^{''}, N_1'}_{N_0, \g_0, N_1, \g_1} \geq \p^{\g_1', N_1'}_{N_0, \g_0, N_1, \g_1},$$
from which we get that $\g_0^* \geq \p^{\g_0', N_0'}_{N_0, \g_0, N_1, \g_1}$ and $\g_1^* \geq \p^{\g_1', N_1'}_{N_0, \g_0, N_1, \g_1}$, and the result follows.
\end{enumerate}
To show that $p$
satisfies clause (6), let $e=\langle (N_0, \g_0), (N_1, \g_1)\rangle\in \tau_q\cup \tau_r$, $\eta \in (\dom(f_r) \cup \dom(f_q))\cap N_0\cap\g_0$, $\eta \neq 0$, and $x \in (f_r(\eta) \cup f_q(\eta))\cap N_0$. We have to show that $\bar \eta \in \dom(f_r) \cup \dom(f_q)$ and $x \in f_r(\bar\eta) \cup f_q(\bar\eta)$, where $\bar\eta=\Psi_{N_0, N_1}(\eta)$.

If $\eta < \a,$ then $\eta \in \dom(f_r)$ and  $x \in f_r(\eta)$.
As $r \leq_{\a} q\restriction\a$, for some $\g_0', \g_1'$ we have $\langle (N_0, \g'_0), (N_1, \g'_1)\rangle\in \tau_r$ (if $e \in \tau_r$, we can take $\g_0'= \g_0$ and $\g_1'=\g_1$, otherwise, we can take $\g'_0 \geq \min\{\g_0, \a   \}_{N_0}$ and  $\g'_1 \geq \min\{\g_1, \a   \}_{N_1}$).
But then $\bar \eta \in \dom(f_r)$ and $x \in f_r(\bar\eta)$.

 Next suppose that $\eta \geq \a.$
In this case we must have $e \in \tau_q$ and $\eta \in \dom(f_q) \setminus \dom(f_r)$. But then $\bar\eta \in \dom(f_q)$ and $x \in f_q(\bar \eta)$.
\end{proof}

Throughout the paper, we write $\MPB\lessdot\MQB$ to denote that $\MPB$ is a complete suborder of $\MQB$ (i.e., $\MPB$ is a suborder of $\MQB$ and maximal antichains in $\MPB$ are also maximal antichains in $\MQB$).

 The following corollary is a trivial consequence of Lemma \ref{compl}.

 \begin{corollary}
  \label{isiteration} $\langle\MQB_\a\,\mid\,\a\leq\k^+\rangle$ is a forcing iteration, in the sense that $\MQB_\a\lessdot\MQB_\b$ for all $\a<\b\leq\k^+$.\end{corollary}

 We say that a partial order $\mtcl P$ is $\sigma$-closed if every descending sequence $(p_n)_{n<\o}$ of $\mtcl P$-conditions has a lower bound in $\mtcl P$.
 
\begin{remark} \normalfont
Suppose $\b \leq \k^+$ and  $\langle \tau_n: n<\omega \rangle$ is a $\leq_{\MPB^{\textsf{e}}_{\vec{\Phi}, \beta}}$-decreasing sequence of
conditions in $\MPB^{\textsf{e}}_{\vec{\Phi}, \beta}$ which are closed under copying. Then $\bigcup_{n<\o}\tau_n$ is also a condition
in $\MPB^{\textsf{e}}_{\vec{\Phi}, \beta}$ and is closed under copying. The reason that $\bigcup_{n<\o}\tau_{n}$ is closed under copying is that if $n < m$  and we have
$$e=\langle (N_0, \g_0), (N_1, \g_1)\rangle \in \tau_{m}$$
and
$$e'=\langle (N'_0, \g'_0), (N'_1, \g'_1)\rangle  \in \tau_{n} \cap N_0,$$ then for some $\g''_0, \g''_1$ we have
$e''=\langle (N'_0, \g''_0), (N'_1, \g''_1)\rangle  \in \tau_{m}$, and by the discussion after Definition \ref{defmarker}, $e'' \in N_0$, so as in the proof of Lemma \ref{compl},
for some $\g_0^* \geq \p^{\g_0', N_0'}_{N_0, \g_0, N_1, \g_1}$ and $\g_1^* \geq \p^{\g_1', N_1'}_{N_0, \g_0, N_1, \g_1}$ we have
$$\langle (\Psi_{N_0, N_1}(N'_0), \g^*_0), (\Psi_{N_0, N_1}(N'_1), \g_1^*)\rangle \in\tau_{m}.$$
\end{remark}

\begin{lemma}\label{countable closure}  $\MQB_\b$ is $\sigma$-closed for every $\b\leq\k^{+}$. In fact, every decreasing $\o$-sequence of $\MQB_\b$-conditions has a greatest lower bound in $\MQB_\beta$. In particular,
forcing with $\MQB_{\beta}$ does not add new $\o$-sequences of ordinals, and therefore this forcing preserves both $\omega_1$ and $\CH$. \end{lemma}

\begin{proof}
Given a decreasing sequence $(q_n)_{n<\o}$ of $\MQB_\b$-conditions, it is immediate to check that $q=(f, \bigcup_{n<\o}\tau_{q_n})$ is the greatest lower bound of the set $\{q_n\mid n<\o\}$, where $\dom(f)=\bigcup_{n<\o}\dom(f_{q_n})$ and, for each $n<\o$ and $\a\in \dom(f_{q_n})$, $f(\a)=\bigcup\{f_{q_m}(\a) \mid m\geq n\}$. For this one proves, by induction on $\a$, that $q\restriction\a\in\MQB_\a$ for every $\a\leq\b$.
\end{proof}

\begin{remark} \normalfont
Lemma \ref{countable closure}, or rather its proof, will be used, often without mention, in several places in which we run some construction, in $\o$ steps, along which we build some decreasing sequence $(q_n)_{n<\o}$ of conditions. At the end of such a construction we will have that the ordered pair $q=(f, \bigcup_{n<\o}\tau_{q_n})$, where $f$ is given as in the above proof, is the greatest lower bound of $(q_n)_{n<\o}$.
\end{remark}

Given $\a \in \mtcl X$, a node $x=(\r, \zeta)$ in $\kappa\times\o_1$, and an ordinal $\bar\r\leq\r$, if $\MQB_\a$ has the $\k$-c.c., we denote by $A^\a_{x, \bar\r}$ the $F$-first maximal antichain of $\MQB_\a$ in $H(\k^+)$ consisting of conditions
 deciding, for some ordinal $\bar\zeta <\o_1$, that $(\bar\r, \bar\zeta)$ is $\lusim{T}_\a$-below $x$.\footnote{In Lemma \ref{cc} we will prove that each $\MQB_\a$ has the $\k$-c.c. Hence, $A^\a_{x, \bar\r}$ will be defined for all $x$ and $\bar \r$.}
If $x_0=(\r_0, \zeta_0)$ and $x_1=(\r_1, \zeta_1)$ are nodes, $\bar\r\leq\r_0$, $\r_1$, $r_0\in A^\a_{x_0, \bar\r}$, $r_1\in A^\a_{x_1, \bar\r}$, and there are ordinals $\bar\zeta_0\neq\bar\zeta_1$ in $\o_1$ such that $r_0$ forces that $(\bar\r, \bar\zeta_0)$ is $\lusim{T}_\a$-below $x_0$ and $r_1$ forces that $(\bar\r, \bar\zeta_1)$ is $\lusim{T}_\a$-below $x_1$, then we say that \emph{$r_0$ and $r_1$ force $x_0$ and $x_1$ to be incomparable in $\lusim{T}_\a$}.\footnote{This terminology is apt: since for each $\r<\k$, $\{\r\}\times\o_1$ is forced to be the $\r$-th level of $\lusim{T}_{\a}$, we have that every condition in $\MQB_\a$ extending both of $r_0$ and $r_1$ must force that $x_0$ and $x_1$ are incomparable nodes in $\lusim{T}_\a$.}

The following lemma will be often used.

\begin{lemma}\label{transfer} Suppose $q$ is a $\MQB_{\k^+}$-condition, $\a\in \dom(f_q)$, $\a\neq 0$, and $\lusim{T}_{\m(\a)}$ is defined. Suppose $\MQB_{\m(\a)}$ has the $\k$-c.c. Suppose $\langle\vec{\mtcl E}, (\r, \a)\rangle$ is a correct $\tau_q$-thread, where $\r<\k$, and $x_0=(\rho_0, \zeta_0)$ and $x_1=(\rho_1, \zeta_1)$ are two nodes such that
\begin{itemize}
\item $\rho_0$, $\rho_1 \leq \r$, and
\item there is some $\bar\r\leq\r_0$, $\r_1$ such that $q \restriction \m(\a)$ extends conditions $r_0\in A^{\m(\a)}_{x_0,  \bar\r}$ and $r_1\in A^{\m(\a)}_{x_1, \bar\r}$ forcing $x_0$ and $x_1$ to be incomparable in $\lusim{T}_{\m(\a)}$.
\end{itemize}
 Let $\bar\a=\Psi_{\vec{\mtcl E}}(\a)$. Then
 \begin{enumerate}
\item $r_0$ and $r_1$ are in $\dom(\Psi_{\vec{\mtcl E}})$, and
\item $\Psi_{\vec{\mtcl E}}(r_0)$ and $\Psi_{\vec{\mtcl E}}(r_1)$ force $x_0$ and $x_1$ to be incomparable in $\lusim{T}_{\m(\bar\a)}$.
\end{enumerate}
\end{lemma}

\begin{proof}
Let $$\vec{\mtcl E}=(\langle (N^i_0, \g^i_0), (N^i_1, \g^i_1)\rangle\,\mid\, i\leq n)$$ Since $\o_1\cup \r+1\subseteq \dom(\Psi_{\vec{\mtcl E}})$ and $\Psi_{\vec{\mtcl E}}$ is a partially defined elementary embedding from $(N^0_0, \in, \Phi_\a)$ into $(N^n_1, \in, \Phi_{\bar\a})$ we have, by definability of $A^{\m(\a)}_{x_0,  \bar\r}$ and $A^{\m(\a)}_{x_1,  \bar\r}$ over the structure $(H(\k^+), \in, \Phi_\a)$ by formulas $\varphi_0$ and $\varphi_1$, respectively, with $x_0$, $x_1$ and $\bar\r$
as parameters, and defi\-na\-bi\-lity of $A^{\m(\bar\a)}_{x_0,  \bar\r}$ and $A^{\m(\bar\a)}_{x_1,  \bar\r}$ over $(H(\k^+), \in, \Phi_{\bar\a})$, also by $\varphi_0$ and $\varphi_1$, respectively, that

\begin{itemize}
\item
$A^{\m(\a)}_{x_0,  \bar\r}$, $A^{\m(\a)}_{x_1, \bar\r} \in\dom(\Psi_{\vec{\mtcl E}})$,
\item $A^{\m(\bar\a)}_{x_0, \bar\r}=\Psi_{\vec{\mtcl E}}(A^{\m(\a)}_{x_0,  \bar\r})$, and
\item $A^{\m(\bar\a)}_{x_1, \bar\r}=\Psi_{\vec{\mtcl E}}(A^{\m(\a)}_{x_1,  \bar\r})$.
\end{itemize}

Again by a definability argument, since $|A^{\m(\a)}_{x_0, \bar\r}|$, $|A^{\m(\a)}_{x_1, \bar\r}|<\kappa$, we also have that  $A^{\m(\a)}_{x_0, \bar\r}$ and $A^{\m(\a)}_{x_1, \bar\r}$ are both subsets of $\dom(\Psi_{\vec{\mtcl E}})$. Finally, we have $\bar\zeta_0\neq\bar\zeta_1$ in $\o_1$ such that
\begin{itemize}
\item $r_0$ forces in $\MQB_{\m(\a)}$ that $(\bar\r, \bar\zeta_0)$ is below $x_0$ in $\lusim{T}_{\m(\a)}$ and
\item $r_1$ forces in $\MQB_{\m(\a)}$ that $(\bar\r, \bar\zeta_1)$ is below $x_1$ in $\lusim{T}_{\m(\a)}$.
\end{itemize}

\noindent But since $\Psi_{\vec{\mtcl E}}$ is a partial elementary embedding  from $(N^0_0, \in, \Phi_\a)$ into $(N^n_1, \in,  \Phi_{\bar\a})$, by Lemma \ref{definability} we have that
\begin{itemize}
\item $\Psi_{\vec{\mtcl E}}(r_0)$  forces in $\MQB_{\m(\bar\a)}$ that $(\bar\r, \bar\zeta_0)$ is below $x_0$ in $\lusim{T}_{\m(\bar\a)}$ and that
\item $\Psi_{\vec{\mtcl E}}(r_1)$ forces in $\MQB_{\m(\bar\a)}$ that $(\bar\r, \bar\zeta_1)$ is below $x_1$ in $\lusim{T}_{\m(\bar\a)}$.
\end{itemize}
\end{proof}

Given functions $f$ and $g$, let us momentarily denote by $f + g$ the function with $\dom(f + g)=\dom(f)\cup \dom(g)$ defined by letting $$(f+g)(x)=f(x)\cup g(x)$$ for all $x\in\dom(f)\cup\dom(g)$.\footnote{Where, we recall, if $h$ is a function and $x\notin \dom(h)$, we are setting $h(x)$ to be $\emptyset$.}

Given $\MQB_{\k^+}$-conditions $q_0$ and $q_1$, let $q_0\oplus q_1$ denote the natural amalgamation of $q_0$ and $q_1$; to be more specific, $q_0\oplus q_1$ is the ordered pair $(f, \tau_{q_0}\oplus \tau_{q_1})$, where $f$ is the closure of $f_{q_0}+f_{q_1}$ with respect to relevant (restrictions of) functions of the form $\Psi_{N_0, N_1}$, for
edges $\langle (N_0, \g_0), (N_1, \g_1)\rangle$ in $\tau_{q_0}\oplus\tau_{q_1}$, so that clause (6) in the definition of condition holds for $q_0\oplus q_1$. Even more precisely, we define $$q_0\oplus q_1=((f_{q_0}+f_{q_1})+f, \tau_{q_0}\oplus\tau_{q_1}),$$ where $f$ is the function with domain $X$---for
$X$ being the collection of all ordinals of the form $\Psi_{\vec{\mtcl E}}(\a)$, for a
connected $\tau_{q_0}\oplus\tau_{q_1}$-thread $\langle \vec{\mtcl E}, (\r, \a)\rangle$ such that $\vec{\mtcl E}$ consists of edges, and such that $(\r, \zeta)\in f_{q_0}(\a)\cup f_{q_1}(\a)$ for some $\zeta<\o_1$---and such that for every $\bar\a\in X$, $f(\bar\a)$ is the collection of all nodes $(\r, \zeta)$, for
connected $\tau_{q_0}\oplus\tau_{q_1}$-threads $\langle \vec{\mtcl E}, (\r, \a)\rangle$ such that
\begin{enumerate}
\item $\vec{\mtcl E}$ consists of edges,
\item $(\r, \zeta)\in f_{q_0}(\a)\cup f_{q_1}(\a)$, and
\item $\Psi_{\vec{\mtcl E}}(\a)=\bar\a$.
\end{enumerate}

Lemma \ref{closure-sym} holds by the construction of $q_0\oplus q_1$.

\begin{lemma}\label{closure-sym} Let $q_0$ and $q_1$ be $\MQB_{\k^+}$-conditions and let $q=q_0\oplus q_1$. Then the following holds.

\begin{enumerate}
\item $\tau_q$ is closed under copying.

\item For every edge $\langle (N_0, \g_0), (N_1, \g_1)\rangle\in \tau_q$ and every $\a\in N_0\cap\g_0$ such that $\bar\a:=\Psi_{N_0, N_1}(\a)<\g_1$, if $\a\neq 0$, $\a\in\dom(f_q)$, and $x\in f_q(\a)\cap N_0$, then
\begin{enumerate}
\item $\bar\a\in\dom(f_q)$, and
\item $x \in f_q(\bar\a)$.
 \end{enumerate}
\end{enumerate}
\end{lemma}

The following lemma is a trivial consequence of Lemmas \ref{high-amalg} and \ref{orbits}.

\begin{lemma}\label{simplifying-amalg} For every two $\MQB_{\k^+}$-conditions $q_0$ and $q_1$, if $q_0\oplus q_1=(f, \tau)$, then for every $\a\in\dom(f)$ and every $x=(\r, \zeta)\in f(\a)$ such that $x\notin f_{q_0}(\a)\cup f_{q_1}(\a)$ there is some $\a^*\in \dom(f_{q_0})\cup\dom(f_{q_1})$ such that $x\in f_{q_0}(\a^*)\cup f_{q_1}(\a^*)$ and some connected $\tau_{q_0}\cup\tau_{q_1}$-thread $\langle\vec{\mtcl E}, (\r, \a^*)\rangle$ such that $\Psi_{\vec{\mtcl E}}(\a^*)=\a$. Furthermore, if $\l<\k$ is such that all edges in $\tau_{q_0}$ involve models of height less than $\l$, then all members of $\vec{\mtcl E}$ involving models of height at least $\l$ are edges in $\tau_{q_1}$.
\end{lemma}

Extending our notation $f+g$ for functions $f$, $g$, if $\mtcl F$ is a set of functions, we denote by $\bigoplus\mtcl F$ the function $g$ with domain $$\bigcup\{\dom(f)\mid f\in\mtcl F\}$$ given by $$g(x)=\bigcup\{f(x)\mid x\in\dom(f)\}.$$

The following  lemma will be used in the proof of Lemma \ref{chain condition theorem}.

\begin{lemma}
\label{pre-separating-pair-compatibility}
Let $\b\leq\k^+$, and suppose $q_0$, $q_1 \in \MQB_{\b}$ are such that for every $\alpha<\b$, if $$(q_0\restriction\a)\oplus (q_1\restriction\a)\in\MQB_{\a},$$ then $$(q_0\restriction\a+1)\oplus (q_1\restriction\a+1)\in \MQB_{\a+1}.$$ Then $q_0\oplus q_1\in \MQB_{\b}$.
\end{lemma}

\begin{proof}
The proof is by induction on $\b$. We only need to argue for the conclusion in the case that $\b$ is a nonzero limit ordinal. In that case the conclusion follows easily from the induction hypothesis and the fact that for every $\a<\b$, $$f_{q_0\oplus q_1}\restriction\a=\bigoplus\{f_{(q_0\restriction\a')\oplus (q_1\restriction\a')}\restriction\a\mid \a\leq\a'<\b\}.$$ To see this equality it suffices to note that for every $\bar\a\in \dom(f_{q_0\oplus q_1})$,  any given $(\rho, \zeta)$ in $$f_{q_0\oplus q_1}(\bar\a)\setminus (f_{q_0}(\bar\a)\cup f_{q_1}(\bar\a))$$ has arrived there, thanks to Lemma \ref{simplifying-amalg}, by virtue of some connected $\tau_{q_0\restriction\a'}\cup\tau_{q_1\restriction\a'}$-thread $\langle\vec{\mtcl E}, (\rho, \a^*)\rangle$ for some high enough $\a'<\b$.

If $q_0\oplus q_1$ were not a $\MQB_\b$-condition, there would be some $\a<\b$ such that some finite piece of information contained in $f_{q_0\oplus q_1}\restriction\a$ fails to satisfy clause (4), (6) or (7)  in the definition of $\MQB_\a$-condition. But that piece of information would occur in $f_{(q_0\restriction\a')\oplus (q_1\restriction\a')}\restriction\a$ for a high enough $\a'<\b$ above $\a$. Hence, by taking $\a'$ high enough we may guarantee that the fact that the piece of information violates some clause in the definition of $\MQB_\a$-condition entails that $$(f_{(q_0\restriction\a')\oplus (q_1\restriction\a')}\restriction\a, \tau_{(q_0\restriction\a')\oplus (q_1\restriction\a')}\restriction\a)$$ is not a $\MQB_\a$-condition. But that contradicts $(q_0\restriction\a')\oplus (q_1\restriction\a')\in\MQB_{\a'}$, which we know is true by induction hypothesis.
\end{proof}
We will now introduce the notion of adequate condition, which we already alluded to at the beginning of this section.



 \begin{definition}\label{adequate-condition}
Given $\b<\k^+$ and $q\in \MQB_{\b}$, we will say that $q$ is \emph{adequate} in case (1) and (2) below hold.

\begin{enumerate}
\item For every $\a\in\dom(f_q)$ such that $\MQB_{\m(\a)}$ has the $\k$-c.c.\ and $\lusim{T}_{\m(\a)}$ is defined and for all distinct $x_0=(\r_0, \zeta_0)$, $x_1=(\r_1, \zeta_1)\in f_q(\a)$ there is some $\bar\r\leq\r_0$, $\r_1$, together with conditions $r_0\in A^{\m(\a)}_{x_0, \bar\r}$ and $r_1\in A^{\m(\a)}_{x_1, \bar\r}$ weaker than $q\restriction\m(\a)$ and such that $r_0$ and $r_1$ force $x_0$ and $x_1$ to be incomparable in $\lusim{T}_{\m(\a)}$.
\item The following holds for every correct $\tau_q$-thread $\langle \vec{\mtcl E}, (\r, \bar\alpha)\rangle$, where
$\a=\Psi_{\vec{\mtcl E}}(\bar\a)$, $\rho<\k$, $\zeta<\o_1$, and $x=(\r, \zeta)\in f_q(\bar\a)$.

\begin{enumerate}
\item $\a\in\dom(f_q)$ and $x \in f_q(\a)$.

 \item For every $\langle (N_0, \g_0), (N_1, \g_1)\rangle\in\tau_q\cap \dom(\vec{\mtcl E})$ with $\g_0$, $\g_1\leq\bar\a$ there are $\g_0'\geq\Psi_{\vec{\mtcl E}}(\g_0)$ and $\g_1'\geq\Psi_{\vec{\mtcl E}}(\g_1)$ such that $$\langle (\Psi_{\vec{\mtcl E}}(N_0), \g_0'), (\Psi_{\vec{\mtcl E}}(N_1), \g_1')\rangle\in\tau_q$$
  \end{enumerate}
    \end{enumerate}
 \end{definition}


 We call a condition \emph{weakly adequate} if it satisfies clause (1) from Definition \ref{adequate-condition}.

 \begin{lemma}\label{weak-adeq}
 The set of weakly adequate conditions is dense in $\MQB_\b$ for each $\b\leq\k^+$.
  \end{lemma}
  \begin{proof}
  For each condition $q \in \MQB_\b$ let $\langle (\a^q_n, x^q_{0, n}, x^q_{1, n}): n < \omega            \rangle$ be the $F$-least enumeration of all triples $(\a, x_0, x_1)$ such that $\a \in \dom(f_{q})$, $\a\neq 0$, is such that $\MQB_{\m(\a)}$ has the $\k$-c.c.\ and $\lusim{T}_{\m(\a)}$ is defined, and  $x_0=(\r_0, \zeta_0)$, $x_1=(\r_1, \zeta_1)\in f_{q_i}(\a)$ are distinct.
  Let also $\varphi: \omega \times \omega \to \omega$ be a bijection such that $\varphi(m, n) \geq m$ for all $m, n < \omega$.

  By induction on $i < \omega$ we define a decreasing sequence $\langle q_i: i<\omega \rangle$ of $\MQB_\b$-conditions as follows. To start, set $q_0=q$. Now suppose that $i< \o$ and that $q_{i}$ is defined. Let $m$, $n$ be such that $\varphi(m, n)=i$ and set
  $(\a, x_0, x_1)= (\a^{q_m}_{n}, x^{q_m}_{0, n}, x^{q_m}_{1, n})$.
  Let $q_{i+1}$ be an extension of $q_{i}$ such that there are $\bar\r\leq\r_0$, $\r_1$, together with conditions $r_0\in A^{\m(\a)}_{x_0, \bar\r}$ and $r_1\in A^{\m(\a)}_{x_1, \bar\r}$ weaker than $q_{i+1}\restriction\m(\a)$ and such that $r_0$ and $r_1$ force $x_0$ and $x_1$ to be incomparable in $\lusim{T}_{\m(\a)}$. Then the greatest lower bound of $\{q_i\,\mid\,i<\omega\}$ is weakly adequate.
\end{proof}

In fact, the set of adequate conditions is dense in $\MQB_\b$ for each $\b\leq\k^+$, as shown in Lemma \ref{adequate}.
To show this, we need the following lemma.

\begin{lemma}\label{str-props} Let $\a<\k^+$, $q\in \MQB_{\a+1}$,
 $e=\langle  (N_0, \a+1), (N_1, \g_1)\rangle$ a genera\-lized edge coming from $\tau_{q}$, and $\vec{\mtcl E}=(\langle (N^i_0, \g^i_0), (N^i_1, \g^i_1)\rangle\,\mid\,i\leq n)$ a sequence of generalized edges coming from $\tau_{q}$ with $\langle (N^{0}_0, \g^{0}_0), (N^0_1, \g^0_1)\rangle=e$ and such that $\langle\vec{\mtcl E}, (\emptyset, \a)\rangle$ is a correct thread. Let $\bar\a$ be such that $\Psi_{\vec{\mtcl E}}(\a)=\bar\a$. Let also $\d=\min\{\d_{N^i_0}\mid i\leq n\}$.
 Suppose that
  \begin{enumerate}
  \item $q\restriction \a$ is adequate,
  and that
  \item $q\restriction\m(\a)$ forces every two distinct nodes in $f_q(\bar\a)\cap (\d\times\o_1)$ to be incomparable in $\lusim{T}_{\m(\a)}$.
  \end{enumerate}
 Then there is an extension $q^*\in\MQB_{\a+1}$ of $q$ such that $f_{q}(\bar\a)\cap (\d\times\o_1)\subseteq f_{q^*}(\a)$.
\end{lemma}

\begin{proof}
We may obviously assume $\bar\a \neq \a$ as otherwise there is nothing to prove. Let $r=(f_q, (\tau_q\setminus\tau) \cup (\tau_q \restriction \a) \cup (\tau\restriction \a)),$ where
$$\tau = \{\langle (N_0',  \g_0'), (N_1', \g_1')\rangle \in \tau_q\mid \max\{\g_0', \g_1'\}=\a+1,\, \delta_{N_0'}\geq \delta_{N_0} \}.$$
Then $r \in \MQB_{\a+1}^{\delta_{N_0}}$ and $e\restriction \a$ comes from $\tau_r$. To see the latter claim, note that
 $e$ clearly comes from $\tau$, and hence $e\restriction \a$ comes from $\tau \restriction \a \subseteq \tau_r$.
Also note that $r \restriction \m(\a)=q\restriction\m(\a),$ hence by clause (2), $r\restriction\m(\a)$ forces every two distinct nodes in $f_q(\bar\a)\cap (\d\times\o_1)$ to be incomparable in $\lusim{T}_{\m(\a)}$.

Set
$$\vec{\mtcl E}^{\ast}=\langle e \rangle^{\frown } (\langle (N^i_0, \g^{i}_0), (N^i_1, \g^i_1) \rangle \restriction \a \mid 0< i \leq n).$$
Then
 $(\vec{\mtcl E}^{\ast}, (\emptyset, \a))$ is a correct thread (since $\bar\a<\a$).
Since clearly all members of $\vec{\mtcl E}^{\ast}$ come from $\tau_r\cup\{e\}$, by clause (7) of the definition of $\MQB_{\a+1}$-condition for $q$
there is an extension $r^* \in \MQB_{\a+1}^{\delta_{N_0}}$ of $r$ such that
$$f_{r^*}(\bar\a)\supseteq f_{r}(\a)\cap (\d\times\o_1)=f_{q}(\a)\cap (\d\times\o_1)$$ and such that
 $r^*\restriction \m(\a)$ forces every two distinct nodes in $$(f_q(\bar\a)\cap (\d\times\o_1))\cup f_q(\a)$$ to be incomparable in $\lusim{T}_{\m(\a)}$.
Let
$$q^*=(f_{r^*}\restriction \a \cup\{(\a, f_q(\a)\cup (f_q(\bar\a)\cap (\d\times\o_1)))\}, \tau_{q}\cup\tau_{r^*}\restriction\a).$$
It suffices to show that  $q^*$ is a $\MQB_{\a+1}$-condition, as then it is an extension of $q$ in $\MQB_{\a+1}$ as desired.

To see that $q^* \in \MQB_{\a+1}$, we only need to show that $q^*$ satisfies clause (6) of the definition of $\MQB_{\a+1}$-condition, as all other clauses can be checked easily.
Thus let $e'=\langle (N_0',  \a+1), (N_1', \g_1')\rangle \in \tau_{q}\cup\tau_{r^*}\restriction\a$ be such that
$\alpha' := \Psi_{N_0', N_1'}(\alpha)<\g_1'$ and let us note that in fact $e'\in\tau_q$. We have to show that
\[
\big(f_q(\a)\cup (f_q(\bar\a) \big)\cap (\d_{N_0}\times\o_1) \cap N_0' \subseteq f_{q^*}(\a').
\]
 We may assume that $\a'<\a$.
Since $e'\in\tau_q$, we have that
$f_q(\a) \cap N_0' \subseteq f_{q}(\a') \subseteq f_{q^*}(\a')$.
To show that $f_q(\bar\a)\cap (\d_{N_0}\times\o_1) \cap N_0' \subseteq f_{q^*}(\a')$, let $x=(\rho, \zeta) \in f_q(\bar\a)\cap (\d_{N_0}\times\o_1)$ with $\rho<\d_{N_0'}$. Let $$\vec{\mtcl E}^{-1}=(\langle (N^{n-i}_1, \g^{n-i}_1), (N^{n-i}_0, \g^{n-i}_0)\rangle\mid i\leq n\rangle$$  and let us consider the correct
$\tau_q\restriction \a$-thread $\langle \vec{\mtcl F}, (\rho, \bar\a)    \rangle$, where $$\vec{\mtcl F}=(\vec{\mtcl E}^{-1}\restriction \a)^{\frown} \langle e'\restriction \a \rangle.$$ Then $\Psi_{\vec{\mtcl F}}(\bar\a)=\a'$,
and since $q\restriction \a$ is adequate, we have $x \in f_{q^*}(\a')$. Thus, the conclusion follows.
\end{proof}

\begin{lemma}\label{adequate} For every $\b\leq\k^+$, the set of adequate $\MQB_\b$-conditions is dense in $\MQB_\b$.
\end{lemma}

\begin{proof} Let $q\in\MQB_\b$.
We will find an adequate $\MQB_\b$-condition $q^*$ stronger than $q$. We prove this by induction on $\b$.

First, suppose that $\b$ is a limit ordinal of countable cofinality and let $(\b_i)_{i<\omega}$ be an increasing sequence of ordinals cofinal in $\b.$
We define, by induction on $i<\omega$, two sequences $(q_i)_{i<\omega}$ and $(r_i)_{i<\omega}$ of conditions such that $q_0=q$ and such that for all $i<\omega,$
\begin{enumerate}

\item $q_i \in \MQB_{\b}$,
\item $r_i$ is an adequate $\MQB_{\b_i}$-condition,
\item $r_i \leq_{\MQB_{\b_i}} q_i \restriction \b_i,$ and
\item $q_{i+1}=(f_{r_i} \cup f_{q_i}\restriction [\b_i, \b), \tau_{r_i} \cup \tau_{q_i})$.\footnote{Note that by Lemma \ref{compl} (and Lemma \ref{limit}), each $q_{i+1}$ is indeed a $\MQB_\b$-condition.}
\end{enumerate}

The construction can be carried out using the induction hypothesis.

Let $q^*$ be the greatest lower bound of the sequence $(q_i)_{i<\omega}$, which exists by Lemma \ref{countable closure}. Then $q^*$ is an adequate $\MQB_\b$-condition extending $q$. The point is that every instance of adequacy depends on ordinals $\a$, $\bar\a<\b$ and is in fact verified at some high enough stage $i$ of the construction.

If $\b$ is a limit ordinal of uncountable cofinality, then we fix some $\bar\b<\b$ such that $\dom(f_q)\subseteq\bar\b$,
find an adequate extension $q'$ of $q\restriction\bar\b$ in $\MQB_{\bar\b}$ (which exists by the induction hypothesis), and note that $q^*=(f_{q'}, \tau_{q'}\cup\tau_q)$ is an extension of $q$ in $\MQB_\b$ by Lemma \ref{compl}. But then we are done since $q^*$ is adequate by the choice of $q'$.

Finally, suppose that $\b=\a+1$ is a successor ordinal. By induction hypothesis together with Lemma \ref{compl}, we may assume that $q\restriction \a$ is adequate. We may also assume that there is some generalized edge coming from $\tau_q$ of the form $\langle (N_0, \b), (N_1, \g_1)\rangle$, as otherwise we are done. We build $q^*$ as the greatest lower bound of a suitably constructed descending sequence $(q_n)_{n<\o}$ of
conditions extending $q_0=q$ and such that $q_n$ is weakly adequate, and $q_n\restriction\a$ is adequate for each $n$. For every $n$, and assuming $q_n$ has been found, we construct $q_{n+1}$ in the following way.

Let us pick some correct $\tau_{q_n}$-thread $\langle\vec{\mtcl E}, (\rho, \bar\a)\rangle$. Let
$\rho<\k$, $\zeta<\o_1$, and suppose $x=(\r, \zeta)\in f_{q_n}(\bar\a)$. Let $\d=\min\{\d_{N^i_0}\,:\,i\leq m\}$, where $\vec{\mtcl E}=(\langle (N^i_0, \g^i_0), (N^i_1, \g^i_1)\rangle\,:\,i\leq m)$. Let $\a=\Psi_{\vec{\mtcl E}}(\bar\a)$.
By the adequacy of $q_n\restriction\a$ and using Lemma \ref{transfer}, we have that $q_n\restriction\mu(\a)$ forces any two distinct nodes in $f_{q_n}(\bar\a)\cap (\d\times\o_1)$ to be incomparable in $\lusim{T}_{\m(\a)}$.
This is true since for any two distinct nodes $x, y$ in $f_{q_n}(\bar\a)\cap (\d\times\o_1)$, by weak adequacy of $q_n$, some condition $r \in \MQB_{\mu(\bar\alpha)}$ weaker than $q_n \restriction \mu(\bar\alpha)$ and belonging to the final model of $\vec{\mathcal E}$ forces $x$ and $y$ to be incomparable in $T_{\mu(\bar\alpha)}$. By adequacy of $q_n \restriction \mu(\bar\alpha), \Psi_{\vec{\mathcal E}}(r)$ is also weaker than $q_n \restriction \mu(\bar\alpha)$. And by Lemma \ref{transfer}, $\Psi_{\vec{\mathcal E}}(r)$ is a condition in $\MQB_{\mu(\alpha)}$ forcing $x$ and $y$, which of course are fixed by $\Psi_{\vec{\mathcal E}}$, to be incomparable in $T_{\mu(\alpha)}$.

By Lemma \ref{str-props} we may find an extension $q_{n+1}^0\in \MQB_\b$ of $q_n$ such that $$f_{q_n}(\bar\a)\cap(\d\times\o_1)\subseteq f_{q_{n+1}^0}(\a)\cap (\d\times\o_1).\footnote{In fact, by the proof of Lemma \ref{str-props} it suffices for this to simply add to $f_{q_n}(\a)$ all nodes in $f_{q_n}(\bar\a)\cap (\d\times\o_1)$.}$$

Let then $q_{n+1}^1=(f_{q_{n+1}^0},\tau_{q_{n+1}^1})$, where $\tau_{q_{n+1}^1}$ is the union of $\tau_{q_{n+1}^0}$ and the set of edges of the form $$\Psi_{\vec{\mtcl E}}(\langle (N_0', \g_0'), (N_1', \g_1')\rangle)$$ with $\langle (N_0', \g_0'), (N_1', \g_1')\rangle\in\tau_{q_n}\cap N^0_0$ and $\g_0'$, $\g_1'\leq\bar\a$.
We note that $\Psi_{\vec{\mtcl E}}(\langle (N_0', \g_0'), (N_1', \g_1')\rangle)$ is obtained from some already present edge $\langle \Psi_{\vec{\mtcl E}}(N_0', \tilde\g_0'), \Psi_{\vec{\mtcl E}}(N_1', \tilde\g_1')\rangle$ in $q_n$ by, at most, increasing some of the markers 
$\tilde\g_\epsilon'$ to $\a$, and that doing so does not force us to add working parts that were not already present in $q_{n+1}^0$.

Let now $q_{n+1}$ be an extension of $q_{n+1}^1$, obtained by first extending $q_{n+1}^1\restriction\a$ to an adequate condition using the induction hypothesis and then applying clause (7) in the definition of condition, such that $$f_{q_{n+1}^1}(\a)\cap (\d\times\o_1)\subseteq f_{q_{n+1}}(\bar\a).$$ By further extending $q_{n+1}$ using Lemma \ref{weak-adeq} and the induction hypothesis (and Lemma \ref{compl}), we may assume in addition that $q_{n+1}$ is weakly adequate and $q_{n+1}\restriction\a$ is adequate.

Using some suitable book-keeping, we can make sure that $(q_n)_{n<\o}$ is built in such a way that every relevant $\langle\vec{\mtcl E}, (\rho, \bar\a)\rangle$ for which there is some $x=(\rho, \zeta)\in f_{q_n}(\bar\a)$, occurring at any stage $m$ in the construction, is taken care of at infinitely many stages $n>m$. Let $q^*$ be the greatest lower bound of $\{q_n \mid n<\o\}$.
  We then have that $q^*$ is an adequate condition extending $q$.
\end{proof}

Given a $\MQB_{\k^+}$-condition $q$ and a model $N$, we denote by $q\restriction N$ the ordered pair $(f, \tau_q\cap N)$, where $f$ is the function with domain $\dom(f_q)\cap N$ such that $f(x)=f_q(x)\cap N$ for every $x\in \dom(f)$.

 It will be necessary, in the proof of Lemma \ref{preserving ch}, to adjoin a certain edge to some given condition. This will be accomplished by means of the following lemma.



\begin{lemma}\label{technical-lemma2}
Suppose for every $\b<\k^+$, $\MQB_{\b}$ has the $\k$-c.c. Let $\a<\k^+$ and let $$e=\langle (N_0, \a+1), (N_1, \g_1)\rangle$$ be a generalized edge.
Suppose $(N_0, \in, \Phi_{\a+1})\prec (H(\k^+), \in, \Phi_{\a+1})$. Suppose $\MQB_{\x}\cap N_0$ is a complete suborder of $\MQB_{\x}$ for every $\x\in (\a+2)\cap N_0$.
Let $r\in\MQB^{\d_{N_0}}_{\a+1}$ and suppose $e\restriction\a$ comes from $\tau_r$. Let $$\vec{\mtcl E}=(\langle (N^i_0, \g^i_0), (N^i_1, \g^i_1)\rangle\,\mid\,i< n)$$ be a sequence of generalized edges coming from $\tau_r\cup\{e\}$ such that $\langle (N^0_0, \g^0_0), (N^1_0, \g^1_0)\rangle=e$ and $\langle \vec{\mtcl E}, (\emptyset, \a)\rangle$ is a correct thread. Let $\d=\min\{\d_{N^i_0}\,\mid\, i<n\}$ and $\bar\a= \Psi_{\vec{\mtcl E}}(\a)$. Suppose $\bar\a<\a$ and suppose $r\restriction\m(\a)$ forces every two distinct nodes in $f_r(\bar\a)\cap (\d\times\o_1)$ to be incomparable in $\lusim{T}_{\m(\a)}$.

Then there is an extension $r^*\in\MQB_{\a+1}^{\delta_{N_0}}$ of $r$
 such that
 \begin{enumerate}
 \item $f_{r}(\a)\cap (\d\times\o_1)\subseteq f_{r^*}(\bar\a)$ and
 \item $r^*\restriction \m(\a)$ forces every two distinct nodes in $$(f_{r}(\bar\a)\cap (\d\times\o_1))\cup f_r(\a)$$ to be incomparable in $\lusim{T}_{\m(\a)}$.
 \end{enumerate}
\end{lemma}

\begin{proof}
The proof is by induction on $\a$.
To start with, since $\MQB_{\a+1}\cap N_0\lessdot \MQB_{\a+1}$, $r\restriction N_0$ may be extended to a condition $q_0\in\MQB_{\a+1}\cap N_0$ forcing that every condition in $\dot G_{\MQB_{\a+1}\cap N_0}$ is compatible with $r$.

\begin{claim} We may extend $q_0$ to a condition $q_1\in\MQB_{\a+1}\cap N_0$ for which there is a generalized edge $$e'=\langle (N_0', \a), (N_1', \g_1')\rangle\in \tau_{q_1}$$ such that $\MQB_{\x}\cap N_0'$ is a complete suborder of $\MQB_{\x}$ for every $\x\in (\a+1)\cap N_0'$, together with some
$r'\in\MQB^{\d_{N_0'}}_{\a+1}$ with $e'\restriction\a$ coming from $\tau_{r'}$, and together with a sequence $$\vec{\mtcl E'}=(\langle (N'^{i}_0, \g'^{i}_0), (N'^{i}_1, \g'^{i}_1)\rangle\,\mid\,i< n)$$ of generalized edges coming from $\tau_{r'}\cup\{e'\}$ such that $$\langle (N'^{0}_0, \g'^{0}_0), (N'^{1}_0, \g'^{1}_0)\rangle=e'$$ and such that $\langle \vec{\mtcl E'}, (\emptyset, \a)\rangle$ is a correct thread and such that, letting $\d'=\min\{\d_{N'^{i}_0}\,\mid\, i<n\}$ and $\bar\a'= \Psi_{\vec{\mtcl E'}}(\a)$, we have that $\rho<\d'$ for every $(\rho, \nu)\in f_{q_0}(\a)$ with $\rho<\d$, $\bar\a'<\a$, and that $r'\restriction\mu(\a)$ forces every two distinct nodes in $f_{r'}(\bar\a')\cap (\d'\times\o_1)$ to be incomparable in $\lusim{T}_{\m(\a)}$. Moreover, $q_0\in N_0'$,
$q_0\restriction N_1'=q_0\restriction N_0'\cap N_1'$, and
$f_{q_1}(\bar\a')\cap (\d'\times\o_1)=f_{q_1}(\a)\cap (\d'\times\o_1)$.
\end{claim}

\begin{proof}
In order to find $q_1$, we first find  $e'$, $\vec{\mtcl E}'$ and $r'$ in $N_0$ as in the statement. The existence of such objects is witnessed by $e\restriction\a$, $\vec{\mtcl E}$ and $r$, respectively, and can be expressed by a sentence over $(H(\k^+), \in, \Phi_{\a+1})$ with parameters in $N_0$.

We can now find a suitable condition $q_1^-$ in $\MQB_{\a+1}\cap N_0$ extending $q_0$ and such that $e'\in \tau_{q_1^-}$. Indeed, $q_1^-$ is obtained by adding $e'$ to $\tau_{q_0}$ and copying the relevant information coming from $q_0$ into $N_1'$ via $\Psi_{N_0', N_1'}$ so as to make clauses (5) and (6) in the definition of condition hold for $q_1^-$; in other words, $q_1^-=q_0\oplus (\emptyset, \{ e'\})$.

The result of copying any piece of information carried by $q_0$ in $N_0'$ into $N_1'$ will not interfere with any piece of information previously carried by $q_0$ in $N_1'$ as that information is also in $N_0'$ and therefore fixed by $\Psi_{N_0', N_1'}$. Also, clause (7) in the definition of condition is ensured for $e'$ at all ordinals $\x+1\in N_0'\cap\a$ by the induction hypothesis applied to all $\x\in (N_0'\cup N_1')\cap\a$. It then easily follows that $q_1^-$ is a condition.

But now we may find $q_1$ as desired by simply copying the relevant information coming from $q_0$ at stage $\a$ via $\Psi_{N_0', N_1'}$, which again is possible since $q_0\restriction N_1'=q_0\restriction N_0'\cap N_1'$.
\end{proof}

Let us fix $q_1$, $e'$, $\vec{\mtcl E}'$ and $r'$ as given by the claim. By the choice of $q_0$ we may find a common extension $r^*_0\in\MQB_\a$ of $q_1\restriction\a$ and $r\restriction\a$, which we may assume is adequate by Lemma \ref{adequate}.
By adequacy of $r^*_0$ it follows that $f_{r^*_0}(\bar\a')\cap (\d'\times\o_1)=f_{r^*_0}(\bar\a)\cap (\d'\times\o_1)$. Let $q_2\in N_0$ be the amalgamation, as given by Lemma \ref{compl}, of $r^*_0\restriction N_0$ and $q_1$. In order to finish the proof it suffices to argue that $(f, \tau_{q_2})$ is a condition in $\MQB_{\a+1} \cap N_0$, where $f$ is the function such that $f\restriction\a=f_{q_2}\restriction\a$ and $f(\a)=f_{q_0}(\a)\cup (f_{r_0^*}(\bar\a)\cap (\d\times\o_1))$, since then we can take $r^*$ to be any condition in $\MQB_{\a+1}^{\delta_{N_0}}$ extending both $(f, \tau_{q_2})$ and $r$, which exists by the choice of $q_0$.

$(f, \tau_{q_2})$ is of course in $N_0$. By our hypothesis, the only way  $(f, \tau_{q_2})$ could fail to be a condition in $\MQB_{\a+1}$ is that there are distinct $x\in f_{q_0}(\a)$ and $y\in f_{r_0^*}(\bar\a)\cap (\d'\times\o_1)$ such that $q_2\restriction\m(\a)$ does not force $x$ and $y$ to be incomparable nodes in $\lusim{T}_{\m(\a)}$. We can then extend $q_2\restriction\m(\a)$ to some $r'\in\MQB_{\m(\a)}\cap N_0$ forcing $x$ and $y$ to be $\lusim{T}_{\m(\a)}$-comparable. By the $\k$-c.c.\ of $\MQB_{\m(\a)}$ we may of course assume that $r'$ extends some $\bar r\in\MQB_{\m(\a)}\cap N_0'$ forcing $x$ and $y$ to be $\lusim{T}_{\m(\a)}$-comparable. Once again by the choice of $q_0$, let $q\in\MQB_\a$ be a common extension of $r'$ and $r\restriction\m(\a)$ which, thanks to Lemma \ref{adequate}, we may assume is adequate. But now $\Psi_{\vec{\mtcl E}'}(\bar r)$ is a condition weaker than $q\restriction\m(\bar\a')$ (by clauses (5) and (6) in the definition of condition applied to $q$) and forcing $x$ and $y$ to be $\lusim{T}_{\m(\bar\a')}$-comparable, which of course is a contradiction since $x$, $y\in f_q(\bar\a')$.
\end{proof}

\section{The chain condition}
\label{preservation properties}

This section is devoted to proving Lemma \ref{chain condition theorem}.

\begin{lemma}
\label{chain condition theorem}
For each $\beta \leq \k^{+}$, $\MQB_\beta$ has the $\kappa$-chain condition. \end{lemma}

As we will see, the weak compactness of $\k$ is used crucially in order to prove Lemma \ref{chain condition theorem}.
Let $\mathcal F$ be the weak compactness filter on $\k$, i.e., the filter on $\k$ ge\-ne\-rated by the sets $$\{\l<\k\,\mid\, (V_\l, \in, B\cap V_\l)\models\psi\},$$ where $B\subseteq V_\k$ and where $\psi$ is a  $\Pi^1_1$ sentence for the structure $(V_\k, \in, B)$ such that $(V_\k, \in, B) \models \psi$. $\mathcal F$ is a proper normal filter on $\kappa$. Let also $\mtcl S$ be the collection of $\mathcal F$-positive subsets of $\k$, i.e., $$\mathcal S=\{X\subseteq\k\,\mid\, X\cap C\neq \emptyset\mbox{ for all }C\in\mtcl F\}$$

We will call a model $Q$ \emph{suitable} if $Q$ is an elementary submodel of cardinality $\k$ of some high enough $H(\theta)$, closed under ${<}\k$-sequences, and such that $\langle \MQB_\a\mid \a<\k^+\rangle\in Q$. Given a suitable model $Q$, a bijection $\varphi:\k\to Q$, and an ordinal $\l<\k$, we will denote $\varphi``\l$ by $M^\varphi_\l$. It is easily seen that $$\{\l < \k \mid M^\varphi_\l\prec Q,~ M^\varphi_\l \cap \k=\l \text{~and~} ^{<\l}M^\varphi_\l \subseteq M^\varphi_\l               \} \in \mathcal F.$$

\begin{definition}[strong chain condition]
\label{strongchainconditiondef}
Given $\b\leq\k^+$, we will say that \emph{$\MQB_\b$ has the strong $\k$-chain condition} if for every $X\in\mtcl S$, every suitable model $Q$ such that $\b, X\in Q$, every bijection $\varphi:\k\to Q$, and every two sequences $(q^0_\l\mid\l\in X)\in Q$ and $(q^1_\l\mid \l\in X)\in Q$ of adequate $\MQB_\b$-conditions, if
\begin{itemize}
\item  $M^\varphi_\l\cap\k=\l$ and
\item $q^0_\l\restriction M^\varphi_\l = q^1_\l\restriction M^\varphi_\l$ for every $\l\in X$,
\end{itemize}
\noindent then there is some $Y\in\mtcl S$, $Y\subseteq X$, together with sequences $$(r^{0}_\l\mid \l\in Y)$$ and $$(r^{1}_\l\mid\l\in Y)$$ of adequate $\MQB_\b$-conditions with the following properties.

\begin{enumerate}

\item $r^{0}_\l\leq_{\MQB_\b} q^0_\l$ and $r^{1}_\l\leq_{\MQB_\b} q^1_\l$  for every $\l\in Y$.
\item For all $\l_0<\l_1$ in $Y$, $r^{0}_{\l_0}\oplus r^{1}_{\l_1}$ is a common extension of $r^{0}_{\l_0}$ and $r^{1}_{\l_1}$.
\end{enumerate}
\end{definition}
The following lemma is an immediate consequence of Lemma \ref{adequate}.

\begin{lemma}\label{cc} For every $\b\leq\k^+$, if $\MQB_\b$ has the strong $\k$-chain condition, then $\MQB_\b$ has the $\k$-chain condition.
\end{lemma}

Following \cite{golshani-hayut}, given $\b\leq\k^+$, a suitable model $Q$ such that $\b\in Q$, a bijection $\varphi:\k\to Q$, $\l<\k$, and a $\MQB_\b$-condition $q\in Q$, let us say that \emph{$q$ is $\lambda$-compatible with respect to $\varphi$ and $\b$} if, letting $\MQB^*_\b=\MQB_\b\cap Q$, we have that
\begin{itemize}
\item $\MQB^*_\beta\cap M^\varphi_\lambda\lessdot\MQB^\ast_\b$,
\item $q\restriction M^\varphi_\l\in \MQB_\b^\ast$, and
\item $q \restriction M^\varphi_\lambda$ forces in $\MQB^*_\beta \cap M^\varphi_\lambda$ that
$q$ is in the quotient forcing $\MQB^*_\beta / \dot G_{\MQB^*_\beta \cap M^\varphi_\lambda}$; equivalently, for every $r\in\MQB^*_\b \cap M^\varphi_\lambda$, if $r \leq_{\MQB^*_\b \cap M^\varphi_\lambda} q \restriction M^\varphi_\lambda$, then $r$ is compatible with $q$.\footnote{In \cite{laver-shelah}, this situation is denoted by $\ast^\beta_\lambda(q_0, q_0 \restriction M^\varphi_\lambda)$.}
\end{itemize}

Adopting the approach from \cite{laver-shelah}, rather than proving Lemma \ref{chain condition theorem} we will prove the following more informative lemma.

\begin{lemma}\label{chain condition lemma} The following holds for every $\beta< \kappa^+$.
\begin{enumerate}
\item[$(1)_\beta$] $\MQB_\b$ has the strong $\k$-chain condition.
\item[$(2)_\beta$] Suppose $D\in\mtcl F$, $Q$ is a suitable model, $\b$, $D\in Q$, $\varphi:\k\to Q$ is a bijection, and $(q^0_\l\mid\l\in D)\in Q$ and $(q^1_\l\mid \l\in D)\in Q$ are sequences of $\MQB_\b$-conditions. Then there is some $D'\in\mtcl F$, $D'\subseteq D$, such that for every $\l\in D'$ and for all $q^{0'}_\l\leq_{\MQB_\b} q^0_\l$ and $q^{1'}_\l\leq_{\MQB_\b} q^1_\l$, if $q^{0'}_\l\restriction M^\varphi_\l\in \MQB_\b$ and $q^{0'}_\l\restriction M^\varphi_\l = q^{1'}_\l\restriction M^\varphi_\l$, then there are conditions $r^0_\l\leq_{\MQB_\b} q^{0'}_\l$ and $r^{1}_\l\leq_{\MQB_\b} q^{1'}_\l$ such that
\begin{enumerate}
\item $r^{0}_\l\restriction M^\varphi_\l = r^{1}_\l\restriction M^\varphi_\l$ and
\item $r^{0}_\l$ and $r^{1}_\l$ are both $\l$-compatible with respect to $\varphi$ and $\beta$.
\end{enumerate}
\end{enumerate}
\end{lemma}

\begin{corollary} $\MQB_{\k^+}$ has the $\k$-c.c. \end{corollary}
\begin{proof} Suppose $q_i$, for $i<\k$, are conditions in $\MQB_{\k^+}$. By Lemma \ref{adequate}, we may assume that each $q_i$, for $i<\k$,
 is  adequate. We may then fix $\b<\k^+$ such that $q_i\in\MQB_\b$ for all $i<\k$. But by Lemma \ref{chain condition lemma} $(1)_\b$ together with Lemma \ref{cc} there are $i\neq i'$ in $\k$ such that $q_i$ and $q_{i'}$ are compatible in $\MQB_\b$ and hence in $\MQB_{\k^+}$.
\end{proof}

The rest of the section is devoted to proving the above lemma.

\begin{proof} (of Lemma \ref{chain condition lemma}) The proof is by induction on $\b$. Let $\b< \k^+$ and suppose $(1)_\a$ and $(2)_\a$ hold for all $\a<\b$. We will show that $(1)_\b$ and $(2)_\b$ hold as well.

There is nothing to prove for $\b=0$, and the case $\b=1$ is trivial, using the inaccessibility of $\k$ and the fact that $\MQB_1$
is essentially the L\'{e}vy collapse turning $\k$ into $\aleph_2$.

Let us proceed to the case when $\b>1$.
 We start with the proof of $(1)_\b$.

Let $X\in\mtcl S$ be given, together with a suitable model $Q$ such that
 $\b, X\in Q$, a bijection $\varphi:\k\to Q$, and sequences $$\vec\s_0=(q^0_\l\mid\l\in X)\in Q$$ and $$\vec\s_1=(q^1_\l\mid \l\in X)\in Q$$ of adequate $\MQB_\b$-conditions such that $$M^\varphi_\l \cap \k=\l$$ and $$q^0_\l\restriction M^\varphi_\l = q^1_\l\restriction M^\varphi_\l$$ for every $\l\in X$. We need to prove that there is some $Y\in\mtcl S$, $Y\subseteq X$, together with sequences $$(r^{0}_\l\mid \l\in Y)$$ and $$(r^{1}_\l\mid\l\in Y)$$ of $\MQB_\b$-conditions such that the following holds.

\begin{enumerate}

\item $r^{0}_\l\leq_{\MQB_\b} q^0_\l$ and $r^{1}_\l\leq_{\MQB_\b} q^1_\l$ for every $\l\in Y$.
\item For all $\l_0<\l_1$ in $Y$, $r^{0}_{\l_0}\oplus r^{1}_{\l_1}$ is a common extension of $r^{0}_{\l_0}$ and $r^{1}_{\l_1}$.
\end{enumerate}
 We note that $\b+1\subseteq Q$. In what follows, we will write $M_\l$ instead of $M^\varphi_\l$.

Let $\MQB_\a^\ast =\MQB_\a\cap Q$ for every $\a\in \b+1$. By the induction hypothesis, $\MQB_\a$ has the $\k$-c.c.\ for  every $\a\in \b$. Hence, since $^{{<}\kappa} Q\subseteq Q$, we have that $\MQB_\a^\ast\lessdot \MQB_\a$ for every such $\a$; in particular, we have that for every $\a\in \mathcal X \cap \b$, $\MQB^\ast_{\a}$ forces over $V$ that $\lusim{T}_\a$ does not have $\k$-branches.

Given
\begin{itemize}
\item conditions $q^0$, $q^1$ in $\MQB_\b$,
\item nonzero stages $\a\in \dom(f_{q^0})$ and $\a'\in \dom(f_{q^1})$,\footnote{Note that, by induction hypothesis, both $\MQB_{\m(\a)}$ and $\MQB_{\m(\a')}$ have the $\k$-c.c.\ and hence $\lusim{T}_{\m(\a)}$ and $\lusim{T}_{\m(\a')}$ are both defined.}
\item nodes $x=(\rho_0, \zeta_0)$ and $y=(\rho_1, \zeta_1)$ such that $x\in f_{q^0}(\a)$ and $y\in f_{q^1}(\a')$,\footnote{$\a$ and $\a'$ may or may not be equal and the same applies to $x$ and $y$.} and
\item $\l<\k$,
\end{itemize}
\noindent we will say that \emph{$x$ and $y$ are separated below $\l$ at stages $\m(\a)$ and $\m(\a')$ by $q^0$ and $q^1$ (via $\bar x$, $\bar y$)} if there are $\bar\r<\l$ and $\zeta\neq\zeta'$ in $\o_1$ such that $\bar x=(\bar\r, \zeta)$ and $\bar y=(\bar\r, \zeta')$, and such that
\begin{enumerate}
\item $q^0\restriction\m(\a)$ extends a condition in $A^{\m(\a)}_{x, \bar\r}$ forcing $\bar x$ to be below $x$ in $\lusim{T}_{\m(\a)}$ and
\item $q^1\restriction\m(\a')$ extends a condition in $A^{\m(\a)}_{y, \bar\r}$ forcing $\bar y$ to be below $y$ in $\lusim{T}_{\m(\a')}$.
\end{enumerate}

\begin{definition}\label{separating-def} Given $Y\in\mtcl S$ such that $Y\subseteq X$ and such that $M_\l\prec Q$, $M_\l\cap\k=\l$, and $^{{<}\l} M_\l\subseteq M_\l$ for all $\l\in Y$, and given two sequences $\vec \s^\ast_{0}=(r_\l^{0}\,\mid\,\l\in Y)$, $\vec\s^\ast_{1}=(r_\l^{1}\,\mid\,\l\in Y)$ of adequate
$\MQB^\ast_{\b}$-conditions, we say that \emph{$\vec\s^\ast_{0}$, $\vec\s^\ast_{1}$ is a separating pair for $\vec\s_0$ and $\vec\s_1$} if the following holds.

\begin{enumerate}
\item For every $\l\in Y$, $r_\l^{0}\leq_{\MQB_\b} q^0_\l$, $r_\l^{1}\leq_{\MQB_\b} q^1_\l$, and $\dom(f_{r^{0}_{\l}})=\dom(f_{r^{1}_{\l}})$.

\item For every $\l\in Y$, every $\a\in\dom(f_{r_\l^{0}})\cap M_\l$, every nonzero $\a'\in \dom(f_{r_\l^1})$ such that $\a'\leq\a$, and for all $$x\in f_{r^{0}_\l}(\a)\setminus(\l\times\o_1)$$ and $$y\in f_{r^{1}_{\l}}(\a')\setminus(\l\times\o_1),$$ $x$ and $y$ are separated below $\l$ at stages $\m(\a)$ and $\m(\a')$ by $r^{0}_\l$ and $r^{1}_\l$ via some pair $\chi_0(x, y, \a, \a', \l)$, $\chi_1(x, y, \a, \a', \l)$ of nodes.

\item The following holds for all $\l_0<\l_1$ in $Y$.
\begin{enumerate}
\item $r^{0}_{\l_0}\restriction M_{\l_0}=r^{1}_{\l_1}\restriction M_{\l_1}$.
\item $r^{0}_{\l_0}\in M_{\l_1}$.
\item Let $\mtcl N_{\l_\epsilon}$ and $\Xi_{\l_\epsilon}$, for $\epsilon\in\{0, 1\}$, be defined as follows.
\begin{itemize}
\item $\mtcl N_{\l_\epsilon}$ is the union of the sets of the form $N_0\cup N_1$, where $\langle (N_0, \g_0), (N_1, \g_1)\rangle\in \tau_{r^0_{\l_\epsilon}}\cup\tau_{r^1_{\l_\epsilon}}$ and $\d_{N_0}<\l_\epsilon$.
 \item $\Xi_{\l_\epsilon}$ is the collection of ordinals of the form $\Psi_{\vec{\mtcl E}}(\bar\a)$, where $\langle\vec{\mtcl E}, (\r, \bar\a)\rangle$ is a connected $\tau_{r_{\l_\epsilon}^1}$-thread such that $(\r, \bar\a)\in N\cap (\k\times\k^+)$ for some model $N$ of height less than $\l_\epsilon$ coming from some edge in $\tau_{r_{\l_\epsilon}^1}$.
 \end{itemize}
 Then \begin{enumerate}
 \item $\mtcl N_{\l_0}\cap M_{\l_0} = \mtcl N_{\l_1}\cap M_{\l_1}$ and
 \item $\Xi_{\l_0}\cap M_{\l_0}=\Xi_{\l_1}\cap M_{\l_1}$.
 \end{enumerate}
\end{enumerate}

\item For all $\l_0<\l_1$ in $Y$, all ordinals $\a\in \dom(f_{r_{\l_0}^{0}})\cap M_{\l_0}$ and $\a'\in \dom(f_{r_{\l_1}^1})$ such that $0<\a'\leq\a$, and all nodes
$$x\in f_{r_{\l_0}^{0}}(\a)\setminus (\l_0\times\o_1)$$ and $$y'\in f_{r^{1}_{\l_1}}(\a')\setminus(\l_1\times\o_1)$$
there are
\begin{itemize}
\item a node $x'\in f_{r^{0}_{\l_1}}(\a)\setminus (\l_1\times\o_1)$,
\item a stage $\a^*\in\dom(f_{r^1_{\l_0}})$ such that $\a^*\leq\a$, and
\item a node $y\in f_{r_{\l_0}^{1}}(\a^*)\setminus (\l_0\times\o_1)$
\end{itemize}
\noindent such that
\[\chi_0(x, y, \a, \a^*, \l_0)=\chi_0(x', y', \a, \a', \l_1)\] and \[\chi_1(x, y, \a, \a^*, \l_0)=\chi_1(x', y', \a, \a', \l_1)\]
\end{enumerate}

\end{definition}

Let us now prove the following.

\begin{claim}\label{separating-pair-compatibility} Let $Y\in \mtcl S$ be such that $M_\l\cap\k=\l$ for all $\l\in Y$, and suppose $\vec\s^\ast_{0}=(r_\l^0\,\mid\,\l\in Y)$, $\vec\s^\ast_{1}=(r_\l^1\,\mid\,\l\in Y)$ is a separating pair for $\vec\s_0$ and $\vec\s_1$.
Then for all $\l_0<\l<\l_1$ in $Y$, $r_{\l_0}^{0}\oplus r_{\l_1}^{1}$ is a common extension of $r_{\l_0}^{0}$ and $r_{\l_1}^{1}$ in $\MQB_{\b}$.
\end{claim}

\begin{proof}
Suppose, towards a contradiction, that there are
$\l_0<\l<\l_1$
in $Y$ such that $r_{\l_0}^{0}\oplus r_{\l_1}^{1}$ is not a common extension of $r_{\l_0}^{0}$ and $r_{\l_1}^{1}$. It then follows that $r_{\l_0}^{0}\oplus r_{\l_1}^{1}$ is not a condition. Hence, by Lemma \ref{pre-separating-pair-compatibility}, there is an ordinal $\a<\b$ such that $$q:=(r_{\l_0}^{0}\restriction\a)\oplus (r_{\l_1}^{1}\restriction\a)$$ is a condition yet $$q^+:=(r_{\l_0}^{0}\restriction\a+1)\oplus (r_{\l_1}^{1}\restriction\a+1)$$ is not. Assuming that we are in this situation, we will derive a contradiction by proving that $q^+$ is a condition after all.

To start with, note that $\a>0$. We will need the following subclaim.

 \begin{subclaim}\label{shortness} Suppose
 \begin{enumerate}
\item  $\b_1\leq \b_0$  are ordinals in $M_{\l}$,
\item $\vec{\mtcl E}=(\langle (N^i_0, \g^i_0), (N^i_1, \g^i_1)\rangle\,\mid\,i\leq n)$ is a sequence of generalized edges coming
from $\tau_{r_{\l_1}^1}$ such that $\langle \vec{\mtcl E}, (\emptyset, \b_0)  \rangle$ is a $\tau_{r_{\l_1}^1}$-thread,
\item $\b_1=\Psi_{\vec{\mtcl E}}(\b_0)$, and
\item $\d_{N^i_0}\geq\l_1$ for all $i\leq n$.
\end{enumerate} Then $\b_1=\b_0$.
\end{subclaim}

\begin{proof}
By correctness of  $(N^0_0, \in, \Phi_0)$ within $(H(\k^+), \in, \Phi_0)$, we may pick some model $M\in N^0_0$ closed under $\vec e$ such that $\b_0\in M$, $\d_M=\l$, and $\arrowvert M\arrowvert=\l$ (since $M_\l$ is such a a model).
Given that $\b_1\leq \b_0$ are both in $M_\l$, $\d_{M_\l}=\l$, and $\b_0\in M$, by the first part of Lemma \ref{agreement} we then have that $\b_1\in M\subseteq N^0_0$.
But that means that $$(\Psi_{N^n_0, N^n_1}\circ \ldots\circ \Psi_{N^0_0, N^0_1})(\b_0) = \b_1$$ is in fact $\b_0$ since $\b_1\in N^0_0\cap N^n_1$ implies, by the second part of Lemma \ref{agreement}, that $$\b_0=(\Psi_{N^n_1, N^n_0}\circ\ldots\circ\Psi_{N^0_1, N^0_0})(\b_1)=\b_1.$$
\end{proof}

We will also be using the following subclaim.

\begin{subclaim}
\label{copies-of-w-part}
Suppose $\a^*\in \dom(f_{r_{\l_0}^{0}})\cup\dom(f_{r^{1}_{\l_1}})$, $x=(\r, \zeta)\in f_{r^{0}_{\l_0}}(\a^*)\cup f_{r^{1}_{\l_1}}(\a^*)$, $\langle \vec{\mtcl E}_*, (\r, \a^*)\rangle$ is a connected $\tau_{r^{0}_{\l_0}}\cup\tau_{r^{1}_{\l_1}}$-thread, and all members $\langle (N_0, \g_0), (N_1, \g_1)\rangle$ of $\vec{\mtcl E}_*$ such that $\d_{N_0}\geq\l_1$ are edges from $\tau_{\l_1}^1$. Then at least one of the following holds, where $\bar\a=\Psi_{\vec{\mtcl E}_*}(\a^*)$.

 \begin{enumerate}
 \item $\bar\a\in\dom(f_{r_{\l_0}^0})$ and $x\in f_{r_{\l_0}^0}(\bar\a)$.
 \item $\bar\a\in\dom(f_{r_{\l_1}^1})$ and $x\in f_{r_{\l_1}^1}(\bar\a)$.
 \item There is some $\a^{**}\in\dom(f_{r_{\l_0}^0})$ such that $x\in f_{r_{\l_0}^0}(\a^{**})$ and some connected $\tau_{r_{\l_1}^1}$-thread $\langle\vec{\mtcl E}, (\r, \a^{**})\rangle$ such that
 \begin{itemize}
 \item $\Psi_{\vec{\mtcl E}}(\a^{**})=\bar\a$ and
 \item all members of $\vec{\mtcl E}$ are edges $\langle (N_0, \g_0), (N_1, \g_1)\rangle$ such that $\d_{N_0}\geq\l_1$.
  \end{itemize}
\end{enumerate}
\end{subclaim}

\begin{proof} We prove this by induction on $|\vec{\mtcl E}_*|$, which we may obviously assume is nonzero. Let $$\vec{\mtcl E}_*=(\langle (N^i_0, \g^i_0), (N^i_1, \g^i_1)\,\mid\,i\leq m)$$ and $$e=\langle (N^m_0, \g^m_0), (N^m_1, \g^m_1)\rangle.$$ By induction hypothesis, one of (1)--(3) holds for $\langle\vec{\mtcl E}_*\restriction m, (\r, \a^*)\rangle$.\footnote{Note that $\langle\vec{\mtcl E}_*\restriction m, (\r, \a^*)\rangle$ is also a connected thread, so we may indeed apply the induction hypothesis to it. This is in contrast with the fact that it does not follow that $\langle\vec{\mtcl E}_*\restriction m, (\r, \a^*)\rangle$ is a correct thread if we just assume that $\vec{\mtcl E}_*$ is correct.}  Let $\a^\dag=\Psi_{\vec{\mtcl E}_*\restriction m}(\a^*)$.

Suppose (1) holds for $\langle\vec{\mtcl E}_*\restriction m, (\r, \a^*)\rangle$.
We have two cases. If $e\in \tau_{r_{\l_0}^0}$, then (1) holds trivially for $\langle \vec{\mtcl E}_*, (\r, \a^*)\rangle$ by
adequacy of
$r_{\l_0}^0$.
The other case is that $e\in \tau_{r_{\l_1}^1}$. If $\d_{N^m_0}\geq\l_1$, then obviously (3) holds for $\langle \vec{\mtcl E}_*, (\r, \a^*)\rangle$ as witnessed
by $\a^\dag$ and the thread $\langle (e), (\r, \a^\dag)\rangle$. We may thus assume that $\d_{N^m_0}<\l_1$. We know that $\a^\dag\in \mtcl N_{\l_1}\cap M_{\l_1}$ by clause (3)(b) in Definition \ref{separating-def} for the pair $r_{\l_1}^0$, $r_{\l_1}^1$. It follows that $\a^\dag\in \mtcl N_{\l_0}\cap M_{\l_0}$ by (3)(c), and hence $\a^\dag\in\dom(f_{r^1_{\l_1}})$ and $x\in f_{r_{\l_1}^1}(\a^\dag)$ by (3)(a). But then $\bar\a\in \dom(f_{r_{\l_1}^1})$ and $x\in f_{r_{\l_1}^1}(\bar\a)$ by
adequacy of $r_{\l_1}^1$ and so (2) holds for $\langle \vec{\mtcl E}_*, (\r, \a^*)\rangle$.

Next suppose (2) holds for $\langle\vec{\mtcl E}_*\restriction m, (\r, \a^*)\rangle$. Suppose $e\in \tau_{r_{\l_0}^0}$. Since $e\in M_{\l_1}$ by clause (3)(b) in Definition \ref{separating-def}, it follows from (3)(a) that $\a^\dag\in \dom(f_{r_{\l_0}^0})$ and $x\in f_{r_{\l_0}^0}(\a^\dag)$. But then, by
adequacy of
$r_{\l_0}^0$, $\bar\a\in \dom(r_{\l_0}^0)$ and $x\in f_{r_{\l_0}^0}(\bar\a)$. Hence (1) holds for $\langle \vec{\mtcl E}_*, (\r, \a^*)\rangle$. If $e\in \tau_{r_{\l_1}^1}$, then (2) holds trivially for $\langle\vec{\mtcl E}_*, (\r, \a^*)\rangle$, again by
adequacy of
$r_{\l_1}^1$.

Finally, suppose (3) holds for $\langle\vec{\mtcl E}_*\restriction m, (\r, \a^*)\rangle$, as witnessed by $\a^{**}\in \dom(f_{r_{\l_0}^0})$ together with a connected $\tau_{r_{\l_1}^1}$-thread $\langle\vec{\mtcl E}, (\r, \a^{**})\rangle$.
Suppose $e\in \tau_{r_{\l_0}^0}$. Then, since all models occurring in the edges in $\vec{\mtcl E}$ are of height at least $\l_1$ and since, once again by clause (3)(b) in Definition \ref{separating-def}, $r_{\l_0}^0\in M_{\l}$, we have by Subclaim \ref{shortness} that $\a^\dag=\a^{**}$. But then $\bar\a\in \dom(f_{r_{\l_0}^0})$ and $x\in f_{r_{\l_0}^0}(\bar\a)$ by
adequacy of
$r_{\l_0}^0$, and so (1) holds for $\langle\vec{\mtcl E}_*, (\r, \a^*)\rangle$. Finally, suppose $e\in r_{\l_1}^1$. If $\d_{N^m_0}\geq\l_1$, then (3) holds trivially for $\langle\vec{\mtcl E}_*, (\r, \a^*)\rangle$ as witnessed by
 $\langle \vec{\mtcl E}^{\frown}\langle e \rangle, (\r, \a^{**})\rangle$. In the other case, by clauses (3)(c)(ii), (3)(b) and (3)(a) in Definition \ref{separating-def} for the pair $r_{\l_1}^0$, $r_{\l_1}^1$, we have that $\a^{**}\in \Xi_{\l_1}\cap M_{\l_1}=\Xi_{\l_0}\cap M_{\l_0}$, and therefore
   $\a^{**}\in \dom(f_{r_{\l_1}^1})$ and $x\in f_{r_{\l_1}^1}(\a^{**})$ again by (3)(a). But then $\bar\a\in\dom(f_{r_{\l_1}^1})$ and $x\in f_{r_{\l_1}^1}(\bar\a)$ by
   adequacy of
   $r_{\l_1}^1$. Thus we have that (2) holds for $\langle\vec{\mtcl E}_*, (\r, \a^*)\rangle$, which finishes the proof of the subclaim.
\end{proof}

Remember that $$q^+=(r_{\l_0}^{0}\restriction\a+1)\oplus (r_{\l_1}^{1}\restriction\a+1)$$ and that we are aiming to prove that $q^+\in\MQB_{\a+1}$. We also know that $$q=(r_{\l_0}^{0}\restriction\a)\oplus (r_{\l_1}^{1}\restriction\a)$$ is a condition in $\MQB_\a$. One way $q^+$ could fail to be a condition is that there are $\epsilon$, $\epsilon'\in \{0, 1\}$, together with $\a_0\in \dom(f_{r^{\epsilon}_{\l_{\epsilon}}})$, $x_0=(\r_0, \zeta_0)\in f_{r^{\epsilon}_{\l_\epsilon}}(\a_0)$, $\a_1\in \dom(f_{r^{\epsilon'}_{\l_{\epsilon'}}})$, $x_1=(\r_1, \zeta_1)\in f_{r^{\epsilon'}_{\l_{\epsilon'}}}(\a_1)$, $\a_1$, $\a_0\leq\a$, and a nonzero ordinal $\bar\a$ such that
there are connected $\tau_{r_{\l_0}^0\oplus r_{\l_1}^1}$-threads $\langle \vec{\mtcl E}^*_0, (\r_0, \a_0)\rangle$ and $\langle\vec{\mtcl E}^*_1, (\r_1, \a_1)\rangle$, res\-pectively, such that
\begin{itemize}
\item $\bar\a=\Psi_{\vec{\mtcl E}^*_0}(\a_0)=\Psi_{\vec{\mtcl E}^*_1}(\a_1)$,
\item both $\vec{\mtcl E}^*_0$ and $\vec{\mtcl E}^*_1$ consist of edges in $\tau_{r_{\l_0}^0\oplus r_{\l_1}^1}$, and such that
\item $q \restriction \m(\bar \a)$ does not
force $x_0$ and $x_1$ to be incomparable in $\lusim{T}_{\m(\bar\a)}$.\end{itemize}

By Lemma \ref{simplifying-amalg}, we may replace $\langle \vec{\mtcl E}^*_0, (\r_0, \a_0)\rangle$ and $\langle\vec{\mtcl E}^*_1, (\r_1, \a_1)\rangle$ by
connected $\tau_{r_{\l_0}^0\cup r_{\l_1}^1}$-threads $\langle \vec{\mtcl E}'_0, (\r_0, \a_0)\rangle$ and $\langle\vec{\mtcl E}'_1, (\r_1, \a_1)\rangle$ all of whose members involving models of height at least $\l_1$ are edges in $\tau_{r_{\l_1}^1}$.\footnote{$\vec{\mtcl E}'_0$ and $\vec{\mtcl E}'_1$ may of course involve anti-edges.}

By Subclaim \ref{copies-of-w-part}, applied to $\a_0$, $x_0$ and the connected $\tau_{r_{\l_0}^0\cup r_{\l_1}^1}$-thread
$\langle \vec{\mtcl E}'_0, (\r_0, \a_0)\rangle$, at least one of the following holds.
 \begin{enumerate}
 \item[$(1)_0$] $\bar\a\in\dom(f_{r_{\l_0}^0})$ and $x_0\in f_{r_{\l_0}^0}(\bar\a)$.
 \item[$(2)_0$] $\bar\a\in\dom(f_{r_{\l_1}^1})$ and $x_0\in f_{r_{\l_1}^1}(\bar\a)$.
 \item[$(3)_0$] There is some $\a_0^{**}\in\dom(f_{r_{\l_0}^0})$ such that $x_0\in f_{r_{\l_0}^0}(\a_0^{**})$ and some connected $\tau_{r_{\l_1}^1}$-thread $\langle\vec{\mtcl E}_0'', (\r_0, \a_0^{**})\rangle$ such that
 \begin{itemize}
 \item $\Psi_{\vec{\mtcl E}_0''}(\a_0^{**})=\bar\a$ and
 \item all members of $\vec{\mtcl E}_0''$ are edges $\langle (N_0, \g_0), (N_1, \g_1)\rangle$ such that $\d_{N_0}\geq\l_1$.
  \end{itemize}
\end{enumerate}
Similarly, and by applying
Subclaim \ref{copies-of-w-part} to $\a_1$, $x_1$ and the connected $\tau_{r_{\l_0}^0\cup r_{\l_1}^1}$-thread
$\langle \vec{\mtcl E}'_1, (\r_1, \a_1)\rangle$, at least one of the following holds.
 \begin{enumerate}
 \item[$(1)_1$] $\bar\a\in\dom(f_{r_{\l_0}^0})$ and $x_1\in f_{r_{\l_0}^0}(\bar\a)$.
 \item[$(2)_1$] $\bar\a\in\dom(f_{r_{\l_1}^1})$ and $x_1\in f_{r_{\l_1}^1}(\bar\a)$.
 \item[$(3)_1$] There is some $\a_1^{**}\in\dom(f_{r_{\l_0}^0})$ such that $x_1\in f_{r_{\l_0}^0}(\a_1^{**})$ and some connected $\tau_{r_{\l_1}^1}$-thread $\langle\vec{\mtcl E}_1'', (\r_1, \a_1^{**})\rangle$ such that
 \begin{itemize}
 \item $\Psi_{\vec{\mtcl E}_1''}(\a_1^{**})=\bar\a$ and
 \item all members of $\vec{\mtcl E}_1''$ are edges $\langle (N_0, \g_0), (N_1, \g_1)\rangle$ such that $\d_{N_0}\geq\l_1$.
  \end{itemize}
\end{enumerate}
After changing some of the above objects if necessary, we essentially reduce to  one of the following situations.

\begin{enumerate}

\item $\bar\a\in \dom(f_{r_{\l_0}^0})$ and $x_0$, $x_1\in f_{r_{\l_0}^0}(\bar\a)$.

\item $\bar\a\in \dom(f_{r_{\l_1}^1})$ and $x_0$, $x_1\in f_{r_{\l_1}^1}(\bar\a)$.

\item $\bar\a\in\dom(f_{r_{\l_0}^0})\cap\dom(f_{r_{\l_1}^1})$, $x_0\in f_{r_{\l_0}^0}(\bar\a)$, and $x_1\in f_{r_{\l_1}^1}(\bar\a)$.

\item $\a_0\in\dom(f_{r_{\l_0}^0})$, $x_0\in f_{r_{\l_0}^0}(\a_0)$, and there is a connected $\tau_{r_{\l_1}^1}$-thread $\langle\vec{\mtcl E}_0, (\r_0, \a_0)\rangle$ such that
\begin{itemize}
 \item $\Psi_{\vec{\mtcl E}_0}(\a_0)=\bar\a$ and
 \item all members of $\vec{\mtcl E}_0$ are edges involving models of height at least $\l_1$,
 \end{itemize}
 \noindent and such that one of the following holds.

 \begin{enumerate}

 \item $\bar\a\in\dom(f_{r_{\l_0}^0})$ and $x_1\in f_{r_{\l_0}^0}(\bar\a)$.
 \item $\bar\a\in\dom(f_{r_{\l_1}^1})$ and $x_1\in f_{r_{\l_1}^1}(\bar\a)$.
 \item $\a_1\in \dom(f_{r_{\l_0}^0})$, $x_1\in f_{r_{\l_0}^0}(\a_1)$, and there is a connected $\tau_{r_{\l_1}^1}$-thread $\langle\vec{\mtcl E}_1, (\r_1, \a_1)\rangle$ such that
\begin{itemize}
 \item $\Psi_{\vec{\mtcl E}_1}(\a_1)=\bar\a$ and
 \item all members of $\vec{\mtcl E}_1$ are edges involving models of height at least $\l_1$.
 \end{itemize}
\end{enumerate}\end{enumerate}

We may clearly rule out (1) and (2) since
$q\restriction\m(\bar\a)$ extends both of $r_{\l_0}^0\restriction\m(\bar\a)$ and $r_{\l_1}^1\restriction\m(\bar\a)$.
 Let us assume that (4) holds.\footnote{We are considering this case before the case that (3) holds since the proof in the latter case will be a simpler variant of an argument we are about to see.}  We will first consider the subcase when (a) holds. By Subclaim \ref{shortness} applied to the fact that the height of all models occurring in $\vec{\mtcl E}_0$ is at least $\l_1$ and the fact that both $\bar\a$ and $\a_0$ are in $M_\l$, we get that $\bar\a=\a_0$. But then we get a contradiction as in case (1).

Let us now consider the subcase when (b) holds. By
adequacy of $r_{\l_1}^1$ it follows that $\a_0\in\dom(f_{r_{\l_1}^1})$.
If $\r_1<\l_1$, then by
adequacy of $r_{\l_1}^1$ we get that $x_1\in f_{\l_1}^1(\a_0)$ and then $x_1\in f_{r_{\l_0}^0}(\a_0)$ by clauses (3)(a)-(b) in Definition \ref{separating-def} for $r_{\l_0}^0$ and $r_{\l_1}^1$. But then $r_{\l_0}^0 \restriction \m(\a_0)$ extends conditions $r_0\in A^{\m(\a_0)}_{x_0, \bar\rho}$ and $r_1\in A^{\m(\a_0)}_{x_1, \bar\r}$, for some $\bar\r\leq\min\{\r_0, \r_1\}$, forcing $x_0$ and $x_1$ to be incomparable in $\lusim{T}_{\m(\a_0)}$ and, by Lemma \ref{transfer}, $\Psi_{\vec{\mtcl E}_0}(r_0)$ and $\Psi_{\vec{\mtcl E}_0}(r_1)$ are conditions weaker than $q \restriction \mu(\bar \a)$ and forcing $x_0$ and $x_1$ to be incomparable in $\lusim{T}_{\mu(\bar \a)}$. We may thus assume that $\r_1\geq\l_1$. Suppose $\r_0<\l_0$. Then $x_0\in f_{r_{\l_1}^1}(\a_0)$ by clause (3)(a) in Definition \ref{separating-def}, and hence $x_0\in f_{r_{\l_1}^1}(\bar\a)$ by
adequacy of $r_{\l_1}^1$. We again reach a contradiction as in case (2).  Hence we may assume $\l_0\leq\r_0$. The rest of the argument, in this case, is now essentially as in the corresponding proof in \cite{laver-shelah}. Since $\bar\a\leq\a_0$ due to the fact that all members of $\vec{\mtcl E}_0$ are edges, by an appropriate instance of clause (4) in Definition \ref{separating-def}
we may pick

\begin{itemize}
\item a node $x_0'=(\rho', \z_0')\in f_{r^{0}_{\l_1}}(\a_0)\setminus (\l_1\times\o_1)$,
\item a stage $\a^*\in \dom(f_{r_{\l_0}^{1}})$ such that $\a^*\leq\a_0$, and
\item a node $x_1^*=(\rho^*_1, \zeta^*_1)\in f_{r_{\l_0}^{1}}(\a^*)\setminus (\l_0\times\o_1)$
\end{itemize}

\noindent such that
\[\chi_0(x_0, x_1^*, \a_0, \a^*, \l_0)=\chi_0(x_0', x_1, \a_0, \bar\a, \l_1)\] and \[\chi_1(x_0, x_1^*, \a_0, \a^*, \l_0)=\chi_1(x_0', x_1, \a_0, \bar\a, \l_1)\]
\noindent (where $\chi_0$ and $\chi_1$ are the projections in Definition \ref{separating-def}).
 Let $\bar\r$ be such that $$\chi_0(x_0, x_1^*, \a_0, \a^*, \l_0)=(\bar\r, \bar\zeta_0)$$ and $$\chi_1(x_0', x_1, \a_0, \bar\a, \l_1)=(\bar\r, \bar\zeta_1)$$ for some $\bar\zeta_0\neq \bar\zeta_1$ in $\o_1$. We have that $q \restriction \m(\a_0)$ extends a condition $r_0\in A^{\m(\a_0)}_{x_0, \bar\r}$ forcing $\chi_0(x_0, x_1^*, \a_0, \a^*, \l_0)$ to be below $x_0$ in $\lusim{T}_{\m(\a_0)}$ (because this is true about $r_{\l_0}^{0}\restriction\m(\a_0)$).
 Also, $r_{\l_1}^{1}\restriction\m(\bar\a)$ extends a condition $r_1\in A^{\m(\bar\a)}_{x_1, \bar\rho}$ forcing that $\chi_1(x_0', x_1, \a_0, \bar\a, \l_1)$ is below $x_1$ in $\lusim{T}_{\m(\bar\a)}$, and therefore so does $q \restriction \m(\bar\a)$. We also have that
  $q\restriction\m(\bar\a)$ extends
 $\Psi_{\vec{\mtcl E}_0}(r_0)$, and by Lemma \ref{transfer} $\Psi_{\vec{\mtcl E}_0}(r_0)$ forces $\chi_0(x_0, x_1^*, \a_0, \a^*, \l_0)$ to be below $x_0$ in $\lusim{T}_{\m(\bar\a)}$. But now we get a contradiction since $\bar\zeta_0\neq\bar\zeta_1$ and hence $q \restriction \m(\bar\a)$ forces $x_0$ and $x_1$ to be incomparable in $\lusim{T}_{\m(\bar\a)}$.

It remains to consider the subcase that (c) holds. Since  all models occurring in members of $\vec{\mtcl E}_0$ or of $\vec{\mtcl E}_1$ are of height at least $\l_1$ and both $\a_0$ and $\a_1$ are in $M_\l$, by Subclaim \ref{shortness} we get that $\a_0=\a_1$.  But now we get a contradiction by the same argument we have already encountered using Lemma \ref{transfer}.

We finally handle the case when (3) holds. In this case we may assume that both $\r_0\geq\l_0$ and $\r_1\geq\l_1$ hold, as otherwise we get, by an application of clause (3)(a) in Definition \ref{separating-def}, that at least one of $x_0$, $x_1$ is in $f_{r_{\l_0}^0}(\bar\a)\cap f_{r_{\l_1}^1}(\bar\a)$, which immediately yields a contradiction. But now, since $\r_0\geq\l_0$ and $\r_1\geq\l_1$, we obtain a contradiction by a separation argument using clause (4) in Definition \ref{separating-def}---with both $\a$ and $\a'$, in that definition, being $\bar\a$---like the one we have already seen.

We will now prove that clause (7) in the definition of $\MQB_{\a+1}$-condition holds for $q^+$. This will conclude the proof that $q^+$ is a condition (the verification of all remaining clauses in the definition of $\MQB_{\a+1}$-condition is immediate), and will therefore complete the proof of the claim.



Suppose $\bar\a<\a+1$ and $e=\langle (N_0, \bar\a+1), (N_1, \g_1)\rangle$ is a generalized edge coming from  $\tau_{q^+}\restriction\bar\a+1$. We must show that the following holds.

Let $r\in\MQB^{\d_{N_0}}_{\bar\a+1}$ be such that $e\restriction\bar\a$ comes from $\tau_r$, and suppose $$\vec{\mtcl E}=(\langle (N^i_0, \g^i_0), (N^i_1, \g^i_1)\rangle\,\mid\,i< n)$$ is a sequence of generalized edges coming from $\tau_r\cup\{e\}$ such that $\langle (N^0_0, \g^0_0), (N^0_1, \g^0_1)\rangle=e$ and $\langle \vec{\mtcl E}, (\emptyset, \bar\a)\rangle$ is a correct thread. Let $\d=\min\{\d_{N^i_0} \mid i<n\}$ and $\a'= \Psi_{\vec{\mtcl E}}(\bar\a)$. Suppose
  $r\restriction\m(\bar\a)$ forces every two distinct nodes in $f_r(\a')\cap (\d\times\o_1)$ to be incomparable in $\lusim{T}_{\m(\bar\a)}$.
Then there is an extension $r^*\in\MQB^{\d_{N_0}}_{\bar\a+1}$ of $r$ such that

 \begin{enumerate}
 \item $f_{r^*}(\bar\a)\cap(\d\times\o_1)\subseteq f_{r^*}(\a')\cap(\d\times\o_1)$, and
\item $r^*\restriction \m(\bar\a)$ forces every two distinct nodes in $$(f_{r}(\a')\cap (\d\times\o_1))\cup  f_r(\bar\a)$$ to be incomparable in $\lusim{T}_{\m(\bar\a)}$.
 \end{enumerate}

We may assume that $e$ does not come from either $\tau_{r_{\l_0}^0}$ or $\tau_{r_{\l_1}^1}$, as otherwise we would be done since both $r_{\l_0}^0$ and $r_{\l_1}^1$ are conditions. The crucial point is now that, thanks to Lemma \ref{definability}, the above is a fact about $e$ that can be expressed over $(H(\k^+), \in, \Phi_{\bar\a+1})$ with $e$ as parameter. Letting now $\a^\dag = \max\{\bar\a+1, \g_1\}$, $e=\Psi_{\vec{\mtcl E}}(e^*)$ for some generalized edge $e^*$ coming from $(\tau_{r_{\l_0}^0}\restriction\a^*+1)\cup(\tau_{r_{\l_1}^1}\restriction\a^*+1)$, for some $\a^*$, and some appropriate connected $\tau_{r_{\l_0}^0}\cup\tau_{r_{\l_1}^1}$-thread $\langle\vec{\mtcl E}, (e^*, \a^*)\rangle$ such that $\Psi_{\vec{\mtcl E}}(\a^*)=\a^\dag$. The corresponding fact holds in $(H(\k^+), \in, \Phi_{\a^*})$ about $e^*$ since both $r_{\l_0}^0\restriction\a^*+1$ and $r_{\l_1}^1\restriction\a^*+1$ are $\MQB_{\a^*+1}$-conditions. But then the desired fact holds about $e$ in $(H(\k^+), \in, \Phi_{\bar\a+1})$ by correctness of $\vec{\mtcl E}$, using the fact that $\Psi_{\vec{\mtcl E}^{-1}}(\bar\a+1)\leq\a^*$. This concludes the proof of Claim \ref{separating-pair-compatibility}.
\end{proof}

The following technical fact appears essentially in \cite{laver-shelah}.

\begin{claim}\label{main-technical-fact} Suppose $Z\in\mtcl S$,
$(p^{0}_\l\,\mid\,\l\in Z)\in Q$ and $(p^{1}_\l\,\mid\,\l\in Z)\in Q$ are sequences of conditions in $\MQB_{\b}^\ast$, and suppose that for every $\l\in Z$,
\begin{itemize}
\item $p_\l^0\restriction M_\l$ and $p_\l^1\restriction M_\l$ are compatible conditions in $\MQB_{\b}^\ast\cap M_\l$,
\item $p^0_\l$ and $p^1_\l$ are $\lambda$-compatible with respect to $\varphi$ and
$\a$ for all $\a<\b$,
\item $\a_\l\in  \dom(f_{p^0_\l})\cap M_\l$,
\item $\a'_\l\in \dom(f_{p^1_\l})$ is a nonzero ordinal such that $\a_\l'\leq\a_\l$, and
\item $x_\l=(\r^0_\l, \z^0_\l)$ and $y_\l=(\r^1_\l, \z^1_\l)$ are nodes of level at least $\l$  such that
$x_\l\in f_{p^0_\l}(\a_\l)$ and $y_\l\in f_{p^1_\l}(\a'_\l)$.\end{itemize} \noindent Then there is $D \in \mtcl F$, together with two sequences $(p_\l^{2}\,\mid\,\l\in Z\cap D)$, $(p_\l^{3}\,\mid\,\l\in Z\cap D)$ of conditions in $\MQB_{\b}^\ast$ such that

\begin{enumerate}
\it for each $\l\in Z\cap D$, $p_\l^{2}\leq p_\l^0$ and $p_\l^{3}\leq p_\l^1$,
\it for each $\l\in Z\cap D$, $p_\l^{2}\restriction M_\l$ and  $p_\l^{3}\restriction M_\l$ are compatible in $\MQB_{\b}^\ast\cap M_\l$, and
 \it for each $\l\in Z\cap D$, $x_\l$ and $y_\l$ are separated below $\l$ at stages $\m(\a_\l)$ and $\m(\a'_\l)$ by $p^{2}_\l\restriction\m(\a_\l)$ and $p^{3}_\l\restriction\m(\a'_\l)$.\end{enumerate}\end{claim}

\begin{proof} Let $B\subseteq V_\kappa$ code $\varphi$, $(\MQB_\a^\ast)_{\a\in (\b+1)\cap Q}$, the collection of maximal antichains of $\MQB_\a^\ast$, for $\a\in \b\cap Q$, and $(\lusim{T}_\a)_{\a\in \mtcl X\cap\b \cap Q}$. By a reflection argument with an appropriate $\Pi^1_1$ sentence over the structure $(V_\k, \in, B)$, together with the fact that $\MQB^\ast_\a$ has the $\k$-c.c.\ for every $\a\in \b\cap Q$, there is a set $D\in \mtcl F$ consisting of inaccessible cardinals $\l<\k$ for which $M_\l$ is a model such that $M_\l\cap\k=\l$, $M_\l$ is closed under ${<}\l$-sequences, and such that for every $\a\in M_\l\cap\b$,
\begin{enumerate}
\item  $\MQB^\ast_{\m(\a)}\cap M_\l$ forces, over $V$, that $\lusim{T}_{\m(\a)}\cap M_\l$ has no $\l$-branches,
\item $\MQB^\ast_{\a}\cap M_\l$ has the $\l$-c.c., and
\item $\MQB^\ast_\a\cap M_\l\lessdot \MQB^\ast_\a$
\end{enumerate}

Fix $\l\in Z\cap D$.
Thanks to Lemma \ref{compl}, it suffices to show that there are extensions $p_\l^2$ and $p_\l^3$ of $p_\l^0\restriction\a_\l$ and $p_\l^1\restriction\a'_\l$, respectively, such that $p_\l^{2}\restriction M_\l$ and $p_\l^{3}\restriction M_\l$ are compatible in $\MQB_{\a_\l}^\ast\cap M_\l$, and such that $x_\l$ and $y_\l$ are separated below $\l$ at stages $\m(\a_\l)$ and $\m(\a_\l')$ by $p^{2}_\l\restriction\m(\a_\l)$ and $p^{3}_\l\restriction\m(\a'_\l)$.
By (3) we may view $\MQB^\ast_{\a_\l}$ as a  two-step forcing iteration $(\MQB^\ast_{\a_\l}\cap M_\l)\ast \lusim{\MSB}$. By $\l$-compatibility we may then identify $p_\l^0\restriction\a_\l$ and $p_\l^1\restriction\a_\l$ with, respectively, $\langle r^0, \lusim{s}^0\rangle$ and $\langle r^1, \lusim{s}^1\rangle$, both in $(\MQB^\ast_{\a_\l}\cap M_\l)\ast \lusim{\MSB}$.

Note that $r^0$ and $r^1$ are compatible in $\MQB^\ast_{\a_\l}\cap M_\l$. Working in an $(\MQB^\ast_{\a_\l}\cap M_\l)$-generic extension $V[G]$ of $V$ containing $r^0$ and $r^1$, we note that there have to be
\begin{itemize}
\item extensions $\langle r^{00}, \lusim{s}^{00}\rangle$ and $\langle r^{01}, \lusim{s}^{01}\rangle$ of $\langle r^0, \lusim{s}^0\rangle$ and
\item an extension $\langle r^{3}, \lusim{s}^{3}\rangle$ of $\langle r^1, \lusim{s}^1\rangle$
\end{itemize}

\noindent such that $r^{00}$, $r^{01}$ and $r^{3}$ are all in $G$, together with some $\bar\r<\l$ for which there is a pair $\zeta^{00}\neq \zeta^{01}$ of ordinals in $\o_1$ and there is $\zeta^3\in\omega_1$ such that, identifying $\lusim{T}_{\m(\a_\l)}$ and $\lusim{T}_{\m(\a'_\l)}$ with $(\MQB^\ast_{\a_\l}\cap M_\l)\ast \lusim{\MSB}$-names,\footnote{$\lusim{T}_{\m(\a'_\l)}$ is of course a $\MQB^\ast_{\a_\l}$-name since $\MQB^\ast_{\a'_\l}\subseteq\MQB^*_{\a_\l}$, so this identification makes sense.} we have the following.

\begin{itemize}
\item $\langle r^{00}, \lusim{s}^{00}\rangle$ extends a condition in $A^{\mu(\alpha_\lambda)}_{x_\l, \bar\rho}$ forcing that $(\bar\rho, \zeta^{00})$ is below $x_\l$ in $\lusim{T}_{\m(\a_\l)}$.
\item $\langle r^{01}, \lusim{s}^{01}\rangle$ extends a condition in $A^{\mu(\alpha_\lambda)}_{x_\l, \bar\rho}$ forcing that $(\bar\rho, \zeta^{01})$ is below $x_\l$ in $\lusim{T}_{\m(\a_\l)}$.
\item $\langle r^{3}, \lusim{s}^{3}\rangle$ extends a condition in $A^{\mu(\alpha'_\lambda)}_{y_\l, \bar\rho}$ forcing that $(\bar\rho, \zeta^{3})$ is below $y_\l$ in $\lusim{T}_{\m(\a'_\l)}$.
\end{itemize}

Indeed, any condition $\langle r, \lusim{s}\rangle$ in $(\MQB^\ast_{\a_\l}\cap M_\l)\ast \lusim{\MSB}$ such that $r\in G$ can be extended, for any $\bar\r<\l$, to a condition $\langle r^+, \lusim{s}^+\rangle$ such that
\begin{itemize}
\item $\langle r^+, \lusim{s}^+\rangle$ is stronger than some condition in $A^{\m(\a_\l)}_{x_\l, \bar\r}$ deciding some node $(\bar\rho, \zeta)$ to be below $x_\l$ in $\lusim{T}_{\m(\a_\l)}$, and
\item $r^+\in G$,
\end{itemize}
\noindent and similarly with $y_\l$ and $\lusim{T}_{\m(\a'_\l)}$ in place of $x_\l$ and $\lusim{T}_{\m(\a_\l)}$. Hence, if the above were to fail, then the following would hold.

\begin{itemize}
\item For every $\bar\rho<\l$
there is exactly one $\zeta<\o_1$ for which there is some condition $\langle r, \lusim{s}\rangle$ in $(\MQB^\ast_{\a_\l}\cap M_\l)\ast \lusim{\MSB}$ stronger than $\langle r^0, \lusim{s}^0\rangle$ with $r\in G$, and such that $\langle r, \lusim{s}\rangle$ extends a condition in $A^{\m(\a_\l)}_{x_\l, \bar\rho}$ forcing that $(\bar\rho, \zeta)$ is below $x_\l$ in $\lusim{T}_{\m(\a_\l)}$.
\end{itemize}

It would then follow that  $\lusim{T}_{\m(\a_\l)}$ has a $\l$-branch in $V[G]$, which contradicts (1).

Let  $\zeta^3<\o_1$ be such that some condition $\langle r^3, \lusim{s}^3\rangle$ extending $\langle r^1, \lusim{s}^1\rangle$ is such that
\begin{itemize}
\item $\langle r^3, \lusim{s}^3\rangle\restriction \m(\a'_\l)$ extends  a condition in $A^{\m(\a_\l')}_{y_\l, \bar\r}$ forcing $(\bar\rho, \zeta^{3})$ to be below $y_\l$ in $\lusim{T}_{\m(\a'_\l)}$, and
\item $r^3\in G$
\end{itemize}

But now, given conditions $\langle r^{0i}, \lusim{s}^{0i}\rangle$ as above (for $i\in\{0, 1\}$) there must be $i\in\{0, 1\}$ such that $\zeta^{0 i}\neq\zeta^{3}$. We may then set $p_\l^2= \langle r^{0i}, \lusim{s}^{0i}\rangle$ and $p_\l^3=\langle r^3, \lusim{s}^3\rangle$.
\end{proof}

By Claim \ref{separating-pair-compatibility}, in order to conclude the proof of the current instance of $(1)_\b$, it suffices to prove the following.

\begin{claim}\label{there is a separating pair}  There is a separating pair for $\vec\s_0$ and $\vec\s_1$.
\end{claim}
\begin{proof}
This follows from first applying
Claim \ref{main-technical-fact} and $(2)_\a$, for $\a<\b$, countably many times, using the normality of $\mtcl F$, and then running a pressing-down argument again using the normality of $\mtcl F$.

To be more specific, we start by building sequences $$\vec\s_{0, n}=(q_{\l, n}^{0}\,\mid\,\l\in X\cap D_n)$$ and $$\vec\s_{1, n}=(q_{\l, n}^{1}\,\mid\,\l\in  X\cap D_n),$$ for a $\subseteq$-decreasing sequence $(D_n)_{n<\o}$ of sets in $\mtcl F$, such that $\vec\s_{0, 0}=\vec\s_0$ and $\vec\s_{1, 0}=\vec\s_1$, and such that for every $n<\o$, $\vec\s_{0, n+1}$ and $\vec\s_{1, n+1}$ are obtained from $\vec\s_{0, n}$ and $\vec\s_{1, n}$ in the following way.

We first let $$\vec\s_{0, n, +}=(q_{\l, n, +}^{0}\,\mid\, \l\in X\cap D_n)$$ and $$\vec\s_{1, n, +}=(q_{\l, n, +}^{1}\,\mid\,\l\in X\cap D_n)$$ be sequences of
$\MQB_{\b}^\ast$-conditions such that for every $\l\in X\cap D_n$,
\begin{itemize}
\item $q_{\l, n, +}^{0}\leq_{\MQB_\b} q_{\l, n}^{0}$ and $q_{\l, n, +}^{0}\leq_{\MQB_\b} q_{\l, n}^{1}$,
\item $q_{\, n, +}^0\restriction M_\l$ and $q_{\, n, +}^1\restriction M_\l$ are compatible conditions in $\MQB_{\b}^\ast\cap M_\l$, and
\item $q_{\l, n, +}^{0}$ and $q_{\l, n, +}^{1}$ are both $\l$-compatible with respect to $\varphi$ and $\a$ for every $\a< \b$.
\end{itemize}


Recall that $e_\b:\k\longrightarrow \b$ is a surjection. Let also $D_{-1}=\k$. We may take $D_n$ to be the diagonal intersection $\Delta_{\x<\k}D^n_\x$, where for each $\a<\k$, $D^n_\x$ witnesses $(2)_{e_\b(\x)}$ for  $\vec\s_{0, n}$, $\vec\s_{1, n}$, $\varphi$, and $D_{n-1}$, i.e., $D^n_\x\in\mtcl F$ is such that $D^n_\x\subseteq D_{n-1}$ and such that for every $\l\in D^n_\x$ and for all $q^{0'}_\l\leq_{\MQB_{e_\b(\x)}} q^0_{\l,n}$ and $q^{1'}_\l\leq_{\MQB_{e_\b(\x)}} q^1_{\l, n}$, if $q^{0'}_\l\restriction M_\l\in \MQB_{e_\b(\x)}$ and $q^{0'}_\l\restriction M^\varphi_\l = q^{1'}_\l\restriction M_\l$, then there are conditions $r^0_\l\leq_{\MQB_{e_\b(\x)}} q^{0'}_\l$ and $r^{1}_\l\leq_{\MQB_{e_\b(\x)}} q^{1'}_\l$ such that

\begin{enumerate}
\item $r^{0}_\l\restriction M_\l = r^{1}_\l\restriction M_\l$ and
\item $r^{0}_\l$ and $r^{1}_\l$ are both $\l$-compatible with respect to $\varphi$ and $e_\b(\x)$.
\end{enumerate}

Given $\l\in X\cap D_n$, we need to construct $q_{\l, n, +}^0$ and $q_{\l, n, +}^1$. For this, let $W^\epsilon_{\l, n}$ be, for each $\epsilon\in \{0, 1\}$, the set of ordinals $\a\in M_\l\cap\b$ such that
\begin{itemize}
\item $\a\in \dom(f_{q_{\l, n}^\epsilon})$ or
\item there is a connected $\tau_{q^\epsilon_{\l, n}}$-thread $\langle\vec{\mtcl E}, \a\rangle$ with $\Psi_{\vec{\mtcl E}}(\a)\in\dom(f_{q_{\l, n}^\epsilon})$ and such that $\vec{\mtcl E}$ consists of edges.
\end{itemize}
We of course have that $|W^\epsilon_{\l, n}|\leq\al_0$. We may assume that both $W^0_{\l, n}$ and $W^0_{\l, n}$ are nonempty (the proof in the case when at least one of $W^0_{\l, n}$ and $W^0_{\l, n}$ is empty is an easier variation of the proof in the other case). Let $\epsilon\in \{0, 1\}$ be such that $\ssup(W^\epsilon_{\l, n})=\max\{\ssup(W^0_{\l, n}), \ssup(W^1_{\l, n})\}$. We will assume that $\ssup(W^\epsilon_{\l, n})$ has countable cofinality (the proof when $W^\epsilon_{\l, n}$ is empty or has a maximum is an easier variant of the proof in the case that  $\cf(\ssup(W^\epsilon_{\l, n}))=\omega$). Let $(\b_m)_{m<\o}$ be a strictly increasing sequence of ordinals in $M_\l$ converging to $\ssup(W^\epsilon_{\l, n})$. We construct $q_{\l, n, +}^0$ and $q_{\l, n, +}^0$ as the greatest lower bound of $(q_{\l, n, m}^0)_{m<\o}$ and $(q_{\l, n, m}^1)_{m<\o}$, respectively, where $q_{\l, n, 0}^0=q_{\l, n}^0$ and $q_{\l, n, 0}^1=q_{\l, n}^1$ and where, for each $m<\o$,
 $q_{\l, n, m+1}^0=r_{\l, n, m}^0\oplus q_{\l, n, m}^0$ and $q_{\l, n, m+1}^1=r_{\l, n, m}^1\oplus q_{\l, n, m}^1$, where $r^0_{\l, n, m}$, $r^1_{\l, n, m}\in\MQB_{\b_m}$ are conditions extending $q_{\l, n, m}^0$ and $q_{\l, n, m}^1$, respectively, and such that

 \begin{enumerate}
\item $r^{0}_{\l, n, m}\restriction M_\l = r^{1}_{\l, n, m}\restriction M_\l$ and
\item $r^{0}_{\l, n, m}$ and $r^{1}_{\l, n, m}$ are both $\l$-compatible with respect to $\varphi$ and $\b_m$.
\end{enumerate}

We may assume each $r^0_{\l, n, m}$ and $r^1_{\l, n, m}$ to be in $Q$, so that $q_{\l, n, +}^0$ and $q_{\l, n, +}^1$ are both in $\MQB_\b^*$.

 Now we find $D_{n+1}$ and $\vec\s_{0, n+1}$, $\vec\s_{1, n+1}$  by an application of Claim \ref{main-technical-fact} to $\vec\s_{0, n, +}$ and $\vec\s_{1, n, +}$ with an appropriate sequence $\a_{\l, n}$, $\a'_{\l, n}$, $x_{\l, n}$, $y_{\l, n}$ (for $\l\in X\cap D_n$). By extending $q^0_{\l, n+1}$ and $q^1_{\l, n+1}$ if necessary for $\l\in X\cap D_{n+1}$ we may assume that for every such $\l$, $q^0_{\l, n+1}$ and $q^1_{\l, n+1}$ are both adequate conditions, and $\dom(f_{q^0_{\l, n+1}})=\dom(f_{q^1_{\l, n+1}})$ .

Let $r_\l^{0}$ and $r_\l^{1}$ be the greatest lower bound of, respectively, $(q_{\l, n}^{0})_{n<\o}$ and $(q_{\l, n}^{1})_{n<\o}$, for $\l\in X\cap \bigcap_n D_n$.

By construction we have that for all $\l\in X\cap\bigcap_n D_n$, $r_\l^0$ and $r_\l^1$ are both adequate conditions, and  $\dom(f_{r_\l^0})=\dom(f_{r_\l^1})$ .
Also, by a standard book-keeping argument we can ensure that all relevant objects $\a_{\l, n}$, $\a'_{\l, n}$, $x_{\l, n}$, $y_{\l, n}$  (for $n<\o$ and $\l\in X\cap D_n$) have been chosen in such a way that in the end $$(r_\l^{0}\,\mid\,\l\in X\cap \bigcap_n D_n)$$ and $$(r_\l^{1}\,\mid\,\l\in X\cap \bigcap_n D_n)$$ satisfy clause (2) in Definition \ref{separating-def} as well.
Finally, by a standard pressing-down argument using the normality of $\mtcl F$, we may find $Y\in\mtcl S$, $Y\subseteq X\cap \bigcap_n D_n$, such that $\vec\s^\ast_{0}=(r_\l^{0}\,\mid\,\l\in Y)$ and  $\vec\s^\ast_{1}=(r_\l^{1}\,\mid\,\l\in Y)$ satisfy clauses (3) and (4)  in Definition \ref{separating-def}.
\end{proof}

We are left with proving $(2)_\b$. This is established with an argument similar to the one in the corresponding proof from \cite{laver-shelah}. Suppose $D\in\mtcl F$, $Q$ is a suitable model such that $\b, D\in Q$, $\varphi:\k\to Q$ is a  bijection, and $(q^0_\l\mid\l\in D)\in Q$ and $(q^1_\l\mid \l\in D)\in Q$ are sequences of $\MQB_\b$-conditions. By shrinking $D$ if necessary we may assume that $M_\l\cap\k=\l$ for each $\l\in D$. \
It suffices to show that there is some $D'\in\mtcl F$, $D'\subseteq D$, with the property that for every $\l\in D'$,  if $q^{0'}_\l\leq_\b q^0_\l$ and $q^{1'}_\l\leq_\b q^1_\l$ are such that $q^{0'}_\l\restriction M_\l\in\MQB_\b$ and $q^{0'}_\l\restriction M_\l= q^{1'}_\l\restriction M_\l$, then there is a condition $r_\l\leq_{\MQB_\b} q^{0'}_\l\restriction M_\l$ such that every condition in $\MQB_\b\cap M_\l$ extending $r_\l$ is compatible with both $q^{0'}_\l$ and $q^{1'}_\l$.

The case when $\b$ is a limit ordinal follows from the induction hypo\-thesis, using the normality of $\mtcl F$ (cf.\ the proof in \cite{laver-shelah}). Specifically, we fix an increasing sequence $(\b_i)_{i<\cf(\b)}$ of ordinals in $Q$ converging to $\b$. If $\cf(\b)=\k$, we take each $\b_i$ to be $\sup(M_\l\cap\b)$ for some $\l\in D$. For each $i<\cf(\b)$ we fix some $D_i\in\mtcl F$, $D_i\subseteq D$, witnessing $(2)_{\b_i}$ for $(q^0_\l\restriction \b_i\mid\l\in D)$ and $(q^1_\l\restriction \b_i\mid \l\in D)$. We make sure that $(D_i)_{i<\cf(\b)}$ is $\subseteq$-decreasing. If $\cf(\b)<\k$, then $D'=\bigcap_{i<\cf(\b)}D_i$ will witness $(2)_\b$ for $(q^0_\l\mid\l\in D)$ and $(q^1_\l\mid \l\in D)$, and if $\cf(\b)=\k$, $D'=\Delta_{i<\k} D_i$ will witness $(2)_\b$ for these objects. This can be easily shown, using the fact that each $M_\l$ is closed under $\o$-sequences in the case when $\cf(\b)=\o$.

To see this, suppose $\l\in D'$, $q^{0'}_\l\leq_{\MQB_\b}q^0_\l$, $q^{1'}_\l\leq_{\MQB_\b}q^1_\l$, $q^{0'}_\l\restriction M_\l\in \MQB_\b$, and $q^{0'}_\l\restriction M_\l=q^{1'}_\l\restriction M_\l$. Suppose first that $\cf(\b)>\o$. In this case, we pick any $i\in\cf(\b)\cap M_\l$ such that $\b_i$ is

\begin{itemize}
\item above $(\dom(f_{q^{0'}_\l})\cup\dom(f_{q^{1'}_\l}))\cap \sup(M_\l\cap\b)$ and
\item above every ordinal $\a\in M_\l\cap\b$ such that $\Psi_{\vec{\mtcl E}}(\a)\in \dom(f_{q^{0'}_\l})\cup\dom(f_{q^{1'}_\l})$, for  some connected $\tau_{q^{0'}_\l}\cup\tau_{q^{1'}_\l}$-thread $\langle\vec{\mtcl E}, \a\rangle$ such that $\vec{\mtcl E}$ consists of edges,\footnote{Cf.\ the proof of Claim \ref{there is a separating pair}.}

\end{itemize}
\noindent and find a condition $r\in M_\l\cap\MQB_{\b_i}$ with the property that every condition in $\MQB_{\b_i}\cap M_\l$ is compatible with both $q^{0'}_\l\restriction \b_i$ and $q^{1'}_\l\restriction \b_i$. Let $r_\l$ be any condition in $\MQB_\b\cap M_\l$  extending $r$ and $q^{0'}_\l\restriction M_\l$. It then follows that every condition in $\MQB_\b\cap M_\l$ extending $r_\l$ is compatible with both $q^{0'}_\l$ and $q^{1'}_\l$.

Now suppose $\b$ has countable cofinality.  Since $\b\in M_\l$, we may build sequences $(q^{0, i}_\l\,\mid\,i<\o)$, $(q^{1, i}_\l\,\mid\,i<\o)$ and $(r^i_\l\,\mid\,i<\o)$ such that for each $i$,

\begin{enumerate}
\item $q^{0, i}_\l$ and $q^{1, i}_\l$ are conditions in $\MQB_{\b_i}$ extending $q^{0'}_\l\restriction \b_i$ and $q^{1'}_\l\restriction \b_i$, respectively,
\item $r^i_\l\in \MQB_{\b_i}\cap M_\l$,
\item every condition in $\MQB_{\b_i}\cap M_\l$ extending $r^i_\l$ is compatible with both $q^{0, i}_\l$ and $q^{1, i}_\l$,
\item $q^{0, i+1}_\l\restriction\b_i$ extends $q^{0, i}_\l$ and $r^{i}_\l$,
\item $q^{1, i+1}_\l\restriction\b_i$ extends $q^{1, i}_\l$ and $r^{i}_\l$,  and
\item $r^{i+1}_\l\restriction\b_i$ extends $r^i_\l$.\end{enumerate}
Let $r_\l\in \MQB_\b$ be the greatest lower bound of $\{r^i_\l\,\mid\,i<\o\}$, and note that $r_\l\in M_\l$ since $M_\l$ is closed under sequences of length $\o$. But now it is straightforward to verify that every condition in  $\MQB_\b\cap M_\l$ extending $r_\l$ is compatible with $q^{0'}_\l$ and $q^{1'}_\l$.

It remains to consider the case that $\b$ is a successor ordinal, $\b=\b_0+1$.
Assuming the desired conclusion fails, there is some $X\in\mtcl S$, $X\subseteq D$, together with sequences $(q^{0'}_\l\,\mid\l\in X)$ and $(q^{0'}_\l\,\mid\l\in X)$ of conditions in $\MQB_\b$ such that for every $\l\in X$,
\begin{itemize}
\item $q^{0'}_\l$ extends $q^0_\l$ and $q^{1'}_\l$ extends $q^1_\l$,
\item $q^{0'}_\l\restriction M_\l\in \MQB_\b$,
\item $q^{0'}_\l\restriction M_\l=q^{1'}_\l\restriction M_\l$, and
\item for every condition $r$ in $\MQB_\b\cap M_\l$ extending $q^{0'}_\l \restriction M_\l$ there is a condition in $\MQB_\b\cap M_\l$ extending $r$ and incompatible with at least one of $q^{0'}_\l$, $q^{1'}_\l$.
\end{itemize}

Thanks to the induction hypothesis applied to $\b_0$ and to the fact that $(1)_\b$ holds we may assume, after shrinking $X$ to some $Y\in\mtcl S$ and extending the corresponding conditions if necessary, that for each $\l\in Y$,

\begin{itemize}
\item $q^{0'}_\l\restriction\b_0$ and $q^{1'}_\l\restriction\b_0$ are both $\l$-compatible with respect to $\varphi$ and $\b_0$, and
\item $q^{0'}_\l\oplus q^{1'}_{\l^\ast}$ is a condition for each $\l^\ast\in Y$, $\l^\ast>\l$.
\end{itemize}

By our assumption above we may then assume, after shrinking $Y$ if necessary, that for each $\l\in Y$ there is a maximal antichain $A_\l$ of $\MQB_\b\cap M_\l$ below $q^0_\l\restriction M_\l$ consisting of conditions $r$ such that at least one of the following statements holds.
\begin{itemize}
\item[$\theta_{r, 0, \l}$:] $r$ is incompatible with $q^{0'}_\l$.
\item[$\theta_{r, 1, \l}$:] $r$ is incompatible with $q^{1'}_\l$.\end{itemize}

By the definition of $\mtcl F$ coupled with an appropriate $\Pi^1_1$-reflection argument, we may further assume that each $A_\l$ is in fact a maximal antichain of $\MQB_\b$ below $q^{0'}_\l\restriction M_\l$ and that it has cardinality less than $\l$ (cf.\ the proof of Claim \ref{main-technical-fact}). Hence, after shrinking $Y$ one more time using the normality of $\mtcl F$, we may assume, for all $\l<\l^\ast$ in $Y$, that
\begin{itemize}
\item $A_\l=A_{\l^\ast}$ and that
\item for every $r\in A_\l$, $\theta_{r, 0, \l}$ holds if and only if $\theta_{r, 0, \l^\ast}$ does, and $\theta_{r, 1, \l}$ holds if and only if $\theta_{r, 1, \l^\ast}$ does.
\end{itemize}

Let us now fix any $\l<\l^\ast$ in $Y$. Since $A_\l$ is a maximal antichain of $\MQB_\b$ below $q^{0'}_\l\restriction M_\l$, we may find some $r\in A_\l$ compatible with $q^{0'}_\l\oplus q^{1'}_{\l^\ast}$. We have that $\theta_{r, 0, \l}$ cannot hold since $q^{0'}_\l\oplus q^{1'}_{\l^\ast}$ extends $q^{0'}_\l$. Therefore $\theta_{r, 1, \l}$ holds, and hence also $\theta_{r, 1, \l^\ast}$ does. But that is also a contradiction since $q^{0'}_\l\oplus q^{1'}_{\l^\ast}$ extends $q^{1'}_{\l^\ast}$.

This contradiction concludes the proof of $(2)_\b$, and hence the proof of the lemma.
\end{proof}

It may be worthwhile observing that, as opposed to what is usually the case in forcing constructions incorporating models as side conditions, our use of side conditions does not interfere with the $\kappa$-chain condition. The underlying reason is of course the fact that our pure side condition forcing is trivial (Lemma \ref{pedgeistrivial}).

\section{Completing the proof of Theorem \ref{main theorem}}\label{Completing the proof of Main theorem}

In this final section we conclude the proof of Theorem \ref{main theorem}.
By Lemma \ref{countable closure}, $\MQB_{\kappa^{+}}$ does not add new $\o$-sequences of ordinals and hence it preserves $\CH$. We will start this section by proving that $\MQB_{\k^{+}}$ also preserves $2^{\al_1}=\al_2$. Of course, the reason we have incorporated edges in our construction is precisely to make this proof work.

\begin{lemma}
\label{preserving ch}
$\Vdash_{\MQB_{\kappa^{+}}} 2^{\aleph_1}=\kappa$
\end{lemma}

\begin{proof}
Suppose, towards a contradiction, that there is a condition $q \in \MQB_{\kappa^{+}}$
and a sequence $(\lusim{r}_i)_{i < \kappa^+}$
of $\MQB_{\kappa^{+}}$-names for subsets of $\omega_1$
such that
\begin{center}
$q \Vdash_{\MQB_{\kappa^{+}}} \lusim{r}_i \neq \lusim{r}_{i'}$ for all $i  < i' < \kappa^+$
\end{center}
By Lemma \ref{chain condition theorem} we may assume, for each $i$, that $\lusim{r}_i \in H(\kappa^{+})$
and $\lusim{r}_i$ is a $\MQB_{\beta_i}$-name for some $\beta_i <\kappa^+$.

Let $\theta$ be a large enough regular cardinal. For each $i < \kappa^+$
let $N^*_i \preceq H(\theta)$ be such that
\begin{enumerate}
\item  $|N^*_i|=|N^*_i\cap\kappa|$,
\item $N^*_i$ is closed under sequences of length less than $|N^*_i|$,
\item $q$, $\lusim{r}_i$, $\b_i$, $(\Phi_\a)_{\a<\k^+}$, $(\MQB_\alpha)_{\alpha<\k^{+}} \in N^*_i$, and
\item $\MQB_\a\cap N^*_i\lessdot \MQB_\a$ for every $\a\in \k^+\cap N^*_i$.
\end{enumerate}
$N^*_i$ can be found by a $\Pi^1_1$-reflection argument, using the weak compactness of $\k$ and the $\k$-chain condition of each $\MQB_\a$, as in the proof of Claim \ref{main-technical-fact}. Let $N_i = N^*_i \cap H(\kappa^{+})$ for each $i$.

Let now $P$ be the satisfaction predicate for the structure $$\langle H(\kappa^{+}), \in, \vec\Phi\rangle,$$ where $\vec\Phi\subseteq H(\k^+)$ codes $(\Phi_\a)_{\a<\k^+}$ in some canonical way, and let $M$ be an elementary submodel of $H(\theta)$ containing  $q$, $\lusim{r}_i$, $(\b_i)_{i<\k^+}$, $(\MQB_\alpha)_{\alpha\leq\k^{+}}$, $(N^\ast_i)_{i<\k^+}$ and $P$, and such that $|M|=\kappa$ and $^{{<}\kappa}M\subseteq M$.

Let $i_0\in\k^+\setminus M$. By a standard reflection argument we may find $i_1\in \k^+\cap M$ for which there exists an isomorphism
\[
\Psi: (N_{i_0}, \in,  P,  \lusim{r}_{i_0}, \b_{i_0}, q) \cong
(N_{i_1}, \in,  P, \lusim{r}_{i_1}, \b_{i_1}, q),
\]
such that $\Psi(\x)\leq\x$ for every ordinal $\xi\in N_{i_0}$.
Indeed, the existence of such an $i_1$ follows from the correctness of $M$ in $H(\theta)$ about an appropriate statement with parameters $(N_i)_{i<\k^+}$, $q$, $P$, $(\b_i)_{i<\k^+}$, $(\lusim{r}_i)_{i<\k^+}$,  $N_{i_0}\cap M$, and the isomorphism type of the structure $$(N_{i_0}, \in,  P,  \lusim{r}_{i_0}, \b_{i_0}, q),$$ all of which are in $M$.

Let $\bar q= (f_{q}, \tau_{\bar q})$, where
\[
\tau_{\bar q} = \tau_q \cup \{\langle (N_{i_0}, \b_{i_0}+1), (N_{i_1}, \b_{i_1}+1)\rangle     \}.
\]

It follows, using Lemma \ref{technical-lemma2}, that $\bar q \in \MQB_{\kappa^{+}}$.
We now show that $\bar q \Vdash_{\MQB_{\kappa^{+}}}\lusim{r}_{i_0} = \lusim{r}_{i_1}$.

Suppose not, and we will derive a contradiction. Thus we can find $\n < \omega_1$ and $q' \leq_{\k^{+}} \bar q$ such that
\[
q' \Vdash_{\MQB_{\kappa^{+}}} \text{``} \nu \in \lusim{r}_{i_0}  \iff \nu \notin \lusim{r}_{i_1} \text{''}.
\]
Let us assume, for concreteness, that $q' \Vdash_{\MQB_{\k^{+}}} \text{``} \nu \in \lusim{r}_{i_0}$ and $\nu \notin \lusim{r}_{i_1} \text{''}$ (the proof in the case that $q' \Vdash_{\MQB_{\k^{+}}} \text{``} \nu \in \lusim{r}_{i_1}$ and $\nu \notin \lusim{r}_{i_0} \text{''}$ is exactly the same).
 By correctness of $N^*_{i_0}$ we have that this model contains a maximal antichain $A$ of conditions in $\MQB_{\beta_{i_0}}$ deciding the statement ``$\nu \in \lusim{r}_{i_0}$''. By Lemma \ref{chain condition theorem} we know that $|A|<\kappa$ and hence, since $N^*_{i_0}\cap\k\in\k$, $A\subseteq N^*_{i_0}\cap H(\k^+)=N_{i_0}$ (cf.\ the proof of Lemma \ref{transfer}). Hence, we may find a common extension $q''$ of $q'$ and some $r \in N_{i_0}\cap A$ such that $r \Vdash_{\MQB_{\kappa^+}}$``$\nu \in \lusim{r}_{i_0}$''.

Also, note that, since $\Psi$ is an isomorphism between the structures $(N_{i_0}, \in, P,  \lusim{r}_{i_0}, \beta_{i_0}, q)$ and $(N_{i_1}, \in,  P, \lusim{r}_{i_1}, \beta_{i_1}, q)$, and by the choice of $P$, we have that
 $$\Psi(r) \Vdash_{\MQB_{\beta_{i_1}}}\text{``}\nu \in \Psi(\lusim{r}_{i_0})=\lusim{r}_{i_1}\text{~''}$$ But then, by clauses (5) and (6) in the definition of condition,
 we have that  $q'' \leq \Psi(r)$.
We thus obtain that $q'' \Vdash_{\MQB_{\kappa^{+}}}$``$\nu \in \lusim{r}_{i_1}$'', which is impossible as $q' \Vdash_{\MQB_{\kappa^{+}}}$``$\nu \notin \lusim{r}_{i_1}$'' and $q'' \leq q'$.

We get a contradiction and the lemma follows.\footnote{Note the resemblance of this proof with the proof of Lemma \ref{transfer}.}
\end{proof}

\begin{corollary} $\MQB_{\kappa^+}$ forces $\GCH$.\end{corollary}

Lemma \ref{satp}, which completes the proof of Theorem \ref{main theorem}, follows immediately from earlier lemmas, together with a standard density argument.

\begin{lemma}\label{satp} $\MQB_{\k^{+}}$ forces $\SATP_{\al_2}$.
\end{lemma}

\begin{proof}
Let $G$ be $\MQB_{\k^{+}}$-generic over $V$. Since $\CH$ holds in $V[G]$, there are $\aleph_2$-Aronszajn trees there. Hence, it suffices to prove that, in $V[G]$, every $\al_2$-Aronszajn tree is special.

Let $T\in V[G]$ be an $\al_2$-Aronszajn tree. Note that $\aleph_2=\kappa$ in $V[G]$ by Lemmas \ref{levy} and \ref{chain condition theorem}. We need to prove that $T$ is special in $V[G]$. Let us go down to $V$ and let us note there that, by the $\k$-chain condition of $\MQB_{\k^+}$ together with the choice of $\Phi$, we may find some nonzero $\a\in\mtcl X$ such that $\Phi(\a)$ is a $\MQB_{\a}$-name for an $\aleph_2$-Aronszajn tree such that $\Phi(\a)_G=T$. We then have that $\lusim{T}_{\a}=\Phi(\a)$.

For every $\n<\o_1$, let $A_\n=\bigcup\{f_q(\a+\n)\,\mid\, q\in G\}$.  By the definition of the forcing, we have that $A_\n$ is an antichain of  $T$. Also, given any
condition $q\in\MQB_{\k^+}$ and any node $x \in \k\times\o_1$ such that $x\notin f_q(\alpha+\nu)$ for any $\nu<\o_1$, it is easy to see that we may extend $q$ to a condition $q^\ast$ such that $x\in f_{q^\ast}(\alpha+\nu)$ for some $\n<\o_1$; indeed, it suffices for this to pick any $\n<\o_1$ such that $\a'+\n\notin \dom(f_q)$ for any $\a'\in\mtcl X$, which of course is possible since $\dom(f_q)$ is countable, extend $f_q$ to a function $f$ such that $\a+\n\in \dom(f)$ and $f(\a+\n)=\{x\}$, and close under the relevant (restrictions of) functions $\Psi_{N_0, N_1}$ for edges $\langle (N_0, \g_0), (N_1, \g_1)\rangle\in\tau_q$.\footnote{In other words, we may take $q^*=q\oplus q'$, for $q'=(f, \emptyset)$, where $\dom(f)=\{\a+\n\}$ and $f_q(\a+\n)=\{x\}$.} The above density argument shows that every node in $T$ is in some $A_\n$.
It follows that $T$ is special in $V[G]$, which concludes the proof.
\end{proof}

\subsection*{Acknowledgements} The first author acknowledges support of EPSRC Grant EP/N032160/1. The second author's research has been supported by a grant from IPM (No. 97030417). Much of the work on this paper was done while the first author was visiting the Institute for Research in Fundamental Sciences (IPM), Tehran, in December 2017. He thanks the institute for their warm hospitality. We thank Assaf Rinot for a comprehensive list of historical remarks on $\GCH$ vs.\ $\SATP_{\al_2}$;  we have only included the remarks pertaining to the early history of the problem, leading to the Laver-Shelah result from \cite{laver-shelah}. We thank Desmond Lau and Rahman Mohammadpour for their careful reading of some parts of an earlier version of the paper and for
their useful comments. We also thank Tadatoshi Miyamoto and Boban Veli\v{c}kovi\'{c} for their comments on the use of predicates in forcing constructions related to the one in the paper.

\end{document}